\documentclass[reqno]{amsart}
\usepackage{latexsym,amssymb}             
\usepackage{amsmath}                      
\usepackage{amsthm}                       
\usepackage{amscd}                        
\usepackage{mathtools, bbm,enumerate}
\usepackage[usenames,dvipsnames]{xcolor}
\usepackage[colorlinks=true,linkcolor=Blue,citecolor=Green]{hyperref}
\usepackage{stmaryrd}

\makeatletter
\@addtoreset{figure}{section}
\def\thefigure{\thesection.\@arabic\c@figure}
\def\fps@figure{h, t}
\@addtoreset{equation}{section}

\makeatother

\theoremstyle{plain} 
\newtheorem{theorem}             {Theorem}  [section]
\newtheorem{lemma}      [theorem]{Lemma}
\newtheorem{corollary}  [theorem]{Corollary}
\newtheorem{proposition}[theorem]{Proposition}

\newtheorem{conjecture} [theorem]{Conjecture}
\newtheorem{question}   [theorem]{Question}

\theoremstyle{definition}
\newtheorem{definition} [theorem]{Definition}

\theoremstyle{remark}
\newtheorem{remark}     [theorem]{Remark}

\newcommand{\mc}{\mathcal}

\usepackage{todonotes}
\usepackage{verbatim}
\usepackage{tikz}
\usepackage{graphicx}
\usepackage{subcaption}
\usepackage{hyperref}
\usetikzlibrary{patterns}

\usepackage[normalem]{ulem} 

\usepackage[style=alphabetic, 
    backend=biber, eprint=true, 
    eprint=true,        
    doi=false,          
    isbn=false,          
    url=false
]{biblatex}
\allowdisplaybreaks

\title{Resonances on geometrically finite graphs}
\author{Christian Arends}
\address{Department of Mathematics \\ Aarhus University}
\email{arends@math.au.dk}

\author{Carsten Peterson}
\address{IMJ-PRG \\ Sorbonne Universit\'e}
\email{peterson@imj-prg.fr}

\author{Tobias Weich}
\address{Institute of Mathematics and Institute for Photonic Quantum Systems (PhoQS) \\ Paderborn University}
\email{weich@math.upb.de}

\thanks{C.A. was supported by a research grant from the Aarhus University Research Foundation (Grant No. AUFF-E-2022-9-34). C.P. received funding from the European Union’s Horizon 2020 research and innovation programme under the Marie Sk\l{}odowska-Curie grant agreement No 101034255 and from the U.S. National Science Foundation Grant DMS-2503324. T.W. acknowledges funding from the Deutsche Forschungsgemeinschaft (DFG) Grant No. SFB-TRR 358/1 2023 - 491392403 (CRC “Integral Structures in Geometry and Representation Theory”).}

\begin{document}

\title{Resonances on geometrically finite graphs}

\begin{abstract}
  In analogy with the spectral theory of geometrically finite hyperbolic manifolds, we initiate the study of resonances on geometrically finite $(q+1)$-regular graphs of groups. We prove the meromorphic continuation of the resolvent of the adjacency operator on such spaces and give a geometric characterization of the resonant states. In contrast to the hyperbolic surfaces setting, geometrically finite graphs have only finitely many resonances, and these resonances may be computed explicitly, yet they exhibit many of the same qualitative phenomena as in the hyperbolic manifolds setting. Particularly interesting examples arise from algebraic curves over finite fields.
\end{abstract}

\maketitle
\tableofcontents

\section{Introduction}

\textit{Resonances} of the Laplacian, also called quantum resonances, are a well-established spectral invariant on certain classes of noncompact Riemannian manifolds. On closed manifolds, the Laplacian has discrete $L^2$-spectrum; resonances, which are also discrete, serve as a suitable replacement for the failure of this fact to remain true in the setting of infinite volume. A prototypical and yet extremely fruitful setting is that of \textit{geometrically finite} hyperbolic surfaces, where the study of resonances dates back to Patterson \cite{Pat75,Pat76a}. We refer to \cite{GZ97, Zwo99,BJP05,BD18, MN20} for a small selection of papers studying resonances on geometrically finite, or more particularly, convex cocompact hyperbolic surfaces.

Over the past decades, the study of the spectral geometry of the Laplace–Beltrami operator on hyperbolic surfaces has been paralleled by the spectral theory of the adjacency operator on $(q+1)$-regular graphs. From an algebraic perspective, there is already an evident parallel by viewing hyperbolic surfaces as the locally symmetric space of the group $\textnormal{SL}(2, \mathbb{R})$, and viewing $(p+1)$-regular graphs as the ``locally symmetric spaces'' of $\textnormal{SL}(2, \mathbb{Q}_p)$ (via the Bruhat-Tits tree). Additionally, there is by now a long tradition of developing ideas in one setting and subsequently adapting them to the other setting. A notable example of great interest is optimal spectral gaps on random regular graphs and hyperbolic surfaces. Friedman \cite{friedman} proved that a random $(q+1)$-regular graph asymptotically almost surely has adjacency operator spectral gap below $2 \sqrt{q} + \varepsilon$ for every $\varepsilon > 0$ as the number of vertices goes to infinity; this had been previously conjectured by Alon \cite{alon}. A new proof was later given by Bordenave \cite{bordenave}. Follow-up work of Bordenave-Collins \cite{bordenave_collins} pioneered the use of a powerful new tool for tackling such questions known as strong convergence. The techniques of these papers were highly influential for tackling the analogous question of optimal Laplacian spectral gap $\frac{1}{4} - \varepsilon$ for random hyperbolic surfaces, which was subsequently proven by Hide-Magee \cite{hide_magee} and Magee-Puder-van Handel \cite{magee_puder_van_handel} in the setting of random covers of hyperbolic surfaces, and Anantharaman-Monk \cite{anantharaman_monk1, anantharaman_monk2, anantharaman_monk3} and Hide-Macera-Thomas \cite{hide_macera_thomas} in the Weil-Petersson model of random hyperbolic surfaces. See also the recent survey article by Monk-Naud \cite{monk_naud}.

The aim of this article is to establish a notion of resonances for the adjacency operator on a certain class of infinite graphs which are the ``right'' analogues of geometrically finite manifolds in the setting of graphs. In fact, the objects we study must be properly understood not as graphs in the usual sense, but as \textit{graphs of groups}. We defer a detailed discussion of the precise definition of geometrically finite graphs (and the justification for why they deserve to be named as such) to Sections \ref{sec_graph_of_groups} and \ref{sec_graph_of_groups_2}. However, for now we state that, analogously to the setting of hyperbolic surfaces, geometrically finite graphs are those which can be decomposed into finitely many pieces, one of which is a finite graph called the \textit{compact core}, which we denote by $\mathcal{L}$, and the rest of which are either \textit{(orbifold) funnels} or \textit{cusps}. We say that such a graph of groups is \textit{asymptotically $(q+1)$-regular} if it is $(q+1)$-regular outside of a compact set.

We now briefly summarize some foundational results regarding resonances of the Laplacian on a geometrically finite hyperbolic surface $X$. One begins by studying the resolvent $R_X(s) \coloneqq (\Delta - s(1-s))^{-1}$, initially only defined for $\textnormal{Re}(s) > 1$ as a holomorphic family of maps $L^2(X) \to L^2(X)$. Following Mazzeo-Melrose \cite{mazzeo_melrose} and Guillopé-Zworski \cite{guillope_zworski_95}, we know that $R_X(s)$ admits a finitely meromorphic extension to all of $\mathbb{C}$ as a family of operators $L^2_{\textnormal{comp}}(X) \to L^2_\textnormal{loc}(X)$. More precisely, for any $N > 0$, the resolvent $R_X(s)$ extends to a finitely meromorphic family of operators $\rho^N L^2(X) \to \rho^{-N} L^2(X)$ for $\textnormal{Re} (s) > \frac{1}{2} - N$, and where $\rho$ is a \textit{boundary defining function} (see Theorem 6.11 of Borthwick \cite{Bor16}). On each cusp we may speak about the height of a point as its distance from the unique associated horocycle of length 1, and similarly on each funnel we have a notion of height in terms of the distance from the unique closed simple geodesic separating the funnel from the compact core. Then the function $\rho$ is $\frac{1}{\cosh(t)}$ for points at height $t$ on each funnel, and $e^{-t}$ for points at height $t$ on each cusp (and it is bounded and non-zero on the compact core). Because $R_X(s)$ is \textit{finitely} meromorphic, near each resonance $s_0$, we have an expansion
\begin{gather*}
    R_X(s) = \sum_{j = 1}^k \frac{A_j}{(s - s_0)^j} + R_{\textnormal{hol}}(s),
\end{gather*}
where each $A_j$ is a finite-rank operator, and $R_{\textnormal{hol}}(s)$ is holomorphic near $s_0$. The space of \textit{generalized resonant states} is defined as
\begin{gather} 
    \textnormal{grst}(X, s_0) := \textnormal{span}(\textnormal{im} A_1 \cup \dots \cup \textnormal{im} A_k). \label{eqn_gen_resonant_state}
\end{gather}
The multiplicity of the resonance $s_0$ is defined as
\begin{gather}
    \textnormal{mult}(s_0) := \textnormal{dim}(\textnormal{grst}(X, s_0)). \label{eqn_mult}
\end{gather}

The generalized resonant states also admit a more geometric characterization. First off, a function $\psi$ on $X$ is called \textit{purely outgoing with parameter $s_0$} if high up on the funnels it behaves like $\rho^{s_0}$, and high up on the cusps it behaves like $\rho^{s_0-1}$, possibly with additional pre-factors of size $\log(\rho)^m$. For $s_0 \neq \frac{1}{2}$, the generalized resonant states are exactly the purely outgoing functions with parameter $s_0$ which are annihilated by $(\Delta - s_0 (1-s_0))^J$ for some $J$. See e.g. Proposition 4.3 of \cite{pbhr-GHW18} and Section 8.2 of Borthwick \cite{Bor16} for more precise statements. The case of $s_0 = \frac{1}{2}$ is more subtle, but can also be understood. See Section 8.5 of Borthwick \cite{Bor16}.

\begin{table}[ht]
    \centering
    \begin{tabular}{|p{0.48\textwidth}|p{0.48\textwidth}|} 
        \hline
        \textbf{Hyperbolic surfaces} & \textbf{$(q+1)$-regular graphs} \\ \hline
        spectral parameter $s = \frac{1}{2} + i r$     & Satake parameter $\mu = q^{i \theta}$     \\ \hline
        $\Delta$-eigenvalue $s(1-s) = (\frac{1}{2} + i r)(\frac{1}{2} - i r) = \frac{1}{4} + r^2$     & $A$-eigenvalue $\sqrt{q}(\mu + \mu^{-1}) = q^{\frac{1}{2} + i \theta} + q^{\frac{1}{2} - i \theta}$      \\ \hline
        tempered spectrum $\textnormal{Re}(s) = \frac{1}{2}$ or $\textnormal{Im}(r) = 0$      & tempered spectrum $|\mu| = 1$ or $\textnormal{Im}(\theta) = 0$   \\ \hline
        ramification point $s = \frac{1}{2}$ or $r = 0$      & ramification points $\mu = \pm 1$ or $\theta \in \frac{\pi}{\ln q} \mathbb{Z}$      \\ \hline
        boundary defining function $\rho \approx e^{- d(o, \cdot)}$ with $o$ in compact core    & boundary defining function $q^{-d(o, \cdot)}$ with $o$ in compact core     \\ \hline
        outgoing condition $\rho^{s} = \rho^{\frac{1}{2} + i r}$ on funnels and $\rho^{s-1} = \rho^{-\frac{1}{2} + i r}$ on cusps     & outgoing condition $\Big(\frac{1}{\sqrt{q}{\mu}} \Big)^{d(o, \cdot)} = q^{-d(o, \cdot)(\frac{1}{2} + i \theta)}$ on funnels and $ \Big(\frac{\sqrt{q}}{\mu} \Big)^{d(o, \cdot)} = q^{-d(o, \cdot)(-\frac{1}{2} + i \theta)}$ on cusps     \\ \hline
    \end{tabular}
    \caption{Analogous spectral quantities between geometrically finite hyperbolic surfaces and regular graphs.}
    \label{table_hyperbolic_regular}
\end{table}

Two of our main results are the analogues of the results in the previous paragraph. Let $A$ denote the adjacency operator on an asymptotically $(q+1)$-regular geometrically finite graph of groups $\mathcal{G}$. Let $z(\mu) \coloneqq \frac{\mu + \mu^{-1}}{2}$ with $\mu \in \mathbb{C}^\times$. Table \ref{table_hyperbolic_regular} shows what the analogues of certain spectral quantities in the hyperbolic surfaces setting are in the regular graph setting. In the hyperbolic surfaces setting if we write $s = \frac{1}{2} + i r$, and in the regular graphs setting we write $\mu = q^{i \theta}$, then the parameter $r$ plays the analogous role to the parameter $\theta$. However, for various reasons, we prefer to use the variable $\mu$. From a representation-theoretic perspective, $\mu$ (or really $(\mu, \mu^{-1})$) corresponds to the Satake parameters of the spherical representation on which the adjacency operator acts via multiplication by $2 \sqrt{q} z(\mu)$ on the one-dimensional subspace consisting of the spherical vectors.

\begin{theorem}\label{thm:merom_resolvent}
      Let $\mathcal{G}$ be an asymptotically $(q+1)$-regular graph of groups with vertex set $\mathcal{V}$. Then the resolvent
      \begin{equation*}
        R_{\mathcal{G}}(\mu) =\left(\frac{A}{2 \sqrt{q}} - z(\mu) \right)^{-1} \colon \ell^2(\mathcal{V}) \rightarrow \ell^2(\mathcal{V})
      \end{equation*}
      is defined for $\lvert z(\mu) \rvert \gg \frac{\|A\|}{2 \sqrt{q}}$, holomorphic in that range, and admits a finitely meromorphic continuation as a family of operators
      \begin{equation*}
        R_{\mathcal{G}}(\mu) =\left(\frac{A}{2 \sqrt{q}} - z(\mu)\right)^{-1} \colon q^{- N d(o, \cdot)}\ell^2 (\mathcal{V}) \rightarrow q^{N d(o, \cdot)}\ell^2 (\mathcal{V}),
      \end{equation*}
      for $\lvert \mu \rvert > q^{\frac{1}{2} - N}$, where $o$ denotes an arbitrary base point in the compact core of $\mathcal{G}$.
    \end{theorem}

    \noindent The function $q^{-d(o, \cdot)}$ is the analogue of the boundary defining function $\rho$ in this setting.

    We may similarly define generalized resonant states $\textnormal{grst}(\mathcal{G}, \mu_0)$ for a resonance $\mu_0$. First off, we say that a function is purely outgoing of parameter $\mu_0$ if high enough on each cusp it behaves like $(\frac{\mu_0}{\sqrt{q}})^{d(o, \cdot)}$, and high enough on each funnel it behaves like $(\mu_0 \sqrt{q})^{-d(o, \cdot)}$, possibly with additional pre-factors which are polynomial in $d(o, \cdot)$ (see Definition \ref{defn_outgoing} for the precise definition); when this polynomial prefactor is constant, we say that it is of \textit{constant type}. 

    \begin{theorem} \label{thm_resonant_states}
    Let $\mu_0$ be a resonance with $\mu_0 \neq \pm 1$. Then
    \begin{gather*}
        \textnormal{grst}(\mathcal{G}, \mu_0) = \{u \textnormal{ purely outgoing for which } \exists \ J \in \mathbb{N} \textnormal{ s.t. } \Big(\frac{A}{2 \sqrt{q}} - z(\mu_0) \Big)^J u = 0 \}.
    \end{gather*}
    If instead $\mu_0 = \pm 1$, then
    \begin{gather*}
        \textnormal{grst}(\mathcal{G}, \mu_0) = \{u \textnormal{ purely outgoing of constant type s.t. } \Big(\frac{A}{2 \sqrt{q}} - z(\mu_0) \Big) u = 0 \}.
    \end{gather*}
\end{theorem}
\noindent The values $\mu_0 = \pm 1$ are the analogues of the special value $s_0 = \frac{1}{2}$ for hyperbolic surfaces. In the former case, $\pm 1$ are the ramification points of the branched double cover $\mu \mapsto z(\mu)$, and in the latter case $\frac{1}{2}$ is the ramification point of the branched double cover $s \mapsto s(1-s)$. Theorem \ref{thm_resonant_states} is proven in the sequel by combining Theorems \ref{thm_resonant_state_good} and \ref{thm_bad_points}, along with Proposition \ref{prop_always_outgoing}.

In contrast to the setting of hyperbolic manifolds, geometrically finite $(q+1)$-regular graphs of groups have only finitely many resonances.

\begin{theorem} \label{thm_multiplicity}
    The number of resonances is finite. More specifically:
    \begin{enumerate}
        \item The number of resonances, counted without multiplicity, is bounded by $2 |\mathcal{L}|$, where $|\mathcal{L}|$ is the number of vertices in the compact core.
        \item The multiplicity of a given resonance is bounded by $3(|\mathcal{L}| + 2 n_c + (q+1) n_f)^2$, where $n_c$ and $n_f$ are the number of cusps and funnels, respectively.
    \end{enumerate}
\end{theorem}
\noindent The bound on the number of resonances without multiplicity is essentially sharp. On the other hand, the bound on the multiplicity is likely far from sharp; in this work we simply sought to get an explicit bound. Generically we should expect resonances to have multiplicity one. Theorem \ref{thm_multiplicity} is proven by combining Proposition~\ref{prop_bound_multiplicity} (where a slightly sharper bound is given than in (2) above) and Corollary~\ref{cor_num_resonances}.

It is a consequence of our results that resonances can be computed explicitly via reduction to a finite-dimensional linear algebra problem. In Section \ref{sec_examples} we compute resonances explicitly for several examples. A particularly interesting class of examples is constructed via algebraic curves over finite fields; we summarize this construction in Section \ref{sec_arithmetic_lattice}. In Section \ref{sec_elliptic_curves}, we compute the resonances explicitly for two examples constructed from elliptic curves. In particular, we observe that the resonances correspond to the ``trivial'' $\ell^2$-eigenvalues ($\pm \sqrt{q}$), the remaining $\ell^2$-eigenvalues (which are known to all lie on the unit circle as a result of Drinfeld's proof of the Langlands correspondence and Petersson conjecture for $\textnormal{GL}(2)$ over function fields \cite{drinfeld2, drinfeld1}), the square roots of the zeros of the Hasse-Weil zeta function (which are known to all lie on the circle of radius $\frac{1}{\sqrt{q}}$ as a result of Weil's proof of the Riemann hypothesis in this setting \cite{weil}), and additional ``topological'' resonances at the special values $\pm 1$. We discuss this further in Section \ref{sec_arithmetic2}.

In the setting of geometrically finite hyperbolic manifolds, the theory of resonances is very closely connected with the Lax-Phillips scattering theory \cite{pbhr-LP76}. Roughly speaking, the resonances correspond to the poles of the scattering operator. The Lax-Phillips scattering theory is itself defined in terms of the (automorphic) wave equation, and in general the resonances control the long-term behavior of waves on the manifold. In follow-up work, we plan to connect the theory of resonances on geometrically finite graphs developed in this paper with the wave equation on regular trees and graphs \cite{brooks_lindenstrauss1, brooks_lindenstrauss2, anker_wave}, and in particular with the Lax-Phillips scattering theory in this setting studied by the second author in \cite{peterson}. This should in turn be very closely related to the analogue of Eisenstein series in this setting which we plan to also develop; this should also allow us to rigorously describe in representation-theoretic terms the structure of the resonances for arithmetic examples mentioned in the previous paragraph.

More generally, there is a plethora of natural follow-up questions by analogy with the manifolds setting. One such question is the connection to Ruelle resonances of the geodesic flow (see Hilgert's recent review article, where the missing notion of resonances for the averaging operator on infinite graphs was marked as an important open question \cite[Problem 6.31]{Hil24}).
Such a quantum-classical correspondence should lead to a  connection between resonances and zeros of zeta functions (in analogy to the correspondence with zeros of the Selberg zeta function obtained in \cite{BJP05} for geometrically finite hyperbolic surfaces, see also \cite{AP26, BHW25} for recent results on quantum-classical correspondence for finite graphs).
Furthermore, the distribution of resonances on large random geometrically finite graphs can be expected to have interesting properties.
In addition to its intrinsic interest, we hope that this setting may act as a testing ground for existing conjectures and new approaches to the study of resonances on geometrically finite hyperbolic manifolds. Prominent examples of such longstanding open conjectures are the fractal Weyl law \cite{GLZ04}, the Phillips-Sarnak conjecture on the paucity of embedded $L^2$-eigenvalues for generic cusped surfaces \cite{phillips_sarnak1, phillips_sarnak2, phillips_sarnak3}, and the Jakobson-Naud conjecture \cite{jakobson_naud} on the essential spectral gap for convex cocompact surfaces (see also \cite{DZ16, BD18}). In Section \ref{sec_large_random_cusps}, we show numerical experiments for large random geometrically finite graphs with cusps, and we discuss the analogue of the Phillips-Sarnak conjecture in this setting.
In Section \ref{sec_large_random_funnel} we propose an analogue of the Jakobson-Naud conjecture for convex cocompact graphs and we provide strong numerical evidence for its veracity. It would be very interesting to try to prove these analogous conjectures in this setting as they seem more approachable than their hyperbolic surfaces counterparts. 
For example, it would be very interesting to adapt and extend the techniques of Calderón-Magee-Naud \cite{calderon_magee_naud} regarding spectral gaps for random covers of convex cocompact hyperbolic surfaces.

Our results are the first systematic study of resonances on general geometrically finite regular graphs of groups. However, related questions have been considered previously in the literature.
Sanghoon Kwon and his collaborators have studied geometrically finite graphs, and in particular properties of the Ihara-Selberg zeta function \cite{kwon, kwon_lim, hong_kwon}; see also Deitmar-Kang \cite{deitmar_kang}. Efrat has studied the connection between the spectral theory of certain geometrically finite arithmetic graphs of groups and automorphic forms \cite{efrat1, efrat3, efrat2}.
Some aspects of the spectral theory of geometrically finite graphs with only funnels have been studied by Colin de Verdière-Truc \cite{colin_de_verdiere_truc}. Spectral properties of unbounded Laplacians on much more general infinite graphs have been studied in \cite{KL12}, but none of these works gives a comprehensive notion of resonances.
The scattering theory and resonances of quantum graphs, i.e. graphs with semi-infinite leads with a one-dimensional Schrödinger equation on the edges and Kirchhoff boundary conditions on the vertices, have been studied by various authors, e.g. \cite{kottos_smilansky, kuchment, exner_lipovsky, berkolaiko_kuchment, BBS12}. Finally, we mention that in \cite{SIB} resonances for certain degenerating families of convex cocompact surfaces have been studied via Ihara-Selberg zeta functions of certain finite limiting graphs related to the nature of the degeneration, thus providing a direct link between the hyperbolic surfaces setting and the graph setting.

Let us briefly outline the article. In Section \ref{sec_graph_of_groups} we discuss graphs of groups and give the precise definition of geometric finiteness in this setting.
In Section \ref{sec_graph_of_groups_2} we give further justification for why the class of objects we study is the natural analogue of geometric finiteness in this setting;
this section is primarily expository and not strictly necessary for the later sections.
In this section we also discuss a particularly interesting class of examples of geometrically finite graphs which arise from algebraic curves over finite fields.
In Section \ref{sec_model_resolvent} we study the resolvent on our model geometries, i.e.\@ on funnels and cusps. In Section \ref{sec_geometrically_finite_resolvent} we use these results to prove Theorem \ref{thm:merom_resolvent}. In Section \ref{sec_resonant_states}, we study generalized resonant states and prove Theorem \ref{thm_resonant_states}. Finally, in Section \ref{sec_examples}, we discuss the explicit computation of resonances and provide many examples, including several arithmetic examples whose resonance structure has deep connections to arithmetic, and show numerics for the distribution of resonances for large random geometrically finite graphs.

\section{Geometrically finite graphs of groups: definitions} \label{sec_graph_of_groups}

\subsection{Graphs of groups}

In order to properly discuss our results, we must utilize the language of \textit{graphs of groups}, which lies in the domain of \textit{Bass-Serre theory}. The standard reference for this theory is Serre \cite{serre_trees}. When a group $\Gamma$ acts on a tree $\mathcal{T}$, we want to construct a quotient ``graph'' $\Gamma \backslash \mathcal{T}$ which will allow us to in turn recover $\Gamma$. In case $\Gamma$ acts freely on $\mathcal{T}$, then we can make sense of $\Gamma \backslash \mathcal{T}$ as a graph in the usual sense, and we can recover $\Gamma$ from the quotient via the usual covering space theory. We shall however be particularly interested in the case that $\Gamma$ does not act freely.

Arguably the most important example of a discrete subgroup of $\textnormal{SL}(2, \mathbb{R})$ is $\Gamma = \textnormal{SL}(2, \mathbb{Z})$. However, seeing as this group has torsion elements, the quotient $\Gamma \backslash \mathbb{H}$ has orbifold points. We naturally would want to view $\Gamma$ as an example of a geometrically finite group, but this would require in general working with orbifolds rather than manifolds. On the other hand, the Selberg lemma tells us that $\Gamma$ (and more generally any finitely generated subgroup of a linear group over a field of characteristic zero) has a finite index torsionfree subgroup $\Gamma'$. Consequently $\Gamma' \backslash \mathbb{H}$ is a genuine manifold. In light of the Selberg lemma, we see that in the manifold setting, there is essentially no need to work with orbifolds since, up to finite (orbifold) covering (which from a spectral theoretic perspective is easy to handle), we can always assume we are working with a genuine manifold.

In the setting of discrete subgroups of the isometry group of the tree there is no analogue of the Selberg lemma. In fact, torsion is an essential part of the theory and, in particular, the function field theoretic analogue of the modular curve, which is the discrete subgroup $\Gamma = \textnormal{SL}(2, \mathbb{F}_q[T^{-1}])$ inside of $\textnormal{SL}(2, \mathbb{F}_q((T)))$, contains no finite index subgroup which is torsionfree; see Section \ref{sec_arithmetic_lattice} for further discussion of this example. Note that $\Gamma$ naturally acts on the Bruhat-Tits tree of $\textnormal{SL}(2, \mathbb{F}_q((T)))$, which is a $(q+1)$-regular tree $\mathcal{T}$, but we cannot (and should not) interpret the quotient $\Gamma \backslash \mathcal{T}$ as a graph in the usual sense because we are then losing crucial torsion information about $\Gamma$.

We now give the definition of a graph of groups. Let $\mathcal{G} = (\mathcal{V}, \mathcal{E})$ be a graph consisting of a set of vertices $\mathcal{V}$ and edges $\mathcal{E}$. We shall allow for the possibility of loops and multiple edges. Let $\vec{\mathcal{E}}$ denote the set of oriented edges; note that we treat each loop as having two distinct orientations. An oriented edge $\overrightarrow{e} \in \vec{\mathcal{E}}$ gives rise to a tuple $(v_1, v_2)$ of vertices consisting of an \textit{origin} $o(\overrightarrow{e}) = v_1$ and a \textit{terminus} $t(\overrightarrow{e}) = v_2$ of the oriented edge. Let $\{G_v\}_{v \in \mathcal{V}}$ and $\{G_e\}_{e \in \mathcal{E}}$ be collections of groups; these are called the \textit{vertex groups} and \textit{edge groups}, respectively. Suppose for each oriented edge $\overrightarrow{e}$ with terminus $t(\overrightarrow{e})$ we have a monomorphism $\phi_{\overrightarrow{e}}: G_e \hookrightarrow G_{t(\overrightarrow{e})}$. The ensemble of data $(\mathcal{G}, \{G_v\}, \{G_e\}, \{\phi_{\overrightarrow{e}}\})$ is called a \textit{graph of groups}; we shall often abuse notation and simply refer to $\mathcal{G}$ as a graph of groups.

Suppose $\Gamma$ is a group acting on a tree $\mathcal{T}$ without edge inversions. From this we may then naturally form a corresponding \textit{quotient graph of groups} $\Gamma \backslash \backslash \mathcal{T}$. Vertices $\mathcal{V}$ and edges $\mathcal{E}$ correspond to orbits of vertices and edges, respectively, of $\mathcal{T}$ under $\Gamma$. All vertices (resp.\@ edges) in the same orbit have conjugate (and thus isomorphic) stabilizer subgroups. Thus for each vertex $v \in \mathcal{V}$, we set $G_v$ as the stabilizer of any lift of $v \in \mathcal{T}$, and similarly for the edge subgroups (viewed as abstract groups). Finally, given an oriented edge $\overrightarrow{f}$ in $\mathcal{T}$ with terminus $w$, we naturally have the stabilizer of $f$, denoted $\Gamma_f$, as a subgroup of the stabilizer of $w$, denoted $\Gamma_w$. This provides a monomorphism $\phi_{\overrightarrow{f}}: \Gamma_f \hookrightarrow \Gamma_w$ abstractly as groups; any other choice of $\overrightarrow{f}$ in the same orbit would give rise to the same monomorphism. Thus for the directed edge $\overrightarrow{e}$ in $\mathcal{E}$ which $\overrightarrow{f}$ projects to, we define $\phi_{\overrightarrow{e}} \coloneqq \phi_{\overrightarrow{f}}$.

The fundamental theorem of Bass-Serre theory says that there is a natural one-to-one correspondence between groups acting on trees without inversion (and up to conjugacy by the automorphism group of the tree), and graphs of groups. The construction in the previous paragraph provides a map from the former set to the latter. The map going the other way involves the construction of the \textit{fundamental group of a graph of groups}. However, we will not have need for this construction, so we do not describe it further here. We remark that the quotient graph of groups $\Gamma \backslash \backslash \mathcal{T}$ contains complete information about the $\Gamma$ action on $\mathcal{T}$. Under this correspondence, given a graph of groups $\mathcal{G}$, we refer to the corresponding tree $\mathcal{T}$ as the \textit{Bass-Serre tree} of $\mathcal{G}$, and the corresponding group $\Gamma$ as the \textit{fundamental group of $\mathcal{G}$}.

Let $\mathcal{G}$ be a graph of groups, and let $\mathcal{T}$ and $\Gamma < \textnormal{Aut}(\mathcal{T})$ be the corresponding Bass-Serre tree and fundamental group. In case $\mathcal{G}$ is such that $\mathcal{V}$ and $\mathcal{E}$ are countable, and all vertex and edge groups are finite, then $\mathcal{T}$ is locally finite, and $\Gamma$ is discrete (in fact, the converse is also true). From now on we shall only be interested in this case, and these conditions shall be assumed without further mention. We define the \textit{degree} of $v \in \mathcal{V}$ as the degree of any of its lifts in $\mathcal{T}$. This can be readily read off from the data of $\mathcal{G}$:
\begin{gather}
    \textnormal{deg}(v) := \sum_{\overrightarrow{e} \in \vec{\mathcal{E}}: t(\overrightarrow{e}) = v} \frac{|G_v|}{|G_e|}. \label{eqn_degree}
\end{gather}
We shall be particularly interested in $(q+1)$-regular graphs of groups in the sense that every vertex has degree $q+1$; this is equivalent to $\mathcal{T}$ being an infinite $(q+1)$-regular tree.

The vertices of $\mathcal{T}$ carry a natural measure just assigning mass one to each vertex. The natural quotient measure $\nu$ on $\mathcal{V}$ is defined as $\nu(v) = \frac{1}{|G_v|}$. This can be explained as follows. If $G_v = 1$, then elements in the orbit of $\Gamma$ of any lift of $v$ in $\mathcal{T}$ are in bijection with elements in $\Gamma$. However, more generally elements in the orbit of any lift $\tilde{w}$ of a vertex $w \in \mathcal{V}$ are in bijection with cosets of $G_w = \Gamma_{\tilde{w}}$ inside of $\Gamma$, each of which is of ``size'' $\frac{1}{|G_w|}$ relative to $\Gamma$. Thus we should have that $\nu(w) = \frac{1}{|G_w|}$.

\subsection{Geometrically finite graphs of groups}

We are now ready to define the class of graphs of groups which we will study, namely \textit{geometrically finite graphs of groups}. These are graphs of groups which can be ``cut'' into finitely many pieces, each of which is of a specified form. To make precise the notion of ``cutting'' a graph of groups, we slightly extend the notion of graphs of groups to that of \textit{graphs of groups with boundary}. This means that we allow some vertices to have ``half-edges'' $e$ which are ``edges'' with specified terminus vertex $v$, but not origin vertex, and with $G_e$ and $\phi_{\overrightarrow{e}}: G_e \hookrightarrow G_v$ specified. We call such half-edges \textit{prongs emanating from $v$}. Heuristically these are the objects one obtains when one ``cuts'' a graph of groups along a collection of edges, and one just considers some subcollection of the resulting connected components. Clearly two graphs of groups with boundary can be ``glued together'' along a prong from each, assuming the edge groups are the same (and upon choosing an isomorphism between these edge groups).

We can now define the main constituents of geometrically finite graphs of groups, namely orbifold funnels and cusps.

\begin{definition} \label{defn_funnel}
    Let $\mathcal{G} = (\mathcal{V}, \mathcal{E})$ be a locally finite rooted tree with root $o$. An \textit{orbifold funnel} is a graph of groups with boundary whose underlying vertex and edge set are the same as $\mathcal{G}$, but with a prong $h$ emanating from $o$. The vertex and edge groups are such that $G_o = G_h$ and more generally $G_v = G_{\{v, w\}} \leq G_w$ for every vertex $v$ such that there is an edge $\{v, w\}$ with vertex $w$ closer to $o$ (and $h$) than $v$. In case $G_o = 1$, and thus all other $G_v = 1$, we simply call it a \textit{funnel}. See Figure \ref{fig_funnels}.
\end{definition}
Note that an orbifold funnel is simply the quotient of a funnel by a finite group (which necessarily preserves the root and the prong).

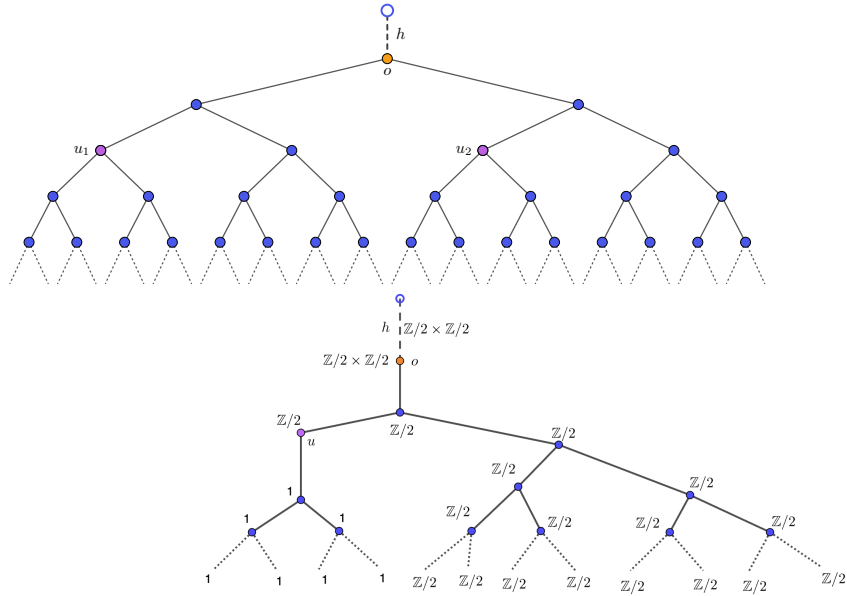
\begin{figure}[htbp]
    \centering
    \resizebox{0.8\textwidth}{!}{%
      \begin{tikzpicture}[x=0.52cm,y=1cm,line cap=round,line join=round]
        \definecolor{b}{RGB}{73,88,235}
        \definecolor{o}{RGB}{247,158,22}
        \definecolor{p}{RGB}{186,96,225}
        \tikzset{
          vb/.style={circle,fill=b,draw=black,minimum size=6.2pt,inner sep=0pt},
          vo/.style={circle,fill=o,draw=black,minimum size=6.2pt,inner sep=0pt},
          vp/.style={circle,fill=p,draw=black,minimum size=6.2pt,inner sep=0pt},
          vh/.style={circle,fill=white,draw=b,line width=1.2pt,minimum size=7.2pt,inner sep=0pt}
        }

        \coordinate (topo) at (0,0);
        \coordinate (h) at (0,1.05);

        \draw[dashed, very thick,darkgray] (topo)--(h);
        \node[vh] at (h) {};
        \node[right=2pt] at (0,0.55) {$h$};

        \foreach \i/\x in {0/-8,1/8}{
          \coordinate (d1-\i) at (\x,-1);
          \draw[thick,gray!68!black] (topo)--(d1-\i);
        }

        \foreach \i/\x in {0/-12,1/-4,2/4,3/12}{
          \coordinate (d2-\i) at (\x,-2);
        }
        \draw[thick,gray!68!black] (d1-0)--(d2-0);
        \draw[thick,gray!68!black] (d1-0)--(d2-1);
        \draw[thick,gray!68!black] (d1-1)--(d2-2);
        \draw[thick,gray!68!black] (d1-1)--(d2-3);

        \foreach \i/\x in {0/-14,1/-10,2/-6,3/-2,4/2,5/6,6/10,7/14}{
          \coordinate (d3-\i) at (\x,-3);
        }
        \foreach \p/\l/\r in {0/0/1,1/2/3,2/4/5,3/6/7}{
          \draw[thick,gray!68!black] (d2-\p)--(d3-\l);
          \draw[thick,gray!68!black] (d2-\p)--(d3-\r);
        }

        \foreach \i/\x in {0/-15,1/-13,2/-11,3/-9,4/-7,5/-5,6/-3,7/-1,8/1,9/3,10/5,11/7,12/9,13/11,14/13,15/15}{
          \coordinate (d4-\i) at (\x,-4);
        }
        \foreach \p/\l/\r in {0/0/1,1/2/3,2/4/5,3/6/7,4/8/9,5/10/11,6/12/13,7/14/15}{
          \draw[thick,gray!68!black] (d3-\p)--(d4-\l);
          \draw[thick,gray!68!black] (d3-\p)--(d4-\r);
        }

        \foreach \i in {0,...,15}{
          \draw[dotted,thick,gray!68!black] (d4-\i)--++(-0.8,-0.9);
          \draw[dotted,thick,gray!68!black] (d4-\i)--++(0.8,-0.9);
        }

        \foreach \i in {0,1}{\node[vb] at (d1-\i) {};}
        \foreach \i in {0,1,2,3}{\node[vb] at (d2-\i) {};}
        \foreach \i in {0,...,7}{\node[vb] at (d3-\i) {};}
        \foreach \i in {0,...,15}{\node[vb] at (d4-\i) {};}

        \node[vp] at (d2-0) {};
        \node[vp] at (d2-2) {};
        \node[vo] at (topo) {};
        \node[below=2pt] at (topo) {$o$};
        \node[left=3pt] at (d2-0) {$u_1$};
        \node[left=3pt] at (d2-2) {$u_2$};
      \end{tikzpicture}
    }


    \vspace{0cm} 

    \hspace*{3.3cm} \includegraphics[width=0.7\textwidth]{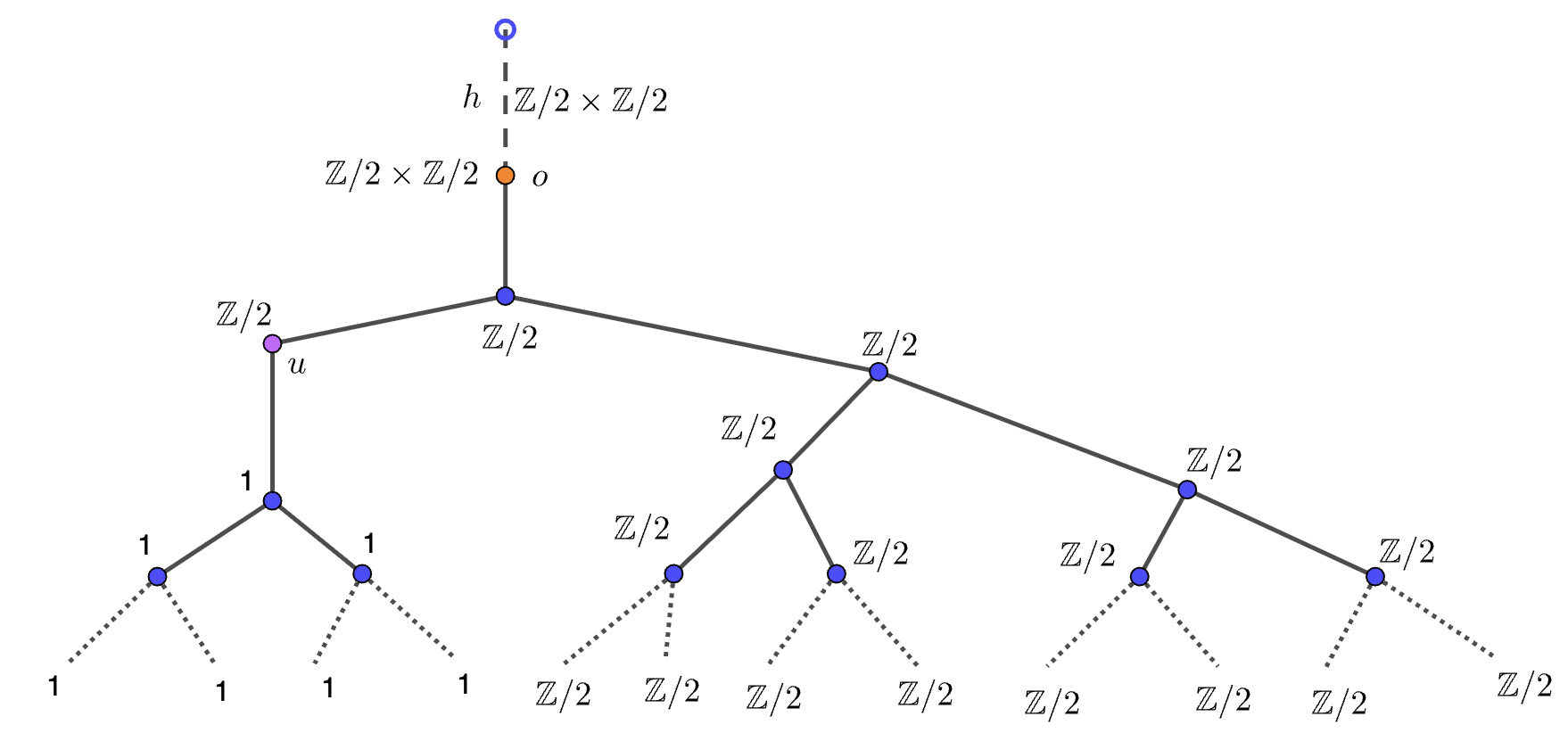}
    \caption{The top image above shows a 3-regular funnel. The bottom image above shows a 3-regular orbifold funnel arising as the quotient of the funnel on the LHS by $\mathbb{Z}/2 \times \mathbb{Z}/2$. The first factor of $\mathbb{Z}/2$ swaps the two descendant branches of $o$. The second factor of $\mathbb{Z}/2$ swaps the two descendant branches of $u_1$ and the two descendant branches of $u_2$ (which are the two lifts of $u$ under the quotient mapping). Only vertex groups are shown; each edge group is equal to the vertex group of the vertex of the edge further away from $o$. Each label $\mathbb{Z}/2$ on the bottom image really corresponds to the subgroup $1 \times \mathbb{Z}/2 < \mathbb{Z}/2 \times \mathbb{Z}/2$. Note that the bottom image is not 3-regular as a graph, but it is 3-regular as a graph of groups.}
     \label{fig_funnels}
\end{figure}

\begin{definition}
    Let $\mathcal{G}$ be a ray graph, whose vertices are labelled by $\mathbb{Z}_{\geq 0}$. We call the vertex labelled by $0$ the \textit{root}, and denote it by $o$. A \textit{cusp} is a graph of groups whose underlying vertex and edge set is the same as $\mathcal{G}$, but with a prong $h$ emanating from $o$, and such that $G_i = G_{\{i, i+1\}}$ for every $i \in \mathbb{Z}_{\geq 0}$. See Figure \ref{fig_cusp}.
\end{definition}

\begin{figure}[ht]
    \centering
    \includegraphics[width=0.7\textwidth]{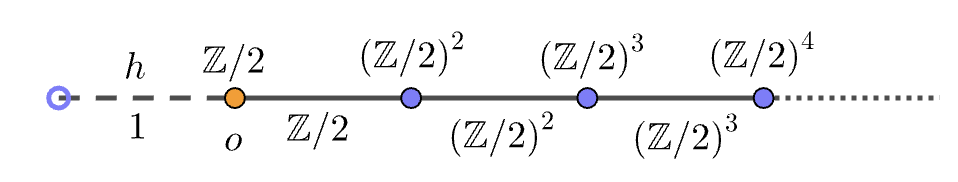}
    

    \caption{An example of a 3-regular cusp. See also Figure \ref{fig_nagao}.}
    \label{fig_cusp}
\end{figure}

If we ignore the prong, then a cusp consists of an infinite sequence of inclusions $G_{i} \hookrightarrow G_{i+1}$. This is also called a \textit{Nagao ray} in the graph of groups literature.

\begin{definition} \label{defn_geom_fin}
    A graph of groups $\mathcal{G}$ is called \textit{geometrically finite} if it can be cut into finitely many pieces, one of which is a finite graph of groups with boundary (called the \textit{compact core}), and the rest of which are either orbifold funnels or cusps. See Figure \ref{fig_geometrically_finite}.
\end{definition}

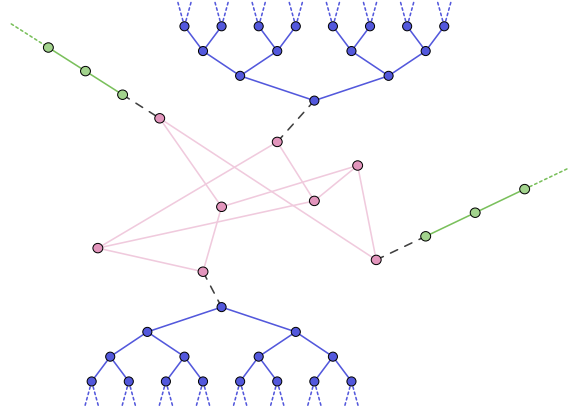
\begin{figure}[ht]
    \centering
    


    \resizebox{0.6\textwidth}{!}{%
      \begin{tikzpicture}[x=1.1cm,y=1.05cm,line cap=round,line join=round]
        \definecolor{b}{RGB}{83,88,220}
        \definecolor{p}{RGB}{226,149,188}
        \definecolor{g}{RGB}{122,192,92}
        \def\TreeLevelDist{0.42}
        \def\TreeBaseSep{0.6}
        \tikzset{tree node/.style={circle,draw=black,thin,fill=b,inner sep=0pt,minimum size=4.6pt}}
        \newcommand{\ThreeRegularTree}[2]{
          \coordinate (#1-0) at (0,0);
          \foreach \i in {0,...,7} {
            \pgfmathsetmacro{\x}{(\i - 3.5) * \TreeBaseSep}
            \coordinate (#1-3-\i) at (\x,{#2*3*\TreeLevelDist});
          }
          \foreach \i in {0,...,3} {
            \pgfmathsetmacro{\x}{(\i*2 + 0.5 - 3.5) * \TreeBaseSep}
            \coordinate (#1-2-\i) at (\x,{#2*2*\TreeLevelDist});
            \pgfmathtruncatemacro{\childA}{2*\i}
            \pgfmathtruncatemacro{\childB}{2*\i+1}
            \draw[thick,b] (#1-2-\i)--(#1-3-\childA);
            \draw[thick,b] (#1-2-\i)--(#1-3-\childB);
          }
          \foreach \i in {0,...,1} {
            \pgfmathsetmacro{\x}{(\i*4 + 1.5 - 3.5) * \TreeBaseSep}
            \coordinate (#1-1-\i) at (\x,{#2*\TreeLevelDist});
            \pgfmathtruncatemacro{\childA}{2*\i}
            \pgfmathtruncatemacro{\childB}{2*\i+1}
            \draw[thick,b] (#1-1-\i)--(#1-2-\childA);
            \draw[thick,b] (#1-1-\i)--(#1-2-\childB);
            \draw[thick,b] (#1-0)--(#1-1-\i);
          }
          \foreach \i in {0,...,7}{
            \draw[dash pattern=on 1.3pt off 1.5pt,thick,b] (#1-3-\i)--++(-0.11,{#2*0.42});
            \draw[dash pattern=on 1.3pt off 1.5pt,thick,b] (#1-3-\i)--++(0.11,{#2*0.42});
          }

          \node[tree node] at (#1-0) {};
          \foreach \i in {0,...,7} {
            \node[tree node] at (#1-3-\i) {};
          }
          \foreach \i in {0,...,3} {
            \node[tree node] at (#1-2-\i) {};
          }
          \foreach \i in {0,...,1} {
            \node[tree node] at (#1-1-\i) {};
          }
        }

        \coordinate (p1) at (-2,1.2);
        \coordinate (p2) at (-0.1,0.8);
        \coordinate (p3) at (1.2,0.4);
        \coordinate (p4) at (1.5,-1.2);
        \coordinate (p5) at (-1,- 0.3);
        \coordinate (p6) at (-1.3,-1.4);
        \coordinate (p7) at (-3,-1);
        \coordinate (p9) at (0.5,-0.2);

        \draw[thick,p!50] (p1)--(p5);
        \draw[thick,p!50] (p1)--(p4);
        \draw[thick,p!50] (p2)--(p7);
        \draw[thick,p!50] (p2)--(p9);
        \draw[thick,p!50] (p3)--(p4);
        \draw[thick,p!50] (p3)--(p5);
        \draw[thick,p!50] (p3)--(p9);
        \draw[thick,p!50] (p5)--(p6);
        \draw[thick,p!50] (p6)--(p7);
        \draw[thick,p!50] (p7)--(p9);

        \coordinate (u) at (0.5,1.5);
        \draw[dash pattern=on 4pt off 4pt,thick,darkgray] (u)--(p2);
        \begin{scope}[shift={(u)}]
          \ThreeRegularTree{tu}{1}
        \end{scope}

        \coordinate (d) at (- 1,-2);
        \draw[dash pattern=on 4pt off 4pt,thick,darkgray] (d)--(p6);
        \begin{scope}[shift={(d)}]
          \ThreeRegularTree{td}{-1}
        \end{scope}

        \draw[dash pattern=on 1.3pt off 1.5pt,thick,g] (-4.4,2.8)--(-3.8,2.4);
        \draw[thick,g] (-3.8,2.4)--(-3.2,2.0)--(-2.6,1.6);
        \draw[dash pattern=on 4pt off 4pt,thick,darkgray] (-2.6,1.6)--(p1);

        \draw[dash pattern=on 4pt off 4pt,thick,darkgray] (p4)--(2.3,-0.8);
        \draw[thick,g] (2.3,-0.8)--(3.1,-0.4)--(3.9,0);
        \draw[dash pattern=on 1.3pt off 1.5pt,thick,g] (3.9,0)--(4.7,0.4);

        \foreach \pt in {p1,p2,p3,p4,p5,p6,p7,p9}{
          \filldraw[fill=p,draw=black,thin] (\pt) circle (0.08);
        }
        \foreach \x/\y in {-3.8/2.4,-3.2/2.0,-2.6/1.6,2.3/-0.8,3.1/-0.4,3.9/0}{
          \filldraw[fill=g!70!white,draw=black,thin] (\x,\y) circle (0.08);
        }
      \end{tikzpicture}
    }

    \caption{An example of a 3-regular geometrically finite graph with compact core (in pink), two funnels (in blue), and two cusps (in green).}
    \label{fig_geometrically_finite}
\end{figure}

In Section \ref{sec_graph_of_groups_2}, we further motivate and justify calling this class of graphs of groups \textit{geometrically finite} (and our use of the terms funnel and cusp). However, we already remark briefly that geometrically finite hyperbolic surfaces can be defined as those which can be cut into finitely many pieces, one of which is a compact manifold with boundary, called the compact core, and the rest of which are either funnels or cusps, in the usual sense of hyperbolic geometry. Moreover, Paulin \cite{paulin} has shown that if one uses the more general definition of geometrically finite in variable negative curvature given by Bowditch \cite{bowditch}, then, specialized to the context of groups acting on locally finite trees, geometrically finite graphs of groups are exactly those described in the above definition. We discuss this further in Section \ref{sec_geom_fin_general}. 

\subsection{Adjacency operator on graphs of groups}

We shall be interested in the spectral theory of the \textit{adjacency operator} $A$ acting on graphs of groups. It is defined in such a way that it ``lifts'' to the adjacency operator on the corresponding Bass-Serre tree. The specific structure of the vertex and edge groups is not relevant for the spectral theory; we only need the data of their sizes. We thus define a \textit{stabilizer function} $\mathcal{S}: \mathcal{V} \sqcup \mathcal{E} \to \mathbb{N}$ with $\mathcal{S}(v) = |G_v|$ and $\mathcal{S}(e) = |G_e|$ for $v \in \mathcal{V}$ and $e \in \mathcal{E}$. More generally, we define a \textit{graph with stabilizer function} as the data of a graph and a function $\mathcal{S}: \mathcal{V} \sqcup \mathcal{E} \to \mathbb{R}_+$. We can always convert a graph of groups into a graph with stabilizer function, and it is only this data which is relevant for the spectral theory.

We introduce the following function spaces:

\begin{align*}
  C(\mathcal{V}) &\coloneqq \{f \colon \mathcal{V} \rightarrow \mathbb{C}\},\\
  C_c(\mathcal{V}) &\coloneqq \{f \in C(\mathcal{V}) \mid f(v) \neq 0 \text{ for finitely many } v \in \mathcal{V}\},\\
  \ell^2(\mathcal{V}) &\coloneqq \{f \colon \mathcal{V} \rightarrow \mathbb{C} \mid \sum_{v \in \mathcal{V}} |f(v)|^2\frac{1}{\mathcal S(v)}  < \infty\},\\
  \rho \ell^2(\mathcal{V}) &\coloneqq \left\{ f \colon \mathcal{V} \rightarrow \mathbb{C} \mid \sum_{v \in \mathcal{V}}
  \lvert f(v)\rvert^2\frac{1}{\rho^2(v)\mathcal S(v)} < \infty \right\},
\end{align*}
where $\rho: \mathcal{V} \to \mathbb{R}_+$ is an arbitrary function. Note that $\ell^2(\mathcal{V}) = \ell^2(\mathcal{V}, \nu)$ where $\nu$ is the measure on vertices discussed in the previous section. The adjacency operator acting on functions $f \in C(\mathcal{V})$ is given by 
\begin{equation}\label{eq:average}
 Af(v) = \sum_{\overrightarrow{e} \in \vec{\mathcal{E}} : t(\overrightarrow{e}) = v} \frac{\mathcal{S}(v)}{\mathcal{S}(e)}f(o(\overrightarrow{e})).
\end{equation}
The adjacency operator and the volume measure on graphs of groups are compatible in the following sense:
\begin{proposition} \label{prop_self_adjoint}
 If $\sup_{v \in \mathcal{V}} \textnormal{deg}(v) \leq R$, then $A:\ell^2(\mathcal V)\to \ell^2(\mathcal V)$ is a bounded self-adjoint operator, with $\|A\| \leq R$. 
\end{proposition}

\begin{proof}
    From the formula \eqref{eq:average}, it is clear that the kernel function for $A$ is given by
    \begin{gather*}
        K(v, w) = \frac{\mathcal{S}(v) \cdot \mathcal{S}(w)}{\mathcal{S}(\{v, w\})},
    \end{gather*}
    if $\{v, w\}$ is an edge with $v \neq w$, twice this expression if $\{v, w\}$ is an edge with $v = w$, and zero otherwise. The bound $\| A \| \leq R$ then follows immediately from the Schur test. Self-adjointness is also clear from this formula as $K(v, w) = K(w, v)$.
\end{proof}
If $\mathcal{T}$ is the corresponding Bass-Serre tree, and $\Gamma$ the corresponding subgroup of $\textnormal{Aut}(\mathcal{T})$, then we can simply interpret $A$ as the adjacency operator on $\mathcal{T}$ restricted to $\Gamma$-invariant functions (which we can naturally view as functions on $\mathcal{V}$).

\section{Geometrically finite graphs of groups: motivations} \label{sec_graph_of_groups_2}

The purpose of this section is to motivate the term ``geometrically finite'' for the graphs defined in Definition~\ref{defn_geom_fin}.\footnote{The content of this section is not necessary for understanding the analysis in subsequent sections, but it is an important motivation for studying the geometric objects treated in this article and provides rich sources of such graphs. However, a reader who is already satisfied with Definition~\ref{defn_geom_fin} and wishes to focus on the analysis can safely skip this section.} We begin by discussing a result of Paulin \cite{paulin} defining and characterizing geometrically finite graphs of groups by appropriately adapting a general definition of geometric finiteness given by Bowditch \cite{bowditch}. We then motivate the definition of funnels and cusps in the setting via explicit analogies with hyperbolic geometry. We also prove that convex cocompact graphs of groups are exactly those which are geometrically finite but without any cusps. Finally we discuss a rich source of geometrically finite graphs of groups which are the analogues of arithmetic hyperbolic surfaces. Their construction utilizes algebraic curves defined over finite fields.

\subsection{Geometric finiteness in negative curvature} \label{sec_geom_fin_general}
Originally, a hyperbolic 2- or 3-manifold was called geometrically finite if it admitted a finite-sided convex polyhedral fundamental domain in the corresponding hyperbolic space. However, this definition no longer makes sense if one works more generally with negatively curved manifolds; in fact this definition is no longer even suitable if one restricts to higher dimensional hyperbolic manifolds.

Let $X$ be a Hadamard manifold with pinched negative curvature. Let $\Gamma$ be a group acting discretely and isometrically on $X$. Bowditch \cite{bowditch} discussed several competing definitions of geometric finiteness in this setting (and which were all known to be equivalent to the classical definition of geometric finiteness for hyperbolic 2- and 3-manifolds), and he proves that these definitions are all equivalent. We shall focus on one particular definition which ultimately also makes sense more generally in the setting of groups acting on negatively curved metric spaces (such as trees). 

The space $X$ has a natural \textit{boundary at $\infty$} denoted $\partial X$. Its points may be defined as classes of asymptotically equivalent geodesic rays. The space $X \sqcup \partial X$ naturally compactifies $X$. The \textit{limit set} $\Lambda_\Gamma$ of $\Gamma$ is defined as the collection of accumulation points in $\partial X$ of $\Gamma.x$ for any point $x \in X$ (the choice of $x$ ends up being irrelevant). 

A point $\lambda \in \Lambda_\Gamma$ is said to be \textit{conical} if for some (any) $x \in X$ and some (any) geodesic ray $\beta$ tending towards $\lambda$, there is an $r \geq 0$ such that $\Gamma.x \cap N_r(\beta)$ accumulates at $\lambda$, where $N_r(\beta)$ is a neighborhood of radius $r$ around $\beta$. This definition is essentially saying that there exists $\gamma_n \in \Gamma$ such that $\gamma_n.x$ follows along the geodesic $\beta$ to within a bounded distance.

A subgroup $G < \Gamma$ is called \textit{parabolic} if there is a unique point $p \in \partial X$ which is fixed by all of $G$, and $G$ setwise preserves horocycles with respect to $p$ (see, e.g., Bowditch \cite{bowditch} for the definition of horocycles; see also Section \ref{sec_cusps} below). We say that $p$ is a \textit{bounded parabolic fixed point} if the subgroup $\Gamma_p < \Gamma$ fixing $p$ is parabolic, and $(\Lambda_\Gamma \setminus \{p\})/\Gamma_p$ is compact.

Bowditch proves that one (equivalent) definition of the geometric finiteness of $\Gamma$ is that every element of $\Lambda_\Gamma$ is either a conical limit point or a bounded parabolic fixed point.

The analogue of a Hadamard manifold in the setting of graphs is a locally finite tree $\mathcal{T}$ such that every vertex has degree at least 2. The tree also has a natural boundary $\partial \mathcal{T}$ which can also be defined in terms of classes of asymptotically equivalent geodesic rays. Suppose $\Gamma$ is a group acting discretely on $\mathcal{T}$. Then we may again naturally speak of the limit set $\Lambda_\Gamma \subseteq \partial \mathcal{T}$. The previously given definitions of conical limit point and bounded parabolic fixed point thus also apply in this setting. Paulin proved that
\begin{theorem}[Paulin \cite{paulin}, Théorème 1.1] \label{thm_paulin}
    Suppose $\Gamma < \textnormal{Aut}(\mathcal{T})$ is a discrete subgroup. Then the following are equivalent:
    \begin{enumerate}
        \item Every element of the limit set $\Lambda_\Gamma$ is either a conical limit point or a bounded parabolic fixed point.
        \item The quotient graph of groups $\Gamma \backslash \backslash \mathcal{T}$ is geometrically finite in the sense of Definition \ref{defn_geom_fin}.
    \end{enumerate}
\end{theorem}
We remark to the reader that, in case one were to consult Théorème 1.1 of Paulin, one would need to also read the comment on the top of page 6 of the same article to make the direct connection with Definition \ref{defn_geom_fin}. We also remark that, whereas geometrically finite groups acting on Hadamard manifolds are finitely generated and have finitely many conjugacy classes of finite subgroups, neither of these facts are true in the setting of geometrically finite groups acting on trees.

\subsection{Funnels and convex cocompact groups}
Next we want to give an independent motivation for the notion of funnels introduced above, based on analogies with hyperbolic geometry. Suppose $\gamma \in \textnormal{SL}(2, \mathbb{R})$ is a hyperbolic element in the sense that it fixes exactly two points $\alpha, \beta \in \partial \mathbb{H} \simeq S^1$. Then $\gamma$ acts as a translation by some length $\ell$ along the unique geodesic connecting $\alpha$ and $\beta$. Consider the hyperbolic surface $\langle \gamma \rangle \backslash \mathbb{H}$. This is called a hyperbolic cylinder. The geodesic connecting $\alpha$ and $\beta$ projects to a closed geodesic of length $\ell$. If we remove this geodesic, we obtain two disconnected ``funnels''. In general in the context of hyperbolic geometry, a funnel is any hyperbolic surface with boundary which arises in this way; see, e.g., \cite{Bor16}.

We now describe the analogous situation for trees. Suppose $\mathcal{T}$ is a locally finite tree. Let $\gamma \in \textnormal{Aut}(\mathcal{T})$ be an element which fixes exactly two points $\alpha, \beta$ of the boundary $\partial \mathcal{T}$. Then $\gamma$ translates by some length $\ell$ along the unique geodesic connecting $\alpha$ and $\beta$, and $\langle \gamma \rangle \backslash \mathcal{T}$ is isomorphic to a cycle graph of length $\ell$ (the image of the geodesic under the projection) with each vertex on the cycle having several attached rooted trees. See Figure \ref{fig_cylinders}. Let the vertices of the cycle graph be denoted $v_1, \dots, v_\ell$, with edges $\{v_1, v_2\}, \{v_2, v_3\}, \dots, \{v_\ell, v_1\}$. If we ``cut'' along all edges containing some $v_i$ but which are not part of the cycle, then in addition to the cycle graph (which now is a graph (of groups) with boundary), we obtain finitely many funnels in the sense of Definition \ref{defn_funnel}. 

More generally suppose $\Gamma$ is a discrete subgroup of $\textnormal{Aut}(\mathcal{T})$ which setwise preserves some geodesic. Then $\Gamma$ is either finite or contains a finite index subgroup isomorphic to $\mathbb{Z}$. Let $\ell$ be the smallest non-zero translation length of any element in $\Gamma$ (or set $\ell = 0$ in case $\Gamma$ is finite). Then, as a graph of groups $\Gamma \backslash \backslash \mathcal{T}$ consists of a cyclic graph of length $\ell$ together with several attached orbifold funnels.

A subgroup $\Gamma < \textnormal{SL}(2, \mathbb{R})$ is called \textit{convex cocompact} if $\Gamma$ acts cocompactly on the convex hull of the limit points of $\Gamma$, denoted $\textnormal{Conv}(\Lambda_{\Gamma})$. It turns out that $\Gamma$ is convex cocompact if and only if $\Gamma \backslash \mathbb{H}$ is an orbifold consisting of a compact core (which contains all of the orbifold points) and finitely many attached funnels (i.e. geometrically finite with no cusps). We now prove the analogous statement in the tree setting, further strengthening the analogy between funnels in the two settings.

\begin{proposition} \label{prop_convex_cocompact}
    Suppose $\Gamma < \textnormal{Aut}(\mathcal{T})$ is discrete. Then the following are equivalent:
    \begin{enumerate}
        \item $\Gamma$ is convex cocompact,
        \item All limit points of $\Gamma$ are conical,
        \item $\Gamma \backslash \backslash \mathcal{T}$ is geometrically finite with only funnels (possibly none) and no cusps.
    \end{enumerate}
\end{proposition}

\begin{proof}
    We show that (1) $\implies$ (2) $\implies$ (3) $\implies$ (1). Showing (1) $\implies$ (2) is straightforward. Suppose $\Gamma$ is convex cocompact. Let $\alpha \in \Lambda_{\Gamma}$. Let $x \in \textnormal{Conv}(\Lambda_{\Gamma})$ and let $\gamma$ be the geodesic ray from $x$ to $\alpha$. By convexity, $\gamma$ always lies in $\textnormal{Conv}(\Lambda_{\Gamma})$. By convex cocompactness, we know that we can choose a compact fundamental domain for the action of $\Gamma$ on $\textnormal{Conv}(\Lambda_{\Gamma})$. Let $R$ be the greatest distance between any two points in this fundamental domain. Then every point along $\gamma$ lies within some translate of this fundamental domain, and hence is within distance at most $R$ of $\Gamma.x$. This proves that $\alpha$ is a conical limit point.

    We now show (2) $\implies$ (3). This follows from (the proof of) Theorem \ref{thm_paulin} of Paulin \cite{paulin}. More specifically, if all limit points are conical, then $\Gamma$ is necessarily geometrically finite, but Paulin's proof also implies that in the absence of bounded parabolic fixed points, $\Gamma \backslash \backslash \mathcal{T}$ contains no cusps.

    Finally we show (3) $\implies$ (1). We begin by clarifying that, in the context of graphs of groups, geodesics correspond to the images of geodesics under the natural map $\pi: \mathcal{T} \to \Gamma \backslash \backslash \mathcal{T}$. This corresponds to geodesics on the underlying graph of $\Gamma \backslash \backslash \mathcal{T}$ except that a geodesic which has a step from $v_1 \to v_2$ is allowed to backtrack back to $v_1$ if (and only if) $G_{\{v_1, v_2\}}$ is strictly smaller than $G_{v_2}$.
    
    Let $F \subset \Gamma \backslash \backslash \mathcal{T}$ denote those vertices which lie on some closed geodesic.  This is clearly a finite set since no point on any funnel can lie in $F$. Furthermore, if we remove $F$ and all edges connecting points in $F$, then the resulting graph must not contain any loops. Let $\tilde{F}$ denote all lifts of elements in $F$ to $\mathcal{T}$. Clearly $\Gamma$ acts cocompactly on $\tilde{F}$. We claim that in fact $\tilde{F} = \textnormal{Conv}(\Lambda_\Gamma)$. First off note that clearly $\tilde{F} \subset \textnormal{Conv}(\Lambda_\Gamma)$ since the lifts of closed geodesics in $\Gamma \backslash \backslash \mathcal{T}$ cover all of $\tilde{F}$ and each such lifted geodesic is the geodesic connecting two distinct limit points. We now show the reverse inclusion. 
    
    We first claim that $\tilde{F}$ is convex. Suppose $\tilde{x}, \tilde{y} \in \tilde{F}$ are lifts of $x, y \in \Gamma \backslash \backslash \mathcal{T}$. Let $\gamma$ be the geodesic connecting them, and let $\pi(\gamma)$ be the projection of $\gamma$ to $\Gamma \backslash \backslash \mathcal{T}$. Let $\omega_x$ be some closed geodesic starting at $x$ and $\omega_y$ some closed geodesic starting at $y$. Then the sequence $\omega_x, \pi(\gamma), \omega_y, \pi(\gamma)^{-1}$ is a closed geodesic from $x$ to $x$ which contains all points on $\pi(\gamma)$. Therefore $\gamma \in \tilde{F}$, and thus $\tilde{F}$ is convex. Suppose $\alpha$ and $\beta$ are in $\Lambda_\Gamma$. Let $\tilde{x} \in \tilde{F}$. We can represent $\alpha$ as the limit of some $\gamma_n.\tilde{x}$ with $n \to \infty$ and $\gamma_n \in \Gamma$. Similarly, we can represent $\beta$ as some $\eta_n.\tilde{x}$. The geodesic from $\tilde{x}$ to $\gamma_n.\tilde{x}$ always lies in $\tilde{F}$, and as $n \to \infty$ these geodesics converge to the geodesic from $\tilde{x}$ to $\alpha$, which must also lie in $\tilde{F}$. Similarly, the geodesic from $\tilde{x}$ to $\beta$ must lie in $\tilde{F}$. The geodesic from $\alpha$ to $\beta$ traces out a path which is a subset of the union of the geodesics from $\tilde{x}$ to $\alpha$ and from $\tilde{x}$ to $\beta$; hence this geodesic lies in $\tilde{F}$. Thus $\textnormal{Conv}(\Lambda_\Gamma) \subseteq \tilde{F}$.
\end{proof}

\begin{figure}[ht]
     \centering
     \begin{subfigure}[c]{0.49\textwidth}
         \centering

         \resizebox{\textwidth}{!}{%
           \begin{tikzpicture}[
             x=1.1cm,y=1.05cm,
             line cap=round,
             line join=round
             ]
           \definecolor{b}{RGB}{83,88,220}
           \definecolor{p}{RGB}{226,149,188}
           \tikzset{
             blueN/.style={circle,draw=black,thin,fill=b,inner sep=0pt,minimum size=4.6pt},
             blueE/.style={thick,b},
             stub/.style={dash pattern=on 1.3pt off 1.5pt,thick,b},
             dsh/.style={dash pattern=on 4pt off 4pt,thick,darkgray},
             sq/.style={thick,p!50}
           }

           \newcommand{\stubs}[1]{%
             \draw[stub] (#1) ++( 0, 0.15) -- ++(0, 0.30);
             \draw[stub] (#1) ++( 0,-0.15) -- ++(0,-0.30);
             \draw[stub] (#1) ++( 0.15,0) -- ++(0.30,0);
             \draw[stub] (#1) ++(-0.15,0) -- ++(-0.30,0);
           }

           \coordinate (v1) at (0,0);
           \coordinate (v2) at (4,0);
           \coordinate (v3) at (4,-4);
           \coordinate (v4) at (0,-4);

           \draw[sq] (v1)--(v2)--(v3)--(v4)--cycle;

           \coordinate (T1) at (0,2);
           \coordinate (T2) at (4,2);
           \coordinate (L1) at (-2,0);
           \coordinate (L2) at (-2,-4);
           \coordinate (R1) at (6,0);
           \coordinate (R2) at (6,-4);
           \coordinate (B1) at (0,-6);
           \coordinate (B2) at (4,-6);

           \foreach \x/\y in {0/2,4/2}{
             \draw[blueE] (\x-1,\y)--(\x,\y)--(\x+1,\y);
             \draw[blueE] (\x,\y)--(\x,\y+1);
           }
           \foreach \x/\y in {0/-6,4/-6}{
             \draw[blueE] (\x-1,\y)--(\x,\y)--(\x+1,\y);
             \draw[blueE] (\x,\y)--(\x,\y-1);
           }
           \foreach \x/\y in {-2/0,-2/-4}{
             \draw[blueE] (\x,\y+1)--(\x,\y)--(\x,\y-1);
             \draw[blueE] (\x,\y)--(\x-1,\y);
           }
           \foreach \x/\y in {6/0,6/-4}{
             \draw[blueE] (\x,\y+1)--(\x,\y)--(\x,\y-1);
             \draw[blueE] (\x,\y)--(\x+1,\y);
           }

           \draw[dsh] (T1)--(v1);
           \draw[dsh] (T2)--(v2);
           \draw[dsh] (L1)--(v1);
           \draw[dsh] (L2)--(v4);
           \draw[dsh] (R1)--(v2);
           \draw[dsh] (R2)--(v3);
           \draw[dsh] (B1)--(v4);
           \draw[dsh] (B2)--(v3);


           \foreach \x/\y in {0/2,4/2}{
             \node[blueN] at (\x,\y) {};      
             \node[blueN] at (\x-1,\y) {};    \stubs{\x-1,\y}
             \node[blueN] at (\x+1,\y) {};    \stubs{\x+1,\y}
             \node[blueN] at (\x,\y+1) {};    \stubs{\x,\y+1}
           }

           \foreach \x/\y in {0/-6,4/-6}{
             \node[blueN] at (\x,\y) {};
             \node[blueN] at (\x-1,\y) {};    \stubs{\x-1,\y}
             \node[blueN] at (\x+1,\y) {};    \stubs{\x+1,\y}
             \node[blueN] at (\x,\y-1) {};    \stubs{\x,\y-1}
           }

           \foreach \x/\y in {-2/0,-2/-4}{
             \node[blueN] at (\x,\y) {};
             \node[blueN] at (\x,\y+1) {};    \stubs{\x,\y+1}
             \node[blueN] at (\x,\y-1) {};    \stubs{\x,\y-1}
             \node[blueN] at (\x-1,\y) {};    \stubs{\x-1,\y}
           }

           \foreach \x/\y in {6/0,6/-4}{
             \node[blueN] at (\x,\y) {};
             \node[blueN] at (\x,\y+1) {};    \stubs{\x,\y+1}
             \node[blueN] at (\x,\y-1) {};    \stubs{\x,\y-1}
             \node[blueN] at (\x+1,\y) {};    \stubs{\x+1,\y}
           }

           \filldraw[fill=p,draw=black,thin] (v1) circle (0.08);
           \filldraw[fill=p,draw=black,thin] (v2) circle (0.08);
           \filldraw[fill=p,draw=black,thin] (v3) circle (0.08);
           \filldraw[fill=p,draw=black,thin] (v4) circle (0.08);
           \node[below right] at (v1) {$v_1$};
           \node[below left]  at (v2) {$v_2$};
           \node[above left]  at (v3) {$v_3$};
           \node[above right] at (v4) {$v_4$};

         \end{tikzpicture}
       }

     \end{subfigure}
     \hfill 
     \begin{subfigure}[c]{0.49\textwidth}
         \centering

         \resizebox{\textwidth}{!}{%
           \begin{tikzpicture}[x=1.1cm,y=1.1cm,line cap=round,line join=round]
             \definecolor{b}{RGB}{86,110,230}
             \definecolor{g}{RGB}{90,220,70}

             \coordinate (r) at (0,0);
             \coordinate (a1) at (-2,-1.6);
             \coordinate (a2) at (2,-1.6);
             \coordinate (b1) at (-3.6,-3.1);
             \coordinate (b2) at (-1,-3.1);
             \coordinate (b3) at (1,-3.1);
             \coordinate (b4) at (3.6,-3.1);
             \coordinate (c1) at (-4.6,-4.2);
             \coordinate (c2) at (-2.9,-4.2);
             \coordinate (c3) at (-1.8,-4.2);
             \coordinate (c4) at (-0.5,-4.2);
             \coordinate (c5) at (0.5,-4.2);
             \coordinate (c6) at (1.8,-4.2);
             \coordinate (c7) at (2.9,-4.2);
             \coordinate (c8) at (4.6,-4.2);

             \draw[dotted,thick,g] (0,3) -- (0,4);
             \draw[thick,g] (0,2) -- (0,3);
             \draw[thick,g] (0,1) -- (0,2);
             \draw[dash pattern=on 4pt off 4pt,thick,darkgray] (0,0) -- (0,1);
             \foreach \y in {1,2,3}{\filldraw[fill=g,draw=black,thin] (0,\y) circle (0.08);}

             \draw[thick,b] (r)--(a1)--(b1);
             \draw[thick,b] (a1)--(b2);
             \draw[thick,b] (r)--(a2)--(b3);
             \draw[thick,b] (a2)--(b4);
             \foreach \u/\v in {b1/c1,b1/c2,b4/c7,b4/c8,b2/c3,b2/c4,b3/c5,b3/c6}{
               \draw[dotted,thick,b] (\u)--(\v);
             }

             \foreach \p in {r,a1,a2,b1,b2,b3,b4}{
               \filldraw[fill=b,draw=black,thin] (\p) circle (0.09);
             }

             \node[right=6pt] at (0,3) {$\left[\begin{matrix}1 & \mathbb{F}_2T^{-2}+\mathbb{F}_2T^{-1}+\mathbb{F}_2\\0&1\end{matrix}\right]$};
             \node[right=6pt] at (0,2) {$\left[\begin{matrix}1 & \mathbb{F}_2T^{-1}+\mathbb{F}_2\\0&1\end{matrix}\right]$};
             \node[right=6pt] at (0,1) {$\left[\begin{matrix}1 & \mathbb{F}_2\\0&1\end{matrix}\right]$};
             \node[below=2pt] at (0,-0.05) {$1$};

             \foreach \x/\y in {
               -2/-1.6,2/-1.6,-3.6/-3.1,-1/-3.1,1/-3.1,3.6/-3.1
             }{
               \node[above=2pt] at (\x,\y) {$1$};
             }
             \foreach \x/\y in {
               -4.6/-4.2,-2.9/-4.2,-1.8/-4.2,-0.5/-4.2,0.5/-4.2,1.8/-4.2,2.9/-4.2,4.6/-4.2
             }{
               \node[below=1pt] at (\x,\y) {$1$};
             }
           \end{tikzpicture}
         }
     \end{subfigure}
    \caption{The LHS shows the hyperbolic cylinder corresponding to an element $\gamma$ with translation length 4 acting on the 4-regular tree. The RHS shows the parabolic cylinder for the Bruhat-Tits tree of $\textnormal{SL}(2, \mathbb{F}_2((T)))$.} \label{fig_cylinders}
\end{figure}
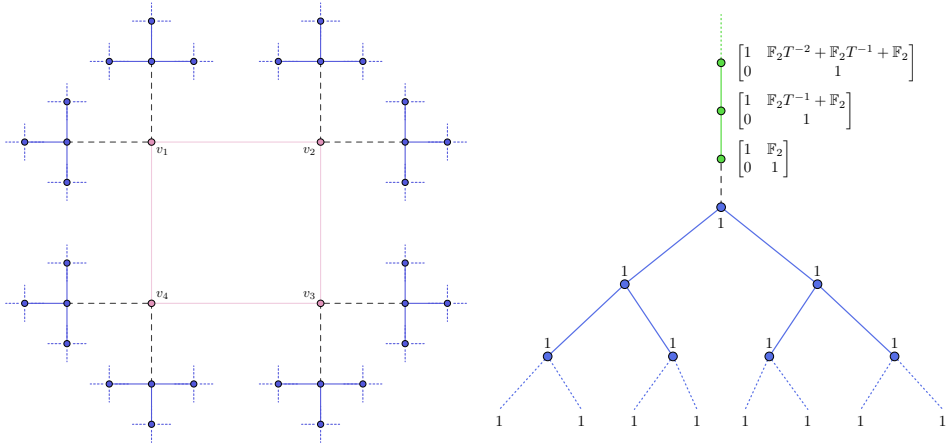

\subsection{Cusps} \label{sec_cusps}

We again start in the setting of hyperbolic surfaces. Suppose $\gamma \in \textnormal{SL}(2, \mathbb{R})$ is parabolic, i.e.\@ fixes exactly one point at infinity. Thus $\gamma$ is conjugate to $\begin{bmatrix} 1 & 1 \\ 0 & 1 \end{bmatrix}$. The hyperbolic manifold $\langle \gamma \rangle \backslash \mathbb{H}$ is called a parabolic cylinder. This surface is foliated by closed horocycles. One end of this surface is ``cusped'' and the other end is ``flared''. If we cut this surface along some closed horocycle and take the ``cusped'' end, we call the resulting surface a cusp. Notice that a cusp is non-compact but has finite volume.

It is well known that for geometrically finite $\Gamma$, we have that $\Gamma \backslash \mathbb{H}$ is compact if and only if it has no cusps and no funnels, and it has finite volume if and only if it has (necessarily finitely many) cusps and no funnels.

On trees, in contrast to the setting of hyperbolic geometry, any element $\gamma \in \textnormal{Aut}(\mathcal{T})$ which fixes only one point on the boundary necessarily also fixes some vertex (in fact infinitely many vertices).
In particular, the group generated by such an element cannot possibly act freely on $\mathcal{T}$, so the quotient cannot possibly be a $(q+1)$-regular graph in the usual sense. Furthermore, and relatedly, we know that a cusp has finite volume but is not compact. However, a finite-volume graph can only have finitely many vertices, at least if all vertices have the same mass. However, using the volume form described in Section \ref{sec_graph_of_groups}, we can indeed have finite volume non-compact quotient graphs of groups.

In order to most clearly highlight the similarities between the setting of hyperbolic surfaces and regular graphs, and in particular to motivate the definition of cusp in the latter setting, we make use of the Bruhat-Tits tree. We shall briefly recall the basic facts that we need, but we encourage the reader to consult Serre \cite{serre_trees} for further details (we also recommend \cite{Cas}). Let $F$ be a non-archimedean local field whose residue field has order $q$. Let $\mathcal{T}$ be the associated Bruhat-Tits tree, which is an infinite $(q+1)$-regular tree. 

For $\mathbb{H}$, the parabolic cylinder arose by quotienting by the additive group $\begin{bmatrix} 1 & n \\ 0 & 1 \end{bmatrix}$ with $n \in \mathbb{Z}$. The fact that each horocycle was closed arose from the fact that $\mathbb{Z}$ is a discrete cocompact subgroup of $\mathbb{R}$. We may thus seek to define a parabolic cylinder in the graph setting by quotienting $\mathcal{T}$ by the subgroup $\begin{bmatrix} 1 & x \\ 0 & 1 \end{bmatrix}$ with $x$ lying in some discrete cocompact additive subgroup of $F$.
If $F$ has characteristic zero (such as $\mathbb{Q}_p$), then no such subgroup exists. For related reasons, any lattice in $\textnormal{SL}(2, F)$ with $F$ of characteristic zero is cocompact (i.e.\@ the quotient has no cusps). On the other hand, if $F$ has positive characteristic, then $F$ looks like $\mathbb{F}_q((T))$ and it is a valued field with valuation
\[
v(a_\ell T^\ell+a_{\ell+1}T^{\ell+1}+\ldots) = \ell \text{ if } a_\ell\neq 0.
\]
Let us denote the ring of integers in $F$ by $\mathcal{O} := \mathbb{F}_q[[T]] = \{ x \in F : v(x)\geq 0 \}$.
The field $F$ contains the discrete cocompact additive subgroup
\[
Z := \mathbb{F}_q[T^{-1}] \simeq F/T\mc O.
\]
The latter is the direct limit of the finite additive subgroups $Z_j = \{a_{-j}T^{-j}+\ldots +a_{-1}T^{-1}+a_0\}$, which have order $q^{j+1}$.
Let $G = \textnormal{PGL}(2, F)$ and $K = \textnormal{PGL}(2, \mathcal{O})$. The vertices of the Bruhat–Tits tree are naturally in bijection with cosets $G/K$. Consider the subgroup
\begin{gather*}
    U := \begin{bmatrix}
        1 & Z \\
        0 & 1
    \end{bmatrix}
\end{gather*}
which is a discrete subgroup of $\textnormal{PGL}(2, F)$. We wish to understand the quotient $U \backslash G / K$; this is the analogue of the parabolic cylinder.
\begin{proposition}[Iwasawa decomposition, see {\cite[Corollary 3.3]{Cas}}]
    Every $g \in G$ can be expressed as
    \begin{gather*}
        g = \begin{bmatrix}
            1 & x \\
            0 & 1
        \end{bmatrix} \begin{bmatrix}
            T^m & 0 \\
            0 & 1
        \end{bmatrix} k,
    \end{gather*}
    with $k \in K$, $m \in \mathbb{Z}$ unique, and $x \in F$ unique mod $T^m\mathcal O$.
\end{proposition}
For a given $gK\in G/K$ expressed in the Iwasawa decomposition above, we call $m$ its \emph{level} and $\min(v(x), m)$ its branch point (note that because $x$ is only defined up to $T^m\mc O$, its valuation is only well defined for values $\leq m$). Let us describe the geometric meaning of these notions by decomposing $G/K$ into horocycles, i.e.\@ $U$-orbits (see also Figure~\ref{fig_horocycles}).
First we have the ``standard geodesic'' whose vertices correspond to the cosets $\begin{bmatrix}
    T^n & 0 \\
    0 & 1
\end{bmatrix} K$. We wish to ``fold'' onto this geodesic ``from the perspective'' of the point we approach as $n \to -\infty$. Thus, for example, of the $q+1$ neighbors of $1 K$, $q$ of them fold down onto $\begin{bmatrix}
    T & 0 \\
    0 & 1
\end{bmatrix} K$, and the remaining point $\begin{bmatrix}
    T^{-1} & 0 \\
    0 & 1
\end{bmatrix} K$ of course folds onto itself.
In terms of the Iwasawa decomposition, the unique $m$ in the representation of $g$ tells us that $gK$ folds down onto the vertex $\begin{bmatrix}
    T^m & 0 \\
    0 & 1
\end{bmatrix} K$. However, the valuation of $x$ in the Iwasawa decomposition also carries geometric information: it tells us at which point the geodesic from $g K$ to $-\infty$ first begins to intersect the standard geodesic, i.e., if the valuation is $\ell$, then it begins to intersect at $\begin{bmatrix}
  T^\ell & 0 \\
  0 & 1
\end{bmatrix} K$. Notice that, because we have chosen a uniformizer $T$, we can write $x$ uniquely as a Laurent series in $T$ and, since $x$ is only unique modulo $T^m$, we can choose for $x$ the corresponding finite Laurent series that has all coefficients of $T^n$ with $n \geq m$ equal to zero.

\begin{figure}[ht]
    \centering
    \includegraphics[width=0.8\textwidth]{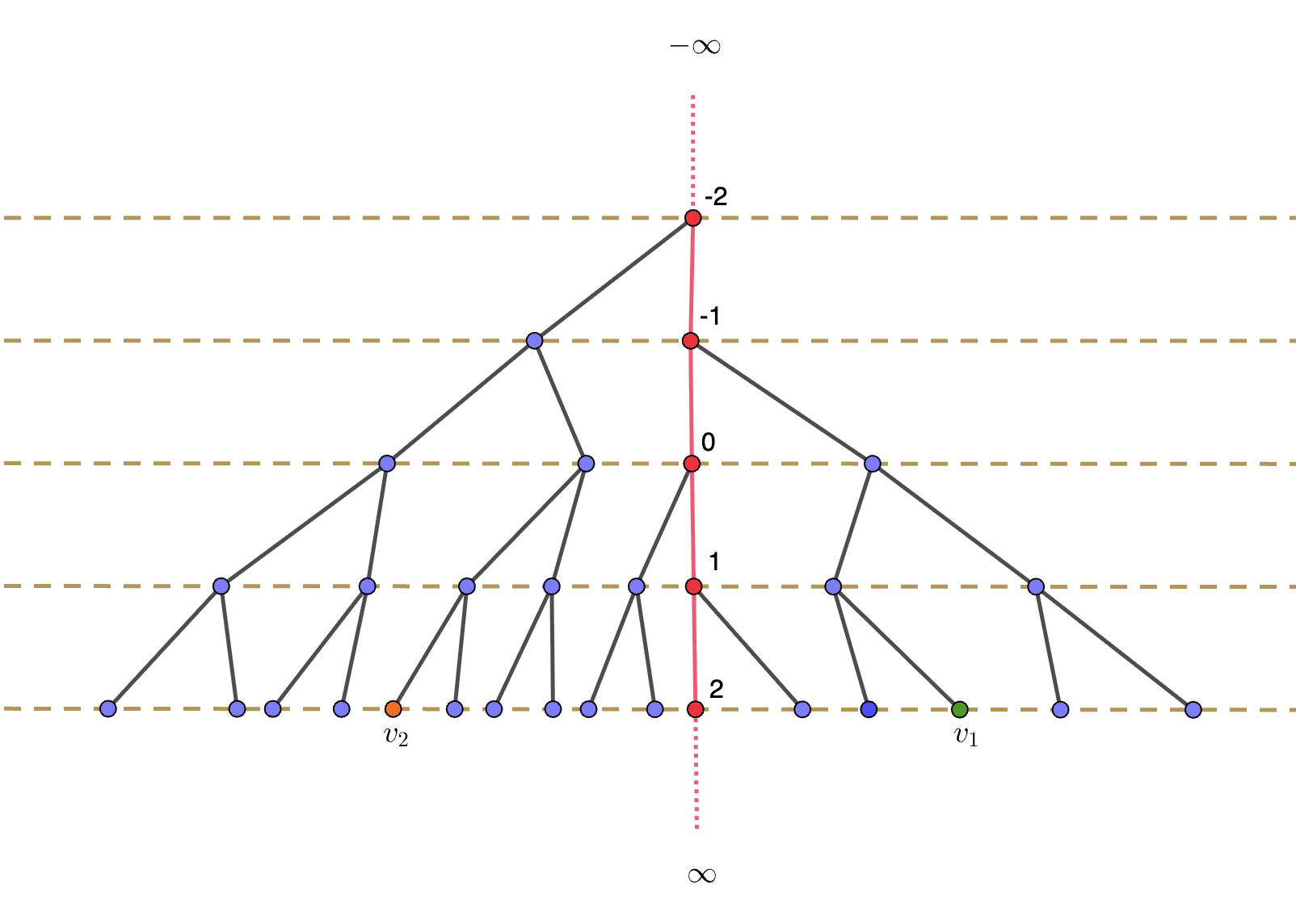}
    

    \caption{Example of the Iwasawa folding for the Bruhat-Tits tree of $\textnormal{PGL}(2, \mathbb{F}_2((T)))$. The red vertical line represents the central geodesic that we are folding onto ``from the perspective of $-\infty$''. Each brown horizontal line represents points on the same horocycle. The point $v_1$ has level 2 and branch point $-1$, so for example it might correspond to $\begin{bmatrix}
        T^2 & T^{-1} \\
        0 & 1
    \end{bmatrix} K$. The point $v_2$ has level 2 and branch point $-2$, so for example it might correspond to $\begin{bmatrix}
        T^2 & T^{-2} \\
        0 & 1
    \end{bmatrix}K$.}
    \label{fig_horocycles}
\end{figure}

\begin{lemma}~\label{lem:horo_quotient}
    \begin{enumerate}
        \item Let $m \leq 1$ be fixed. All points $gK$ with level $m$ lie in the same orbit under $U$, and $U$ acts transitively on points of level $m$ with stabilizer $\begin{bmatrix}
            1 & Z_{-m} \\
            0 & 1
        \end{bmatrix}$
        (with $Z_{-1} := 0$).
        \item Any point with level $m \geq 2$ has a unique point in its $U$-orbit with branch point at least $1$, and $U$ acts freely on these points.
    \end{enumerate}
\end{lemma}
\begin{proof}
   For points $gK$ of level $m$, the $x$ in the Iwasawa decomposition is given by
    \[
    x = a_{\ell} T^\ell + \ldots +  a_{m-1}T^{m-1} \in F/(T^m\mc O).
    \]
    Furthermore, the $U$-action on $gK$ corresponds to an additive $Z$-action on $x\in F/(T^m\mc O)$.
    For $m\leq 1$, $Z$ acts transitively on $F/(T^m\mc O)$ with stabilizer $Z_{-m}$; thus $U$ acts transitively on points of level $m$ with stabilizer group $\begin{bmatrix}
            1 & Z_{-m} \\
            0 & 1
        \end{bmatrix}$.
    For $m\geq 2$, $Z$ acts freely on $F/(T^m\mc O)$. Furthermore, by subtracting the element $a_{\ell} T^\ell + \ldots + a_0 \in Z$ from $x$, we can map $x$ to $a_1T +\ldots +a_{m-1} T^{m-1}$, which is the unique element with valuation $\geq 1$ in the $Z$-orbit of $x$. Translating this back to the $U$-action on $gK$, we see that it acts freely and its orbit contains a unique point with branch point at least $1$.
\end{proof}

\begin{corollary}
    The quotient graph of groups $U \backslash \backslash \mathcal{T}$ consists of a cusp with vertex groups
    \begin{gather*}
        \begin{bmatrix}
            1 & Z_{0} \\
            0 & 1
        \end{bmatrix} \hookrightarrow \begin{bmatrix}
            1 & Z_{-1} \\
            0 & 1
        \end{bmatrix} \hookrightarrow 
        \begin{bmatrix}
            1 & Z_{-2} \\
            0 & 1
        \end{bmatrix} \hookrightarrow \dots.
    \end{gather*}
    glued together along the prong emanating from the root to a funnel (with the edge group of the prong being the trivial group). 
\end{corollary}

\begin{proof}
    By the previous lemma, we can take as our fundamental domain in the Bruhat–Tits tree the points $\begin{bmatrix}
        T^{-i} & 0 \\
        0 & 1
    \end{bmatrix} K$ with $i\in\mathbb N$ together with all points with branch point at least $1$ (which then necessarily have level at least $1$; the unique such point with level exactly $1$ is the point $\begin{bmatrix}
        T & 0 \\
        0 & 1
    \end{bmatrix} K$). The data of the stabilizer groups follow immediately from Lemma~\ref{lem:horo_quotient}. See Figure \ref{fig_cylinders}.
\end{proof}
Analogously to the setting of the hyperbolic plane, we can cut our parabolic cylinder along the edge connecting vertices at levels $0$ and $1$ and obtain a cusp (and a funnel). It is also clear that cusps have finite volume despite being non-compact.

\subsection{Arithmetic lattices} \label{sec_arithmetic_lattice}

A lattice in a topological group is a discrete subgroup with co-finite volume. If we take our group to be $\textnormal{Aut}(\mathcal{T})$ for some locally finite tree $\mathcal{T}$, then a discrete subgroup $\Gamma < \textnormal{Aut}(\mathcal{T})$ is a lattice if and only if $\Gamma \backslash \backslash \mathcal{T}$ has finite volume. However, as discussed in Bass-Lubotzky \cite{bass_lubotzky}, lattices in $\textnormal{Aut}(\mathcal{T})$ are not necessarily geometrically finite, even if one assumes that $\mathcal{T}$ is regular. The corresponding quotient graphs of groups may have infinitely many ends, and the underlying graph may have fundamental group which is a free group on infinitely many generators. This is in stark contrast to the setting of lattices in $\textnormal{SL}(2, \mathbb{R})$ which are always geometrically finite.

On the other hand, if one restricts to the setting of lattices in $\textnormal{PGL}(2, F)$ (or more generally a semisimple algebraic group over $F$ with $F$-rank one, which in turn has an associated Bruhat-Tits tree which is either a regular or biregular tree), then any lattice $\Gamma$ in such a setting is necessarily geometrically finite, and the quotient has finitely many cusps (possibly none) and no funnels. See Theorem 6.1 of Lubotzky \cite{Lubotzky}. As remarked before, in characteristic zero, all lattices are in fact cocompact, and thus in particular the corresponding quotients have no cusps. Cusps are genuinely a positive characteristic phenomenon, as we saw in the previous section.

A very illustrative and in many ways prototypical example of a non-cocompact lattice in $\textnormal{PGL}(2, \mathbb{F}_q((T)))$ is given by $\Gamma = \textnormal{PGL}(2, \mathbb{F}_q[T^{-1}])$. Nagao's theorem \cite{nagao} tells us the structure of $\Gamma \backslash \backslash \mathcal{T}$. The underlying graph is a path graph, with nodes labelled by $\mathbb{Z}_{\geq 0}$. We have $G_0 = \textnormal{PGL}(2,\mathbb{F}_q)$ and for $j \geq 1$,
\begin{gather*}
    G_j = \begin{bmatrix}
        \mathbb{F}_q^\times & a_0 + a_1 T^{-1} + \dots + a_j T^{-j} \\
        0 & 1
    \end{bmatrix}.
\end{gather*}
The edge groups are
\begin{gather*}
    G_{\{0, 1\}} = \begin{bmatrix}
        \mathbb{F}_q^\times & \mathbb{F}_q \\
        0 & 1
    \end{bmatrix},
\end{gather*}
with the obvious inclusions into $G_0$ and $G_1$, and for $j \geq 1$, $G_{\{j, j+1\}} = G_j$ with the obvious inclusions into $G_j$ and $G_{j+1}$. See Serre \cite{serre_trees} Chapter II Section 1.6. In particular, we remark that, if we cut along the edge $\{0, 1\}$, the graph of groups with boundary consisting of vertices $j \geq 1$ is a cusp. This is very similar to $\textnormal{SL}(2, \mathbb{Z}) \backslash \mathbb{H}$ which also consists of a single cusp. See Figure \ref{fig_nagao}.

\begin{figure}[ht]
    \centering
    \includegraphics[width=1.0\textwidth]{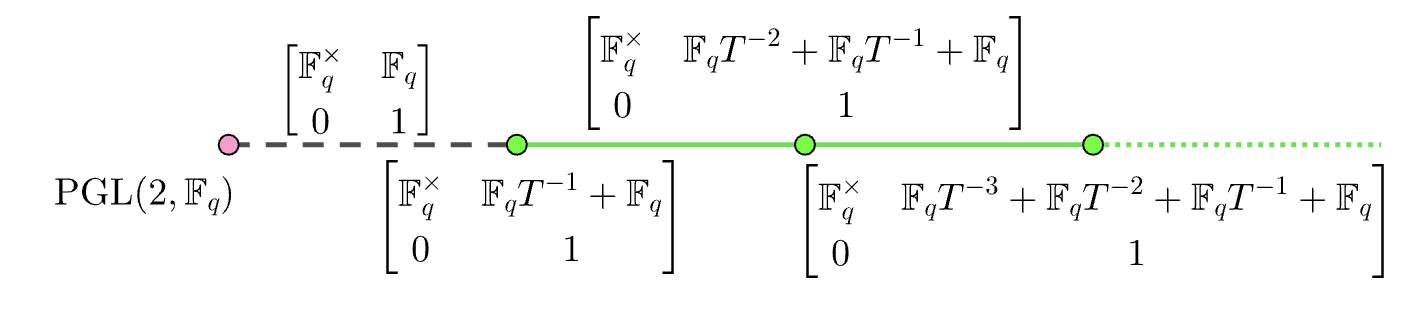}
    

    \caption{Illustration of Nagao's theorem. This is the graph of groups corresponding to the quotient of the Bruhat-Tits tree of $\textnormal{PGL}(2, \mathbb{F}_q((T)))$ by the lattice $\textnormal{PGL}(2, \mathbb{F}_q[T^{-1}])$.}
    \label{fig_nagao}
\end{figure}

We now briefly discuss the general construction of arithmetic lattices in $\textnormal{PGL}(2, \mathbb{F}_q((T)))$. This discussion is based on Serre \cite{serre_trees} Chapter II Section 2. Let $C$ be a smooth projective algebraic curve over a finite field $k_0$. Let $K$ be the function field of $C$, and let $P$ denote a closed point of $C$. Let $K_{P}$ be the completion of $K$ with respect to the valuation associated to $P$ and $\mathcal{O}_P$ the corresponding ring of integers; abstractly, $K_{P}$ is isomorphic to $k((T))$ where $k$ is the residue field at $P$ (a finite extension of $k_0$). Let $C^{\textnormal{aff}} = C \setminus \{P\}$, which is necessarily an affine curve having $P$ as its unique point at infinity. The ring of regular functions $R$ on $C^{\textnormal{aff}}$ is a Dedekind domain.

The subgroup $\Gamma = \textnormal{PGL}(2, R)$ is a lattice in $G = \textnormal{PGL}(2, K_P)$. If we take $C = \mathbb{P}^1(\mathbb{F}_q)$ and $P = \{0\}$, then $\Gamma = \textnormal{PGL}(2, \mathbb{F}_q[T^{-1}])$ which was discussed above. The vertices of the quotient of the Bruhat-Tits tree by $\Gamma$, i.e.\@ elements in
\begin{gather}
    \textnormal{PGL}(2, R) \backslash \textnormal{PGL}(2, K_P) / \textnormal{PGL}(2, \mathcal{O}_P), \label{eqn_function_field_quotient}
\end{gather}
have a geometric interpretation as rank-2 vector bundles on $C$ which trivialize on $C^\textnormal{aff}$ (modulo tensoring with a line bundle $L$ corresponding to a divisor of the form $\ell [P]$ with $\ell \in \mathbb{Z}$). We wish to explain this, at least on a heuristic level. Roughly speaking, rank-2 vector bundles on $C$ which trivialize on $C^\textnormal{aff}$ are constructed as follows. We start with the trivial bundle on $C^\textnormal{aff}$ and the trivial bundle on some arbitrarily small neighborhood $U$ of $P$. To construct a global vector bundle, we must specify the transition data on the intersection of $C^\textnormal{aff}$ and $U$, which is an arbitrarily small punctured disk centered at $P$. Such data is given by an element in $\textnormal{GL}(2, K_P)$ (note that $K_P$ is exactly those functions defined on an arbitrarily small punctured disk). However, if we pre-compose with an automorphism of the trivial bundle on $C^{\textnormal{aff}}$ (i.e.\@ an element in $\textnormal{GL}(2, R)$; $R$ consists of regular functions on $C^\textnormal{aff}$), or post-compose with an automorphism of the trivial bundle on $U$ (i.e.\@ an element of $\textnormal{GL}(2, \mathcal{O}_P)$; $\mathcal{O}_P$ is regular functions on an arbitrarily small neighborhood of $P$), then we get the same bundle. Therefore, such bundles are parametrized by elements in the double coset $\textnormal{GL}(2, R) \backslash \textnormal{GL}(2, K_P) / \textnormal{GL}(2, \mathcal{O}_P)$. The passage from $\textnormal{GL}(2)$ to $\textnormal{PGL}(2)$ amounts to quotienting by the aforementioned equivalence relation, i.e.\@ tensoring with a line bundle $L$ corresponding to a divisor of the form $\ell[P]$. At the level of transition maps, tensoring with such a line bundle corresponds to multiplying the transition map by a meromorphic section of $L$ which is trivial on $C^{\textnormal{aff}}$, i.e.\@ a diagonal matrix in $\textnormal{GL}(2, K_P)$.

The structure of the cusps carries geometric information in the language of vector bundles. Let $\textnormal{Pic}(R)$ denote the Picard group of $R$, i.e.\@ the group of isomorphism classes of line bundles on $C^{\textnormal{aff}}$, with group operation given by tensor product (or equivalently the class group of the Dedekind domain $R$). In the case at hand, it turns out that $\textnormal{Pic}(R)$ is finite. Let $c \in \textnormal{Pic}(R)$. We can then form an associated line bundle $L_c$ on $C$. Let $L_P$ be the line bundle corresponding to the divisor $[P]$. Let $L_P^n$ be the $n$th tensor power of $L_P$. Any rank-2 vector bundle of the form $L_c \oplus (L_P^n \otimes L_c^{-1})$ is trivial on $C^{\textnormal{aff}}$. Furthermore, such vector bundles all define distinct elements in \eqref{eqn_function_field_quotient} for $n \geq 1$. Additionally, the corresponding subgraph of groups for $n \geq m_c$ sufficiently large defines a cusp. As we vary over $c \in \textnormal{Pic}(R)$, we obtain distinct cusps, and all cusps arise in this way.

The spectral theory of the adjacency operator on such arithmetic quotients is closely related to the study of automorphic forms for $\textnormal{PGL}(2, K)$. In the setting of arithmetic hyperbolic surfaces (or more generally arithmetic locally symmetric spaces associated to $\textnormal{SL}(2, F)$, where $F$ is a number field, such as Hilbert modular varieties or Bianchi manifolds), there is a close relationship between resonances and the zeta function of $F$. This link is expressed in the language of Eisenstein series where it essentially corresponds to the Langlands-Shahidi method. In Section \ref{sec_arithmetic2}, we compute explicitly the resonances for several arithmetic examples as described in this section, and we observe the same phenomenon. We discuss this connection further in that section.

\section{Resolvent on model spaces} \label{sec_model_resolvent}
In order to prove Theorem~\ref{thm:merom_resolvent} we follow the approach of \cite{guillope_zworski_95} (see also \cite{Bor16}) which consists of first proving exact meromorphic resolvents on the model geometries, i.e.\@ on cusps and funnels. We use this to then construct a global parametrix in Section \ref{sec_geometrically_finite_resolvent}.

We begin by recalling what it means for a family of operators to be meromorphic.
\begin{definition}
      Let $X , Y$ be Banach spaces. A family of operators
      \begin{equation*}
        \mathbb{C} \supseteq \Omega \ni z \mapsto A(z) \in \mathcal{L}(X,Y) 
      \end{equation*}      
      is called \emph{meromorphic} if for every $z_0 \in \Omega$ there exist $J \in \mathbb{N}_0,\ \varepsilon > 0$ and operators $A_j \in \mathcal{L}(X,Y)$ such that
      \begin{equation*}
        A(z) = \sum_{
          j = - J
        }^{\infty}A_j(z - z_0)^{j} \quad \text{for } z \in B(z_0, \varepsilon)\setminus\{z_0\}, 
      \end{equation*}
      where the series converges in the operator norm. Moreover, we call $A$ \emph{finitely meromorphic} if $A_{- J}, \ldots , A_{- 1}$ are finite-rank operators and \emph{holomorphic} if $A_{- J}= \ldots = A_{- 1} = 0$.
    \end{definition}

\subsection{Resolvent on a funnel}
Funnels are not actually graphs of groups, but rather graphs of groups with boundary. Because of this, it will be most convenient for us to study the resolvent first on the $(q+1)$-regular tree where the analysis is straightforward, and then deduce from this properties for the funnel (which naturally arises as a subgraph of the tree). Handling orbifold funnels is also straightforward because these arise as subgraphs of groups of a quotient of the tree by a finite group.

\subsubsection{Resolvent on the $(q+1)$-regular tree} \label{sec_resolvent_tree}

Let $\mathcal{T} =(\mathcal{V}, \mathcal{E})$ denote a $(q + 1)$-regular tree. Then the $\ell^2$-spectrum of $A$ is given by $[- 2 \sqrt{q}, 2 \sqrt{q}]$. Thus, the resolvent
\begin{equation*}
  R_{\mathcal{T}}(\mu) \coloneqq \left(\frac{A}{2 \sqrt{q}} - z(\mu)\right)^{-1} \colon \ell^2(\mathcal{V}) \rightarrow \ell^2(\mathcal{V}), \quad z(\mu) \coloneqq \frac{\mu + \mu^{-1}}{2},
\end{equation*}
exists for $\lvert z(\mu) \rvert > 1$. In that region we denote by
\begin{equation*}
  K_\mu \colon \mathcal{V} \times \mathcal{V} \rightarrow \mathbb{C}, \quad K_{\mu }(x,y) \coloneqq  (R_{\mathcal{T}}(\mu)\delta_y)(x), \qquad \delta_y(x) =
  \begin{cases}
    1    & \colon x = y,\\
    0    & \colon x \neq y,
  \end{cases},
\end{equation*}
the kernel of $R_{\mathcal{T}}(\mu)$ such that $R_{\mathcal{T}}(\mu)f(x) = \sum_{
  y \in \mathcal{V}
}K_{\mu}(x,y)f(y)$ for all $f \in \ell^2(\mathcal{V})$. Since $A$ is invariant under the action of the automorphism group $\mathrm{Aut}(\mathcal{T})$ of the tree we have $R_{\mathcal{T}}(\mu) \circ g = g \circ R_{\mathcal{T}}(\mu)$ for all $g \in \mathrm{Aut}(\mathcal{T})$. In terms of the kernel this means that $K _{\mu}(g x, g y) = K_{\mu}(x,y)$ and since $\mathrm{Aut}(\mathcal{T})$ acts transitively on paths of the same length, $K_{\mu}(x,y) \eqqcolon k_{\mu}(d(x,y))$ only depends on the distance $d(x,y)$ between $x$ and $y$. We can now compute the kernel explicitly. Note that
\begin{align*}
    [R_{\mathcal{T}}(\mu) \delta_y](x) &= \sum_{z \in \mathcal{V}} k_\mu(d(x, z)) \delta_y(z) = k_\mu(d(x, y)), \\
    \Big[\frac{A}{2 \sqrt{q}} R_{\mathcal{T}}(\mu) \delta_y \Big](x) &= \frac{1}{2 \sqrt{q}} \begin{cases}
        k_\mu(d(x, y) - 1) + q k_\mu(d(x, y) + 1) & x \neq y \\
        (q+1) k_\mu(1) & x = y.
    \end{cases}
\end{align*}
Therefore the identity
\begin{gather*}
    \Big( z(\mu) R_{\mathcal{T}}(\mu) + I \Big) \delta_y = \Big( \frac{A}{2 \sqrt{q}} R_{\mathcal{T}}(\mu) \Big) \delta_y,
\end{gather*}
results in the following equations:
\begin{enumerate}[(i)]
\item \label{enumerate:cqzb9zuq8f} $z(\mu)k_\mu(0) + 1 = \frac{q + 1}{2 \sqrt{q}}k_\mu(1)$,
\item \label{enumerate:cqzbugtuoq} $z(\mu)k_\mu(d) = \frac{1}{2 \sqrt{q}}(k_\mu(d - 1) + q k_\mu(d + 1))$ for $d>0$.
\end{enumerate}
This is an order 2 linear recurrence relation with characteristic polynomial
\begin{gather*}
    q \lambda^2 - (\mu + \mu^{-1}) \sqrt{q} \lambda + 1 = (\sqrt{q} \lambda - \mu)(\sqrt{q} \lambda - \mu^{-1}),
\end{gather*}
so the general solution (for $d > 0$) is given by $C_1 (\frac{\mu}{\sqrt{q}})^d + C_2 (\frac{\mu^{-1}}{\sqrt{q}})^d$. However, our choice of parametrization implies that the resolvent defines a bounded operator on $\ell^2(\mathcal{V})$ for $|z(\mu)| > 1$, which in turn forces that $C_1 = 0$. From \eqref{enumerate:cqzb9zuq8f} we then infer that
\begin{gather*}
    C_2 = -\frac{2}{\mu - q^{-1} \mu^{-1}}.
\end{gather*}
Therefore,
\begin{equation}\label{eq:cq4enzzhjz}
  [R_{\mathcal{T}}(\mu)f](x) := -\frac{2}{\mu - q^{-1} \mu^{-1}} \sum_{
    y \in \mathcal{V}
  }\left(\frac{1}{\mu \sqrt{q}}\right)^{d(x,y)}f(y)
\end{equation}
provides a meromorphic extension of the resolvent to $\mu \in \mathbb{C}^\times \setminus \{\pm \frac{1}{\sqrt{q}}\}$ as a family of operators from $C_c(\mathcal{V})$ to $C(\mathcal{V})$.

\subsubsection{Convergence of the resolvent on weighted $\ell^2$-spaces}
In this section we investigate the convergence of the resolvent on the weighted $\ell^2$-spaces. For this we fix a base point $o \in \mathcal{V}$ and define, for $N \in \mathbb{N}_0$, the functions
\begin{equation*}
  q^{N d(o,\cdot)} \colon \mathcal{V} \rightarrow \mathbb{R} , \quad q^{N d(o,\cdot)}(v) \coloneqq q^{N d(o,v)}.
\end{equation*}
We will prove the following

\begin{theorem}\label{theorem:cqzdnhg2al}
  For every $N \geq 0$ the resolvent $R_{\mathcal{T}}(\mu)$ is finitely meromorphic as a family of operators
  \begin{equation*}
    R_{\mathcal{T}}(\mu) \colon q^{- N d(o, \cdot)}\ell^2(\mathcal{V}) \rightarrow q^{N d(o,\cdot)}\ell^2(\mathcal{V})
  \end{equation*}
  for $\lvert \mu \rvert > q^{- N}$.
\end{theorem}

We first define, for $d > 0$ and $f \in C(\mathcal{V})$, the spherical averaging operators
\begin{equation*}
  L_d f(x) \coloneqq(q + 1)^{-1} q^{- d + 1} \sum_{
    \substack{
      y \in \mathcal{V}\\
      d(x,y) = d
    }}f(y).
\end{equation*}
Thus, \eqref{eq:cq4enzzhjz} can be rewritten as
\begin{equation}\label{eq:cqzdm5lqhj}
  R_{\mathcal{T}}(\mu)f(x) = -\frac{2}{\mu - q^{-1} \mu^{-1}} \left(f(x) + \sum_{
      d = 1
    }^\infty \left(\frac{1}{\mu \sqrt{q}}\right)^{d}(q + 1) q^{d - 1} L_d f(x)\right)
\end{equation}
which converges for $\mu \in \mathbb{C}^\times \backslash \{\pm \frac{1}{\sqrt{q}}\}$ and $f \in C_c(\mathcal{V})$ by the previous section. We now investigate the convergence of this series for functions in weighted $\ell^2$-spaces.
\begin{lemma}\label{lemma:cqzdm5wzbx}
  For each $f \in q^{- N d(o, \cdot)} \ell^2(\mathcal{V})$ we have
  \begin{equation*}
    \lVert  \chi_m(L_d f)\rVert_{\ell^2}^2 \leq \frac{q^{1 - d}}{q + 1}q^{- 2N(m + d)} \frac{q^{4N(\min(m,d) + 1)} - 1}{q^{4N} - 1} \lVert f \rVert_{q^{- N d(o, \cdot)}\ell^2}^2,
  \end{equation*}
  where $\chi_m$ denotes the characteristic function of $S(o,m) \coloneqq \left\{ x \in \mathcal{V} \mid d(o,x) = m \right\}$.
\end{lemma}

\begin{proof}
    We can write $f = \sum_k f_k$ with $f_k := \chi_k f$. Lemma 6.1 of \cite{FTN91} tells us that we can only have $\chi_m(L_d f_k) \neq 0$ if $d + k - m$ is even and $|d-k| \leq m \leq d + k$, i.e., if $k = m + d - 2j$ for some $0 \leq j \leq \min(m,d)$, and that, moreover,
  \begin{equation*}
    \lVert \chi_m(L_d f_k ) \rVert _{\ell^2}^2 \leq \frac{q^{1 - d}}{q + 1}\lVert f_k \rVert _{\ell^2}^2.
  \end{equation*}
  Thus, we calculate
  \begin{align*}
    \lVert & \chi_m(L_d f) \rVert_{\ell^2} \\
    &\leq \sum_{
                                           j = 0
                                     }^{\min(m,d)} \lVert \chi_m(L_d f_{m + d - 2j}) \rVert_{\ell^2} \leq q^{\frac{1 - d}{2}}(q + 1)^{- \frac{1}{2}}\sum_{
                                     j = 0
                                     }^{\min(m,d)} \lVert f_{m + d - 2j} \rVert_{\ell^2} \\
                                         &\leq q^{\frac{1 - d}{2}}(q + 1)^{- \frac{1}{2}}\sum_{
                                           j = 0
                                           }^{\min(m,d)} q^{-(m + d - 2j)N}\lVert (q^{N d(o, \cdot)}f)_{m + d - 2j} \rVert_{\ell^2}\\
                                         &\leq q^{\frac{1 - d}{2}}(q + 1)^{- \frac{1}{2}}\sqrt{\sum_{
                                           j = 0
                                           }^{\min(m,d)}q^{- 2(m + d - 2j)N}} \sqrt{\sum_{
                                           j = 0
                                           }^{\min(m,d)} \lVert(q^{N d(o, \cdot)}f)_{m + d - 2 j} \rVert_{\ell^2}^2} \\
                                         & \leq q^{\frac{1 - d}{2}}(q + 1)^{- \frac{1}{2}}q^{- N(m + d)} \sqrt{\frac{q^{4N(\min(m,d) + 1)} - 1}{q^{4N} - 1}}\sqrt{\sum_{
                                           j = 0
                                           }^{\min(m,d)} \lVert(q^{N d(o, \cdot)}f)_{m + d - 2 j} \rVert_{\ell^2}^2}.
  \end{align*}
  Since $\sum_{
    j = 0
  }^{\min(m,d)} \lVert(q^{N d(o, \cdot)}f)_{m + d - 2 j} \rVert_{\ell^2}^{2} \leq \lVert f \rVert_{q^{- N d(o,\cdot)}\ell^2 }^2$ we obtain the claim.
\end{proof}

\begin{proof}[Proof of Theorem~\ref{theorem:cqzdnhg2al}]
    Suppose $f \in q^{-Nd(o, \cdot)} \ell^2$ and $\mu \neq \pm \frac{1}{\sqrt{q}}$. We wish to show that $R_{\mathcal{T}}(\mu) f \in q^{N d(o, \cdot)} \ell^2$ if $|\mu| > q^{-N}$. Using Lemma \ref{lemma:cqzdm5wzbx}, we begin by observing that
    \begin{align*}
        \|L_d f\|^2_{q^{N d(o, \cdot)} \ell^2} &= \| q^{-N d(o, \cdot)} L_d f \|_{\ell^2}^2 = \sum_{m = 0}^\infty q^{-2 N m} \| \chi_m(L_d f) \|_{\ell^2}^2 \\
        & \lesssim_{N} q^{-d} q^{-2 N d} \| f \|^2_{q^{-Nd(o, \cdot)} \ell^2} \Big(\sum_{m = 0}^\infty q^{-4 N m} q^{4 N \textnormal{min}(m, d)} \Big)\\
        & \lesssim_N q^{-d} q^{-2 N d} \| f \|^2_{q^{-Nd(o, \cdot)} \ell^2} \Big( \sum_{m = 0}^d q^{-4N m} q^{4 N m} + \sum_{m = d+1}^\infty q^{-4Nm}q^{4 N d} \Big) \\
        & \lesssim_N q^{-d} q^{-2 N d} d \| f \|^2_{q^{-Nd(o, \cdot)} \ell^2}
    \end{align*}
    Therefore, using \eqref{eq:cqzdm5lqhj}, we have
    \begin{align*}
        \|R_{\mathcal{T}}(\mu) f \|_{q^{N d(o, \cdot)} \ell^2} & \lesssim_\mu \sum_{d = 0}^\infty \Big( \frac{1}{\mu \sqrt{q}} \Big)^d q^d \|L_d f \|_{q^{N d(o, \cdot)} \ell^2} \\
        & \lesssim_{\mu, N} \| f \|_{q^{-N d(o, \cdot)} \ell^2} \Big( \sum_{d = 0}^\infty \Big( \frac{1}{\mu q^N} \Big)^d \sqrt{d} \Big).
    \end{align*}
    This converges as long as $|\mu| > q^{-N}$. 

    Finally, we prove that it is finitely meromorphic. This follows immediately from the explicit formula \eqref{eq:cq4enzzhjz} because at each pole, $\mu = \pm \frac{1}{\sqrt{q}}$, the residue operator simply involves projecting onto the space of constants (in case $\mu = \frac{1}{\sqrt{q}}$), or the space of functions which take value $(-1)^{d(o, v)} c$ at vertex $v$ for some constant $c$ and choice of $o$ (in case $\mu = -\frac{1}{\sqrt{q}}$). Thus in both cases the residue operator is rank one.
  \end{proof}

  \begin{remark}
  We can also obtain the same result by using the Schur test for the radial kernel $k_\mu$ of the resolvent $R_{\mathcal{T}}(\mu)$. This approach will be used in the discussion of the resolvent on a cusp in Section~\ref{section:resolvent_on_q_cusp}. 
\end{remark}

\begin{remark} \label{rmk_orbifold_funnel}
    Suppose $H$ is a finite group which fixes some point $o$ of $\mathcal{T}$. Let $\mathcal{P} := H \backslash \backslash \mathcal{T}$, and let $R_\mathcal{P}(\mu)$ denote the resolvent on $\mathcal{P}$. The formula for the kernel function of $R_\mathcal{P}(\mu)$ is simply given by
    \begin{gather*}
        R_\mathcal{P}(\mu)(x, y) = \sum_{h \in H} R_{\mathcal{T}}(\mu)(x, h.y).
    \end{gather*}
    From this it is clear that Theorem \ref{theorem:cqzdnhg2al} also applies to $R_\mathcal{P}(\mu)$. In particular, given any orbifold funnel, we can find some such $\mathcal{P}$ containing this funnel as a subgraph of groups.
\end{remark}

\subsection{Resolvent on a cusp}\label{section:resolvent_on_q_cusp}
We now turn our attention to the resolvent on a cusp. We shall start by considering the resolvent of a certain operator which we can in some sense interpret as the adjacency operator on a graph, but we must yet again slightly extend the definition of graph. We start by considering the infinite path graph with vertices labeled by the integers. We then define a not-necessarily-integer-valued stabilizer function where vertex $k$ has $\mathcal{S}(k) = q^k$, and the edge $(i, i+1)$ has $\mathcal{S}(i, i+1) = q^{i}$. See Figure \ref{fig_idealized}. If we define the operator $A_{\mathcal{Z}}$ by the same formula as in \eqref{eq:average}, we see that $A_{\mathcal{Z}}$ simply acts by 
\begin{gather*}
    [A f](k) = q f(k-1) + f(k+1).
\end{gather*}
We also note that, using the same definition for degree as given in \eqref{eqn_degree} (replacing $|G_v|$ and $|G_e|$ by $\mathcal{S}(v)$ and $\mathcal{S}(e)$, respectively), the object $\mathcal{Z}$ can be thought of as a $(q+1)$-regular graph, and the proof of Proposition \ref{prop_self_adjoint} implies that $A_{\mathcal{Z}}$ is in fact bounded in $\ell^2(\mathcal{V})$ (where vertex $k \in \mathbb{Z}$ has mass $q^k$). We thus know that the resolvent must define a bounded operator on $\ell^2(\mathcal{V})$ for large $\mu$. We call $\mathcal{Z}$ the \textit{idealized parabolic cylinder}.

The object $\mathcal{Z}$ is a useful tool to study the resolvent on cusps. The notion of graph of groups with boundary naturally also leads to the notion of graph with stabilizer function with boundary by just remembering the data of the size of each vertex and edge group. If we cut $\mathcal{Z}$ along the edge connecting vertices 0 and 1, then the graph of groups with stabilizer function with boundary corresponding to the positive integers, which we denote $\mathcal{C}$, is a cusp in the sense that its stabilizer function matches that of a cusp. Any other $(q+1)$-regular cusp has the same stabilizer function up to scaling the stabilizer function by some positive integer.

\begin{figure}[ht]
    \centering
    \includegraphics[width=1.0\textwidth]{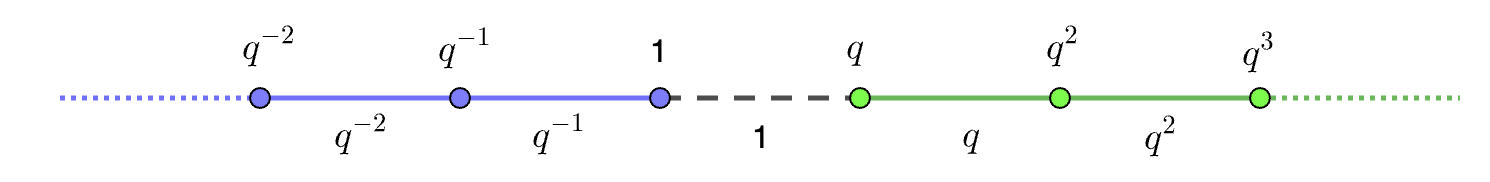}
    

    \caption{The idealized parabolic cylinder $\mathcal{Z}$. As a graph with stabilizer function, the green part is a cusp.}
    \label{fig_idealized}
\end{figure}

We wish to explicitly find the kernel function $K_{\mu}(\cdot, \cdot)$ for $R_\mathcal{Z}(\mu)$. It must satisfy the identity
\begin{gather*}
    [R_\mathcal{Z}(\mu) f](k_1) = \sum_{k_2 \in \mathbb{Z}} K_\mathcal{Z}(k_1, k_2) f(k_2) q^{-k_2}.
\end{gather*}
We have that
\begin{align*}
    [R_{\mathcal{Z}}(\mu) \delta_{k_2}](k_1) &= K_\mu(k_1, k_2) q^{-k_2} \\
    [A \delta_{k_2}](k_1) &= q \delta_{k_2+1}(k_1) + \delta_{k_2 - 1}(k_1) \\
    [A R_{\mathcal{Z}}(\mu) \delta_{k_2}](k_1) &= q^{-k_2} \Big( q K_\mu(k_1-1, k_2) + K_\mu(k_1 + 1, k_2) \Big) \\
    [R_{\mathcal{Z}}(\mu) A \delta_{k_2}](k_1) &= q^{-k_2} \Big(q K_\mu(k_1, k_2 - 1) + K_\mu(k_1, k_2 + 1) \Big).
\end{align*}
Therefore, the defining identities for the resolvent (as a left and right inverse to $\frac{A}{2 \sqrt{q}} - z(\mu)$) give the relations:
\begin{align*}
    q^{k_2} \delta_{k_2}(k_1) &= \frac{\sqrt{q}}{2} K_\mu(k_1 - 1, k_2) + \frac{1}{2 \sqrt{q}} K_\mu(k_1 + 1, k_2) - z(\mu) K_\mu(k_1, k_2) \\
    q^{k_2} \delta_{k_2}(k_1) &= \frac{\sqrt{q}}{2} K_\mu(k_1, k_2 - 1) + \frac{1}{2 \sqrt{q}} K_\mu(k_1, k_2 + 1) - z(\mu) K_\mu(k_1, k_2).
\end{align*}
If we were to fix $k_1$ and consider those $k_2 \geq k_1$, we see that we would get a linear recurrence relation with characteristic polynomial
\begin{gather*}
  \lambda^2 - 2 z(\mu) \sqrt{q} \lambda + q = (\lambda - \sqrt{q} \mu) \left(\lambda - \frac{\sqrt{q}}{\mu}\right).
\end{gather*}
The general solution would then be $C_1(k_1)(\sqrt{q} \mu)^{k_2} + C_2(k_1)\Big(\frac{\sqrt{q}}{\mu}\Big)^{k_2}$. However, we know that for $\mu$ large, $R_\mathcal{Z}(\mu)$ must define a bounded operator on $\ell^2(\mathcal{V})$, which in turn implies that we must have $C_1(k_1) = 0$. For example, if we were to take the function $f = \sum_{k = 0}^\infty \delta_k$, which is clearly in $\ell^2(\mathcal{V})$, then 
\begin{align*}
    [R_{\mathcal{Z}}(\mu) f](\ell) &= \sum_{m \in \mathbb{N}} K_\mu(\ell, m) q^{-m} \\
    &= \sum_{m = 0}^\ell K_\mu(\ell, m) q^{-m} + \sum_{m = \ell+1}^\infty \Big(C_1(\ell) (\sqrt{q} \mu)^m + C_2(\ell) \Big(\frac{\sqrt{q}}{\mu} \Big)^m \Big) q^{-m}
\end{align*}
which must be finite for $|\mu|$ large. 

We can therefore write $K_\mu(k_1, k_2) = C(k_1) \Big( \frac{\sqrt{q}}{\mu} \Big)^{k_2}$. On the other hand, if we fix $k_2$, and consider all $k_1 \leq k_2$, the corresponding sequence must satisfy the same recurrence relation, which is to say that $C(k_1) = C_1 (\sqrt{q} \mu)^{k_1} + C_2 (\frac{\sqrt{q}}{\mu})^{k_1}$. We claim that the boundedness condition on $R_{\mathcal{Z}}(\mu)$ forces $C_2 = 0$. For this we consider the sequence of elements $q^{n/2} \delta_n$ with $n \to -\infty$. These elements all have norm $1$. We have
\begin{gather*}
    [R_\mathcal{Z}(\mu) q^{n/2} \delta_n](n-1) = K_\mu(n-1, n) q^{n/2} q^{-n}.
\end{gather*}
Therefore $\|R_\mathcal{Z}(\mu) q^{n/2} \delta_n\| \geq |K_\mu(n-1, n)| q^{-n + \frac{1}{2}}$. When $|\mu|$ is large, we know that $\|R_\mathcal{Z}(\mu) q^{n/2} \delta_n\|$ must remain bounded as $n \to -\infty$, i.e.
\begin{gather*}
  \left(C_1 (\sqrt{q} \mu)^{n-1} + C_2 \left(\frac{\sqrt{q}}{\mu} \right)^{n-1}\right) \left(\frac{\sqrt{q}}{\mu} \right)^{n} q^{-n + \frac{1}{2}}
\end{gather*}
must be uniformly bounded in $n$. This is only possible if $C_2 = 0$. 

We can apply a similar analysis to the case that $k_1 \geq k_2$. We conclude that
\begin{gather*}
    K_\mu(k_1, k_2) = \begin{cases}
        C (\frac{\sqrt{q}}{\mu})^{k_1} (\sqrt{q} \mu)^{k_2} & k_1 \geq k_2 \\
        C (\frac{\sqrt{q}}{\mu})^{k_2} (\sqrt{q} \mu)^{k_1} & k_2 \geq k_1
    \end{cases} = C \Big(\frac{\sqrt{q}}{\mu} \Big)^{|k_1 - k_2|} q^{\textnormal{min}(k_1, k_2)}.
\end{gather*}
Finally, we can compute what $C$ is by setting $k_1 = k_2 = k$, which results in
\begin{gather*}
    q^k = \frac{\sqrt{q}}{2} C \frac{\sqrt{q}}{\mu} q^{k-1} + \frac{1}{2 \sqrt{q}} C \frac{\sqrt{q}}{\mu} q^k - z(\mu) C q^k,
\end{gather*}
and thus $C = -\frac{2}{\mu - \mu^{-1}}$. In summary we have proven
\begin{proposition}
    The formula
    \begin{gather}
        [R_{\mathcal{Z}}(\mu) f](k) = -\frac{2}{\mu - \mu^{-1}} \Bigg(\sum_{\ell < k} \Big( \frac{\sqrt{q}}{\mu} \Big)^{k-\ell} f(\ell) + \sum_{\ell \geq k} ( \sqrt{q} \mu )^{k-\ell} f(\ell) \Bigg) \label{eqn_resolvent_cusp}
    \end{gather}
    provides a finitely meromorphic continuation of the resolvent of $\frac{A}{2 \sqrt{q}}$ to $\mathbb{C}^\times \setminus \{\pm 1\}$ as a family of operators from $C_c(\mathcal{Z})$ to $C(\mathcal{Z})$. 
  \end{proposition}

  \subsubsection{Convergence of the cusp resolvent on weighted \texorpdfstring{$\ell^2$}{_2}-spaces}
  We shall ultimately be interested in the resolvent on the cusp $\mathcal{C}$, rather than on $\mathcal{Z}$. In Section \ref{sec_geometrically_finite_resolvent}, we shall construct an approximation to the resolvent on geometrically finite graphs by ``gluing together'' our model resolvents on funnels and cusps. The model resolvent on the cusps shall take the form of $R_{\mathcal{C}}(\mu)$, defined below and which we now focus on.

  Let $\mathcal{C}'$ denote the subgraph with stabilizer function corresponding to $n \geq 2$. Let $\mathbbm{1}_{\mathcal{C}'}$ denote the indicator function of $\mathcal{C}'$.

  \begin{theorem} \label{thm_cusp}
    We define
    \begin{gather*}
        R_{\mathcal{C}}(\mu) := \mathbbm{1}_{\mathcal{C}'} R_{\mathcal{Z}}(\mu) \mathbbm{1}_{\mathcal{C}'}.
    \end{gather*}
    For every $N \geq 0$, $R_{\mathcal{C}}(\mu)$ defines a finitely meromorphic family of operators
    \begin{gather*}
        R_{\mathcal{C}}(\mu): q^{-N d(o, \cdot)} \ell^2(\mathcal{C}) \to q^{N d(o, \cdot)} \ell^2(\mathcal{C}),
    \end{gather*}
    on $|\mu| > q^{\frac{1}{2}-N}$ with poles at $\pm 1$.
\end{theorem}

\begin{proof}
    We wish to prove that
    \begin{gather*}
        \tilde{R}_{\mathcal{C}}(\mu) := q^{-N d(o, \cdot)} R_{\mathcal{C}}(\mu) q^{-N d(o, \cdot)}
    \end{gather*}
    defines an operator on $\ell^2(\mathcal{C})$ for $|\mu| > q^{\frac{1}{2}-N}$ for $\mu \neq \pm 1$. The kernel function for $\tilde{R}_{\mathcal{C}}(\mu)$ is given by
    \begin{gather*}
        \tilde{K}_{\mathcal{C}}(\mu)(k_1, k_2) = - \frac{2}{\mu - \mu^{-1}} q^{-N(k_1 + k_2)} \Big(\frac{\sqrt{q}}{\mu} \Big)^{|k_1 - k_2|} q^{\textnormal{min}(k_1, k_2)}
    \end{gather*}
    if $k_1 \neq 0$ and $k_2 \neq 0$, and is zero when $k_1 = 0$ or $k_2 = 0$. 
    
    We can bound the operator norm by the Schur test:
    \begin{gather*}
        \| \tilde{K}_{\mathcal{C}}(\mu) \|^2 \leq \Big( \sup_{k_1 \in \mathbb{N}} \sum_{k_2 \in \mathbb{N}} |\tilde{K}_{\mathcal{C}}(\mu)(k_1, k_2)| q^{-k_2} \Big) \Big(\sup_{k_2 \in \mathbb{N}} \sum_{k_1 \in \mathbb{N}} |\tilde{K}_{\mathcal{C}}(\mu)(k_1, k_2)| q^{-k_1} \Big).
    \end{gather*}
    We have,
    \begin{align*}
        \sum_{k_1 \in \mathbb{N}} \sum_{k_2 \in \mathbb{N}} |\tilde{K}_{\mathcal{C}}(\mu)(k_1, k_2)| q^{-k_2} &= \frac{2}{|\mu - \mu^{-1}|} \sup_{k_1 \in \mathbb{N}} \Big( \sum_{k_2 = 1}^{k_1 - 1} \Big( \frac{\sqrt{q}}{|\mu|} \Big)^{k_1 - k_2} q^{-N(k_1 + k_2)} \\
        & + \sum_{k_2 = k_1}^{\infty} \Big(\frac{1}{|\mu|\sqrt{q}} \Big)^{k_2 - k_1} q^{-N(k_1 + k_2)} \Big)
    \end{align*}
    The first sum evaluates to
      \begin{equation*}
         \left(\frac{1}{\lvert \mu \rvert q^{-\frac{1}{2} + N}}\right)^{k_1}\left(\frac{\left(\frac{\lvert \mu \rvert}{q^{\frac{1}{2} + N}}\right)^{k_1} - 1}{\frac{\lvert \mu \rvert}{q^{\frac{1}{2} + N}} - 1} - 1\right) \leq  \frac{q^{- 2N k_1} -\left(\frac{1}{\lvert \mu \rvert q^{-\frac{1}{2} + N}}\right)^{k_1}}{\frac{\lvert \mu \rvert}{q^{\frac{1}{2} + N}} - 1}
      \end{equation*}
      if $\lvert \mu \rvert \neq q^{\frac{1}{2} + N}$ and to $q^{- 2N k_1} k_1$ if $\lvert \mu \rvert = q^{\frac{1}{2} + N}$, which are both uniformly bounded in $k_1$ for $\lvert \mu \rvert>q^{\frac{1}{2}-N}$. Moreover, for $\lvert \mu  \rvert > q^{\frac{1}{2} + N}$, it is bounded by $\frac{1}{\frac{\lvert \mu \rvert}{q^{\frac{1}{2} + N}} - 1}$. The second sum evaluates to
      \begin{equation*}
        q^{- 2N k_1} \sum_{
          k_2 = k_1
        }^\infty\left(\frac{1}{\lvert \mu \rvert q^{\frac{1}{2}+N}}\right)^{k_2 - k_1 } = \frac{q^{- 2 N k_1}}{1 - \frac{1}{\lvert \mu \rvert q^{\frac{1}{2} + N}}} \leq \frac{1}{1 - \frac{1}{\lvert \mu \rvert q^{\frac{1}{2} + N}}}
      \end{equation*}
      if $\lvert \mu \rvert > \frac{1}{ q^{\frac{1}{2} + N}}$. Thus, the supremum is finite for $\lvert \mu \rvert>q^{\frac{1}{2} - N}$ and the same argument applies to the second supremum.

\end{proof}

\section{Resolvent on geometrically finite graphs of groups} \label{sec_geometrically_finite_resolvent}

The point of this section is to prove Theorem \ref{thm:merom_resolvent}. We begin by recalling a version of the analytic Fredholm theorem which we shall use.

\begin{theorem}[Analytic Fredholm theorem; see Theorem 6.7 of \cite{Bor16}]
      Suppose $E(s)$ is a finitely meromorphic family of compact operators on a Hilbert space $\mathcal{H}$. If $I + E(s)$ is invertible for at least one point in the domain $\Omega \subset \mathbb{C}$, then $(I + E(s))^{-1}$ exists as a finitely meromorphic family of operators on all of $\Omega$.
    \end{theorem}

\begin{proof}[Proof of Theorem \ref{thm:merom_resolvent}]
    We first discuss the proof strategy. We begin by constructing a finitely meromorphic family of operators $P(\mu): q^{-Nd(o, \cdot)} \ell^2(\mathcal{V}) \to q^{N d(o, \cdot)} \ell^2(\mathcal{V})$ for $|\mu| > q^{\frac{1}{2}-N}$, and which acts as a parametrix to $\frac{A}{2 \sqrt{q}} - z(\mu)$ in the sense that
    \begin{gather*}
      \Big(\frac{A}{2 \sqrt{q}} - z(\mu)\Big) P(\mu) = I + K(\mu),
    \end{gather*}
    where $K(\mu) : q^{-N d(o, \cdot)} \ell^2 \to q^{-N d(o, \cdot)} \ell^2$ is a meromorphic family of compact (in fact, finite rank) operators on this same domain. After showing that $I + K(\mu): q^{-N d(o, \cdot)} \ell^2 \to q^{-N d(o, \cdot)} \ell^2$ is invertible for some $\mu$ sufficiently large, we conclude by the analytic Fredholm theorem that
    \begin{gather*}
      R_\mathcal{G}(\mu) = P(\mu) (I + K(\mu))^{-1}
    \end{gather*}
    has the claimed properties.

    Let $\mathcal{F} = \mathcal{L} \sqcup \mathcal{F}_1 \sqcup \dots \sqcup \mathcal{F}_{n_f} \sqcup \mathcal{C}_1 \sqcup \dots \sqcup \mathcal{C}_{n_c}$ be the decomposition of $\mathcal{V}$ into a (not necessarily $(q+1)$-regular) compact core $\mathcal{L}$, a finite number $n_f$ of $(q+1)$-regular orbifold funnels $\mathcal{F}_i$, and a finite number $n_c$ of $(q+1)$-regular cusps $\mathcal{C}_i$. Let $H_j$ be the finite group which we can quotient the standard funnel by to obtain $\mathcal{F}_j$, and let $\mathcal{P}_j := H_j \backslash \backslash \mathcal{T}$. Let $\mathcal{L}'$ denote all vertices at distance at most 1 from $\mathcal{L}$. Let $\mathcal{L}''$ denote all vertices at distance at most 2 from $\mathcal{L}$. Let $\mathcal{F}_k'$ and $\mathcal{C}_\ell'$, respectively, denote all vertices of distance at least one from the root of $\mathcal{F}_k$ and $\mathcal{C}_\ell$, respectively. We thus have
    \begin{gather*}
        \mathcal{V} = \mathcal{L}' \sqcup \mathcal{F}_1' \sqcup \dots \sqcup \mathcal{F}_{n_f}' \sqcup \mathcal{C}_1' \sqcup \dots \sqcup \mathcal{C}_{n_c}'.
    \end{gather*}
    
    For each orbifold funnel $\mathcal{F}_j$, we have a natural inclusion $\iota_{\mathcal{F}_j}: \mathcal{F}_j \hookrightarrow \mathcal{P}_j$; we use the same notation for the natural map $C(\mathcal{F}_j) \hookrightarrow C(\mathcal{P}_j)$. We also have the adjoint map $\iota_{\mathcal{F}_j}^*: C(\mathcal{P}_j) \to C(\mathcal{F}_j)$. Let $\mathcal{P}_j' := \iota_{\mathcal{F}_j}(\mathcal{F}_j')$. We likewise have for each $\mathcal{C}_\ell$ inclusions $\iota_{\mathcal{C}_\ell}: \mathcal{C}_\ell \hookrightarrow \mathcal{Z}$; again we use the same notation for the map $C(\mathcal{C}_\ell) \hookrightarrow C(\mathcal{Z})$. Let $\mathcal{C}' := \iota_{\mathcal{C}_\ell}(\mathcal{C}_\ell')$ (the definition does not depend on the choice of $\ell$). We have that
    \begin{align}
        A_{\mathcal{G}} \iota^*_{\mathcal{F}_j} f &= \iota^*_{\mathcal{F}_j} A_{\mathcal{P}_j} f \label{T_1} \\
        A_{\mathcal{G}} \iota^*_{\mathcal{C}_\ell} g &= \iota^*_{\mathcal{C}_\ell} A_{\mathcal{Z}} g \label{C_1}
    \end{align}
    whenever $f \in C(\mathcal{P}_j)$ is supported on $\mathcal{P}_j'$, and $g \in C(\mathcal{Z})$ is supported on $\mathcal{C}'$.

    Let $R_{\mathcal{L}''}(\mu)$ denote the resolvent of the induced subgraph associated to the vertices in $\mathcal{L}''$ (i.e.\@ we keep all edges between any two vertices in $\mathcal{L}''$). We remark for future use that if we write $A_{\mathcal{G}} = A_{\mathcal{L}''} + B$, then the image of $B$ is supported on $\mathcal{G} \setminus \mathcal{L}'$ because every edge from a vertex in $\mathcal{L}'$ is also an edge in the induced graph of $\mathcal{L}''$.
    
    We define
    \begin{gather}
      P(\mu) \coloneqq  \mathbbm{1}_{\mathcal{L}'} R_{\mathcal{L}''}(\mu_0) \mathbbm{1}_{\mathcal{L}'} + \sum_{k = 1}^{n_f} \iota^*_{\mathcal{F}_j} \mathbbm{1}_{\mathcal{P}_j'} R_{\mathcal{P}_j}(\mu) \iota_{\mathcal{F}_j} \mathbbm{1}_{\mathcal{F}_j'} + \sum_{\ell = 1}^{n_c} \iota_{\mathcal{C}_\ell}^* \mathbbm{1}_{\mathcal{C}'} R_{\mathcal{Z}}(\mu) \iota_{\mathcal{C}_\ell} \mathbbm{1}_{\mathcal{C}_\ell'}, \label{eqn_P_mu}
    \end{gather}
    where $\mu_0$ is some fixed value which we later set to be sufficiently large. It is immediate that if $f \in q^{-N d(o, \cdot)} \ell^2(\mathcal{V})$, then $\iota_{\mathcal{F}_j} \mathbbm{1}_{\mathcal{F}_j'} f \in q^{-N d(o_{\mathcal{P}_j}, \cdot)} \ell^2(\mathcal{P}_j)$, and $\iota_{\mathcal{C}_\ell} \mathbbm{1}_{\mathcal{C}_\ell'} f \in q^{-N d(o_{\mathcal{Z}}, \cdot)} \ell^2(\mathcal{C}')$. Similarly if $g_j \in q^{N d(o_{\mathcal{P}_j}, \cdot)} \ell^2(\mathcal{P}_j)$, then $\iota_{\mathcal{F}_j}^* \mathbbm{1}_{\mathcal{P}_j'} g_j \in q^{N d(o, \cdot)} \ell^2(\mathcal{V})$, and if $g_\ell \in q^{N d(o_{\mathcal{Z}}, \cdot)} \ell^2(\mathcal{C}')$, then $\iota_{\mathcal{C}_\ell}^* g_\ell \in q^{N d(o, \cdot)} \ell^2(\mathcal{V})$. By Theorems \ref{theorem:cqzdnhg2al} (with Remark \ref{rmk_orbifold_funnel}) and \ref{thm_cusp}, we conclude that $P(\mu)$ is a finitely meromorphic family of operators from $q^{-Nd(o, \cdot)} \ell^2(\mathcal{V}) \to q^{N d(o, \cdot)} \ell^2(\mathcal{V})$ for $|\mu| > q^{\frac{1}{2}-N}$, and it is holomorphic for $\mu \notin \{\pm 1, \pm \frac{1}{\sqrt{q}} \}$. 

    We now consider $\big(\frac{A_{\mathcal{G}}}{2 \sqrt{q}} - z(\mu)\big) P(\mu)$. We obtain
    \begin{align}
        &\Big(\frac{A_\mathcal{G}}{2 \sqrt{q}} - z(\mu) \Big) \mathbbm{1}_{\mathcal{L}'} R_{\mathcal{L}''}(\mu_0) \mathbbm{1}_{\mathcal{L}'} + \sum_{j = 1}^{n_f} \iota_{\mathcal{F}_j}^* \Big(\frac{A_{\mathcal{P}_j}}{2 \sqrt{q}} - z (\mu) \Big) \mathbbm{1}_{\mathcal{P}_j'} R_{\mathcal{P}_j}(\mu) \iota_{\mathcal{F}_j} \mathbbm{1}_{\mathcal{F}_j'} \label{p_1} \\
        &+ \sum_{k = 1}^{n_c} \iota^*_{\mathcal{C}_k} \Big(\frac{A_{\mathcal{Z}}}{2 \sqrt{q}} - z(\mu)\Big) \mathbbm{1}_{\mathcal{C}'} R_{\mathcal{Z}}(\mu) \iota_{\mathcal{C}_k} \mathbbm{1}_{\mathcal{C}_k'} \label{p_2} \\
        =& \  \mathbbm{1}_{\mathcal{L}'} \Big(\frac{A_\mathcal{G}}{2 \sqrt{q}} - z(\mu) \Big) R_{\mathcal{L}''}(\mu_0) \mathbbm{1}_{\mathcal{L}'} + \sum_{j = 1}^{n_f} \iota_{\mathcal{F}_j}^*  \mathbbm{1}_{\mathcal{P}_j'} \Big(\frac{A_{\mathcal{P}_j}}{2 \sqrt{q}} - z (\mu)\Big) R_{\mathcal{P}_j}(\mu) \iota_{\mathcal{F}_j} \mathbbm{1}_{\mathcal{F}_j'} \label{p_3} \\
        &+ \sum_{k = 1}^{n_c} \iota^*_{\mathcal{C}_k}  \mathbbm{1}_{\mathcal{C}'} \Big(\frac{A_{\mathcal{Z}}}{2 \sqrt{q}} - z(\mu)\Big) R_{\mathcal{Z}}(\mu) \iota_{\mathcal{C}_k} \mathbbm{1}_{\mathcal{C}_k'} + \frac{[A_{\mathcal{G}}, \mathbbm{1}_{\mathcal{L}'}]}{2 \sqrt{q}} R_{\mathcal{L}''}(\mu_0) \mathbbm{1}_{\mathcal{L}'} \label{p_4} \\
        &+ \sum_{j = 1}^{n_f} \iota_{\mathcal{F}_j}^* \frac{[A_{\mathcal{P}_j}, \mathbbm{1}_{\mathcal{P}_j'}]}{2 \sqrt{q}} R_{\mathcal{P}_j}(\mu) \iota_{\mathcal{F}_j} \mathbbm{1}_{\mathcal{F}_j'} + \sum_{k = 1}^{n_c} \iota_{\mathcal{C}_k}^* \frac{[A_\mathcal{Z}, \mathbbm{1}_{\mathcal{C}'}]}{2 \sqrt{q}} R_{\mathcal{Z}}(\mu) \iota_{\mathcal{C}_k} \mathbbm{1}_{\mathcal{C}_k'}. \label{p_5}
    \end{align}
Note that in the second term in \eqref{p_1} and the term in \eqref{p_2} we have used \eqref{T_1} and \eqref{C_1}. Going from the first two lines to the last three lines simply involves repeatedly using the identity $AB = BA + [A, B]$. We immediately recognize that the second term in \eqref{p_3} is equal to $\sum \mathbbm{1}_{\mathcal{F}_j'}$, and the first term in \eqref{p_4} is equal to $\sum \mathbbm{1}_{\mathcal{C}_j'}$. To analyze the first term in \eqref{p_3}, we write $A_{\mathcal{G}} = A_{\mathcal{L}''} + B$, and use the resolvent identity:
\begin{gather*}
    R_{\mathcal{L}''}(\mu_0) = R_{\mathcal{L}''}(\mu) + (z(\mu_0) - z(\mu)) R_{\mathcal{L}''}(\mu) R_{\mathcal{L}''}(\mu_0).
\end{gather*}
We then get
\begin{align*}
    \mathbbm{1}_{\mathcal{L}'} &\Big(\frac{A_\mathcal{G}}{2 \sqrt{q}} - z(\mu)\Big) R_{\mathcal{L}''}(\mu_0) \mathbbm{1}_{\mathcal{L}'} = \mathbbm{1}_{\mathcal{L}'} \Big(\frac{B}{2 \sqrt{q}} + \frac{A_{\mathcal{L}''}}{2 \sqrt{q}} - z(\mu)\Big) R_{\mathcal{L}''}(\mu_0) \mathbbm{1}_{\mathcal{L}'} \\
    &= \mathbbm{1}_{\mathcal{L}'}\frac{B}{2 \sqrt{q}} R_{\mathcal{L}''}(\mu_0) \mathbbm{1}_{\mathcal{L}'} + \mathbbm{1}_{\mathcal{L}'} + (z(\mu_0) - z(\mu)) \mathbbm{1}_{\mathcal{L}'}R_{\mathcal{L}''}(\mu_0) \mathbbm{1}_{\mathcal{L}'}
\end{align*}
As previously mentioned, the image of $B$ is supported on $\mathcal{G} \setminus \mathcal{L}'$. Therefore the first term above is simply zero. Because $\mathbbm{1}_{\mathcal{L}'} + \sum \mathbbm{1}_{\mathcal{F}_j'} + \sum \mathbbm{1}_{\mathcal{C}_k'} = I$, we obtain that
\begin{gather*}
  \Big(\frac{A_{\mathcal{G}}}{2 \sqrt{q}} - z(\mu) \Big) P(\mu) = I + K(\mu),
\end{gather*}
where
\begin{align}
  K(\mu) &\coloneqq  (z(\mu_0) - z(\mu)) \mathbbm{1}_{\mathcal{L}'}R_{\mathcal{L}''}(\mu_0) \mathbbm{1}_{\mathcal{L}'} + \frac{[A_{\mathcal{G}}, \mathbbm{1}_{\mathcal{L}'}]}{2 \sqrt{q}} R_{\mathcal{L}''}(\mu_0) \mathbbm{1}_{\mathcal{L}'} \label{K_1} \\
       &+ \sum_{j = 1}^{n_f} \iota_{\mathcal{F}_j}^* \frac{[A_{\mathcal{P}_j}, \mathbbm{1}_{\mathcal{P}_j'}]}{2 \sqrt{q}} R_{\mathcal{P}_j}(\mu) \iota_{\mathcal{F}_j} \mathbbm{1}_{\mathcal{F}_j'} + \sum_{k = 1}^{n_c} \iota_{\mathcal{C}_k}^* \frac{[A_\mathcal{Z}, \mathbbm{1}_{\mathcal{C}'}]}{2 \sqrt{q}} R_{\mathcal{Z}}(\mu) \iota_{\mathcal{C}_k} \mathbbm{1}_{\mathcal{C}_k'}. \notag
\end{align}

We now claim that the image of $K(\mu)$ always lies in $C(\mathcal{L}'')$, which is a finite-dimensional space. The first term in \eqref{K_1} clearly has this property. In general if $A$ is the adjacency operator of some graph, and $X$ is some subgraph, then $[A, \mathbbm{1}_X]$ has image supported on $\partial X \cup (X' \setminus X)$ where 
$X'$ is the set of vertices at distance at most one from $X$. To see this, we simply record the image of $\delta_x$ for every vertex $x$. If $x \in (X \setminus \partial X)$, then $[A, \mathbbm{1}_X] \delta_x = 0$. Similarly if $x \notin X'$, then $[A, \mathbbm{1}_X] \delta_x = 0$. Suppose $x \in (X' \setminus X)$. Then $A \mathbbm{1}_X \delta_x = 0$, whereas $\mathbbm{1}_X A \delta_x$ is supported on $\partial X$. Similarly, if $x \in \partial X$, then $A \mathbbm{1}_X \delta_x = A \delta_x$ which is supported on $X'$. However for $y \in X$, we have $[A \delta_x](y) = [\mathbbm{1}_X A \delta_x](y)$, so $[A, \mathbbm{1}_X] \delta_x$ is supported on $X' \setminus X$. This observation immediately implies that the image of $K(\mu)$ always lies in $C(\mathcal{L}'')$. Furthermore it is clear from the formula for $K(\mu)$ that it defines a finitely meromorphic family of compact operators:
$K(\mu) : q^{-N d(o, \cdot)} \ell^2(\mathcal{V}) \to C(\mathcal{L}'') \subset q^{-N d(o, \cdot)} \ell^2(\mathcal{V})$
for $|\mu| > q^{\frac{1}{2}-N}$. 

Now, we can clearly choose $\mu_0$ and $\mu$ sufficiently large so that $\|R_{\mathcal{L}''}(\mu_0)\|$, $\|R_{\mathcal{P}_j}(\mu)\|$ and $\|R_{\mathcal{Z}}(\mu)\|$ are all as small as desired, and thus we can find a $\mu_0$ and $\mu_1$ such that $\|K(\mu_1)\| < \varepsilon$ for any $\varepsilon > 0$. In particular, taking any $\varepsilon < 1$, we obtain that $(I + K(\mu_1))^{-1}$ exists as a convergent Neumann series. The meromorphic Fredholm theorem then implies that 
\begin{gather*}
  (I + K(\mu))^{-1} : q^{-Nd(o, \cdot)} \ell^2(\mathcal{V}) \to q^{-Nd(o, \cdot)} \ell^2(\mathcal{V})
\end{gather*}
is finitely meromorphic on the specified domain. We therefore obtain that $R(\mu) = P(\mu) (I + K(\mu))^{-1}$ is finitely meromorphic on this domain as a family of operators $q^{-Nd(o, \cdot)} \ell^2(\mathcal{V}) \to q^{Nd(o, \cdot)} \ell^2(\mathcal{V})$. 
    
\end{proof}

One consequence of the proof is the following bound on the multiplicity of a resonance. It is likely that this bound can be greatly improved, but for now we content ourselves with simply deriving a finite explicit bound.

\begin{proposition} \label{prop_bound_multiplicity}
  The multiplicity of a resonance is at most $(2 (|\mathcal{L}| + 2 n_c + (1 + q) n_f)+1) (|\mathcal{L}| + 2 n_c + (1 + q) n_f)$.
\end{proposition}

\begin{proof}
    This follows from analyzing the explicit nature of $(1 + K(\mu))^{-1}$ in the proof. We know that the image of $K(\mu)$ is contained in $C(\mathcal{L}'')$. We could thus write $K(\mu)$ with respect to $C(\mathcal{L}'') \oplus C(\mathcal{G} \setminus \mathcal{L}'')$:
    \begin{gather*}
        K(\mu) = \begin{bmatrix}
            A(\mu) & B(\mu) \\
            0 & 0
        \end{bmatrix}.
    \end{gather*}
    We can then explicitly write
    \begin{gather*}
        (I + K(\mu))^{-1} = \begin{bmatrix}
            (I_{\mathcal{L}''} + A(\mu))^{-1} & -(I_{\mathcal{L}''} + A(\mu))^{-1} B(\mu) \\
            0 & I_{\mathcal{G} \setminus \mathcal{L}''}
        \end{bmatrix}.
    \end{gather*}
    We thus see that if we were to Laurent expand at a pole, any coefficient of a negative power in the expansion would be of the form
    \begin{gather*}
        \begin{bmatrix}
            C & D \\
            0 & 0
        \end{bmatrix},
    \end{gather*}
    and thus in particular the image of all such coefficients lies in $C(\mathcal{L}'')$, which has dimension $|\mathcal{L}| + 2 n_c + (1 + q) n_f$. However, we also see that $(I_{\mathcal{L}''} + A(\mu))^{-1}$ has entries which have poles of degree at most $2 (|\mathcal{L}| + 2 n_c + (1 + q) n_f)$ (since each entry of $I_{\mathcal{L}''} + A(\mu)$ is a Laurent polynomial of degree at most 2), so the pole is of order at most this. To obtain the resolvent, we must multiply $(I + K(\mu))^{-1}$ by $P(\mu)$, which itself has poles of order at most 1. Negative Laurent coefficients of $R_{\mathcal{G}}(\mu)$ are thus of the form $\sum P_i B_j$ where $B_j$ is a negative Laurent coefficient of $(I + K(\mu))^{-1}$, and there are at most $2 (|\mathcal{L}| + 2 n_c + (1 + q) n_f) + 1$ terms in the sum, each of which lands in a space of dimension at most $|\mathcal{L}| + 2 n_c + (1 + q) n_f$. We thus obtain the claimed bound on the multiplicity.
\end{proof}

\section{Characterization of resonant states} \label{sec_resonant_states}

\subsection{Purely outgoing functions}

Let $\mathcal{G}$ be geometrically finite. Let $\mathcal{H}$ be any finite connected subgraph of $\mathcal{G}$ containing the compact core $\mathcal{L}$. Then $\mathcal{G} \setminus \mathcal{H} = \mathcal{C}_1 \sqcup \dots \sqcup \mathcal{C}_{n} \sqcup \mathcal{F}_1 \sqcup \dots \sqcup \mathcal{F}_{m}$ consists of finitely many funnels $\mathcal{F}_j$ and cusps $\mathcal{C}_\ell$. Note, however, that it is possible that several of these funnels are subfunnels of the same single funnel appearing as a connected component of $\mathcal{G} \setminus \mathcal{L}$. See Figure \ref{fig_exceptional}.

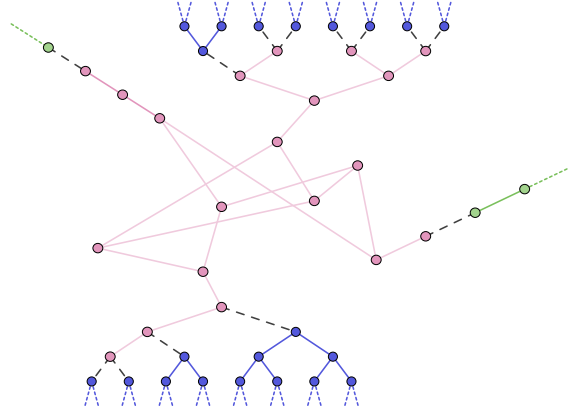
\begin{figure}[ht]
  \centering
  

  \resizebox{0.6\textwidth}{!}{%
    \begin{tikzpicture}[x=1.1cm,y=1.05cm,line cap=round,line join=round]
      \definecolor{b}{RGB}{83,88,220}
      \definecolor{p}{RGB}{226,149,188}
      \definecolor{g}{RGB}{122,192,92}
      \def\TreeLevelDist{0.42}
      \def\TreeBaseSep{0.6}
      \tikzset{tree node/.style={circle,draw=black,thin,fill=b,inner sep=0pt,minimum size=4.6pt}}
      \newcommand{\ThreeRegularTree}[2]{
        \coordinate (#1-0) at (0,0);
        \foreach \i in {0,...,7} {
          \pgfmathsetmacro{\x}{(\i - 3.5) * \TreeBaseSep}
          \coordinate (#1-3-\i) at (\x,{#2*3*\TreeLevelDist});
        }

        \pgfmathsetmacro{\x}{(0*2 + 0.5 - 3.5) * \TreeBaseSep}
        \coordinate (#1-2-0) at (\x,{#2*2*\TreeLevelDist});
        \pgfmathtruncatemacro{\childA}{2*0}
        \pgfmathtruncatemacro{\childB}{2*0+1}
        \draw[thick, b] (#1-2-0)--(#1-3-\childA);
        \draw[thick, b] (#1-2-0)--(#1-3-\childB);
        \foreach \i in {1,...,3} {
          \pgfmathsetmacro{\x}{(\i*2 + 0.5 - 3.5) * \TreeBaseSep}
          \coordinate (#1-2-\i) at (\x,{#2*2*\TreeLevelDist});
          \pgfmathtruncatemacro{\childA}{2*\i}
          \pgfmathtruncatemacro{\childB}{2*\i+1}
          \draw[dash pattern=on 4pt off 4pt,thick,darkgray] (#1-2-\i)--(#1-3-\childA);
          \draw[dash pattern=on 4pt off 4pt,thick,darkgray] (#1-2-\i)--(#1-3-\childB);
        }
        \pgfmathsetmacro{\x}{(0*4 + 1.5 - 3.5) * \TreeBaseSep}
        \coordinate (#1-1-0) at (\x,{#2*\TreeLevelDist});
        \pgfmathtruncatemacro{\childA}{2*0}
        \pgfmathtruncatemacro{\childB}{2*0+1}
        \draw[dash pattern=on 4pt off 4pt,thick,darkgray] (#1-1-0)--(#1-2-\childA);
        \draw[thick,p!50] (#1-1-0)--(#1-2-\childB);
        \draw[thick,p!50] (#1-0)--(#1-1-0);

        \pgfmathsetmacro{\x}{(1*4 + 1.5 - 3.5) * \TreeBaseSep}
        \coordinate (#1-1-1) at (\x,{#2*\TreeLevelDist});
        \pgfmathtruncatemacro{\childA}{2*1}
        \pgfmathtruncatemacro{\childB}{2*1+1}
        \draw[thick,p!50] (#1-1-1)--(#1-2-\childA);
        \draw[thick,p!50] (#1-1-1)--(#1-2-\childB);
        \draw[thick,p!50] (#1-0)--(#1-1-1);

        \foreach \i in {0,...,7}{
          \draw[dash pattern=on 1.3pt off 1.5pt,thick,b] (#1-3-\i)--++(-0.11,{#2*0.42});
          \draw[dash pattern=on 1.3pt off 1.5pt,thick,b] (#1-3-\i)--++(0.11,{#2*0.42});
        }

        \filldraw[fill=p,draw=black,thin] (#1-0) circle (0.08);
        \foreach \i in {0,...,7} {
          \node[tree node] at (#1-3-\i) {};
        }
        \node[tree node] at (#1-2-0) {};
        \foreach \i in {1,...,3} {
          \filldraw[fill=p,draw=black,thin] (#1-2-\i) circle (0.08);
        }
        \foreach \i in {0,...,1} {
          \filldraw[fill=p,draw=black,thin] (#1-1-\i) circle (0.08);
        }
      }

      \newcommand{\ThreeRegularTreeBelow}[2]{
        \coordinate (#1-0) at (0,0);
        \foreach \i in {0,...,7} {
          \pgfmathsetmacro{\x}{(\i - 3.5) * \TreeBaseSep}
          \coordinate (#1-3-\i) at (\x,{#2*3*\TreeLevelDist});
        }

        \pgfmathsetmacro{\x}{(0*2 + 0.5 - 3.5) * \TreeBaseSep}
        \coordinate (#1-2-0) at (\x,{#2*2*\TreeLevelDist});
        \pgfmathtruncatemacro{\childA}{2*0}
        \pgfmathtruncatemacro{\childB}{2*0+1}
        \draw[dash pattern=on 4pt off 4pt,thick,darkgray]  (#1-2-0)--(#1-3-\childA);
        \draw[dash pattern=on 4pt off 4pt,thick,darkgray]  (#1-2-0)--(#1-3-\childB);
        \foreach \i in {1,...,3} {
          \pgfmathsetmacro{\x}{(\i*2 + 0.5 - 3.5) * \TreeBaseSep}
          \coordinate (#1-2-\i) at (\x,{#2*2*\TreeLevelDist});
          \pgfmathtruncatemacro{\childA}{2*\i}
          \pgfmathtruncatemacro{\childB}{2*\i+1}
          \draw[thick,b] (#1-2-\i)--(#1-3-\childA);
          \draw[thick,b] (#1-2-\i)--(#1-3-\childB);
        }
        \pgfmathsetmacro{\x}{(0*4 + 1.5 - 3.5) * \TreeBaseSep}
        \coordinate (#1-1-0) at (\x,{#2*\TreeLevelDist});
        \pgfmathtruncatemacro{\childA}{2*0}
        \pgfmathtruncatemacro{\childB}{2*0+1}
        \draw[thick, p!50] (#1-1-0)--(#1-2-\childA);
        \draw[dash pattern=on 4pt off 4pt,thick,darkgray] (#1-1-0)--(#1-2-\childB);
        \draw[thick,p!50] (#1-0)--(#1-1-0);

        \pgfmathsetmacro{\x}{(1*4 + 1.5 - 3.5) * \TreeBaseSep}
        \coordinate (#1-1-1) at (\x,{#2*\TreeLevelDist});
        \pgfmathtruncatemacro{\childA}{2*1}
        \pgfmathtruncatemacro{\childB}{2*1+1}
        \draw[thick, b] (#1-1-1)--(#1-2-\childA);
        \draw[thick, b] (#1-1-1)--(#1-2-\childB);
        \draw[dash pattern=on 4pt off 4pt,thick,darkgray] (#1-0)--(#1-1-1);

        \foreach \i in {0,...,7}{
          \draw[dash pattern=on 1.3pt off 1.5pt,thick,b] (#1-3-\i)--++(-0.11,{#2*0.42});
          \draw[dash pattern=on 1.3pt off 1.5pt,thick,b] (#1-3-\i)--++(0.11,{#2*0.42});
        }

        \filldraw[fill=p,draw=black,thin] (#1-0) circle (0.08);
        \foreach \i in {0,...,7} {
          \node[tree node] at (#1-3-\i) {};
        }
        \filldraw[fill=p,draw=black,thin] (#1-2-0) circle (0.08);
        \foreach \i in {1,...,3} {
          \node[tree node] at (#1-2-\i) {};
        }
        \filldraw[fill=p,draw=black,thin] (#1-1-0) circle (0.08);
        \node[tree node] at (#1-1-1) {};
      }

      \coordinate (p1) at (-2,1.2);
      \coordinate (p2) at (-0.1,0.8);
      \coordinate (p3) at (1.2,0.4);
      \coordinate (p4) at (1.5,-1.2);
      \coordinate (p5) at (-1,- 0.3);
      \coordinate (p6) at (-1.3,-1.4);
      \coordinate (p7) at (-3,-1);
      \coordinate (p9) at (0.5,-0.2);

      \draw[thick,p!50] (p1)--(p5);
      \draw[thick,p!50] (p1)--(p4);
      \draw[thick,p!50] (p2)--(p7);
      \draw[thick,p!50] (p2)--(p9);
      \draw[thick,p!50] (p3)--(p4);
      \draw[thick,p!50] (p3)--(p5);
      \draw[thick,p!50] (p3)--(p9);
      \draw[thick,p!50] (p5)--(p6);
      \draw[thick,p!50] (p6)--(p7);
      \draw[thick,p!50] (p7)--(p9);

      \coordinate (u) at (0.5,1.5);
      \draw[thick,p!50] (u)--(p2);
      \begin{scope}[shift={(u)}]
        \ThreeRegularTree{tu}{1}
      \end{scope}

      \coordinate (d) at (- 1,-2);
      \draw[thick, p!50] (d)--(p6);
      \begin{scope}[shift={(d)}]
        \ThreeRegularTreeBelow{td}{-1}
      \end{scope}

      \draw[dash pattern=on 1.3pt off 1.5pt,thick,g] (-4.4,2.8)--(-3.8,2.4);
      \draw[dash pattern=on 4pt off 4pt,thick,darkgray] (-3.8,2.4)--(-3.2,2.0);
      \draw[thick,p] (-3.2,2.0)--(-2.6,1.6);
      \draw[thick,p] (-2.6,1.6)--(p1);

      \draw[thick,p!50] (p4)--(2.3,-0.8);
      \draw[dash pattern=on 4pt off 4pt,thick,darkgray] (2.3,-0.8)--(3.1,-0.4);
      \draw[thick,g] (3.1,-0.4)--(3.9,0);
      \draw[dash pattern=on 1.3pt off 1.5pt,thick,g] (3.9,0)--(4.7,0.4);

      \foreach \pt in {p1,p2,p3,p4,p5,p6,p7,p9}{
        \filldraw[fill=p,draw=black,thin] (\pt) circle (0.08);
      }
      \foreach \x/\y in {2.3/-0.8,-3.2/2.0,-2.6/1.6}{
        \filldraw[fill=p,draw=black,thin] (\x,\y) circle (0.08);
      }
      \foreach \x/\y in {-3.8/2.4,3.1/-0.4,3.9/0}{
        \filldraw[fill=g!70!white,draw=black,thin] (\x,\y) circle (0.08);
      }
    \end{tikzpicture}
  }

  \caption{Here $\mathcal{G}$ has two cusps and two funnels. The pink region represents a possible exceptional locus $\mathcal{H}$ of an outgoing function. When we remove $\mathcal{H}$, we are left with 11 funnels and 2 cusps.}
  \label{fig_exceptional}
  \end{figure}

    \begin{definition} \label{defn_outgoing}
        Let $\mu \in \mathbb{C}^\times$, and let $\mathcal{H}$ be a finite connected subgraph containing $\mathcal{L}$ with $\mathcal{G} \setminus \mathcal{H} = \mathcal{C}_1 \sqcup \dots \sqcup \mathcal{C}_m \sqcup \mathcal{F}_1 \sqcup \dots \sqcup \mathcal{F}_n$. A function $f \in C(\mathcal{V})$ is called \textit{purely outgoing of parameter $\mu$ and exceptional locus $\mathcal{H}$} if
        \begin{enumerate}[(i)]
            \item on each $\mathcal{C}_\ell$ there exists a polynomial $p_{C_\ell}$ in one variable such that $f(v) = p_{C_\ell}(d(o, v)) (\frac{\sqrt{q}}{\mu})^{d(o, v)}$ for all $v \in C_\ell$ and any choice of $o \in \mathcal{L}$,
            \item on each $\mathcal{F}_\ell$, there exists a polynomial $p_{\mathcal{F}_\ell}$ in one variable such that $f(v) = p_{\mathcal{F}_\ell}(d(o, v)) (\mu \sqrt{q})^{-d(o, v)}$ for all $v \in \mathcal{F}_\ell$ and any choice of $o \in \mathcal{L}$.
        \end{enumerate}
    We say that a purely outgoing solution is \textit{of constant type} if all of the above polynomials are in fact constants.
    \end{definition}

This definition will be crucial in giving a geometric interpretation to the resonances. We can already immediately observe that if $u$ is a purely outgoing solution of constant type with parameter $\mu_0$ and exceptional locus $\mathcal{H}$, then $(\frac{A}{2 \sqrt{q}} - z(\mu_0)) u$ is compactly supported (it is supported on $\mathcal{H}'$).

\begin{lemma}\label{lemma:cq4fkr0wif}
      Let $f \in C_c(\mathcal{V})$ and define the meromorphic family of vectors
      \begin{equation*}
        \tilde{f}_\mu \coloneqq(I + K(\mu))^{-1} f
      \end{equation*}
      in $q^{N d(o, \cdot)} \ell^2(\mathcal{V})$ for $\lvert \mu \rvert > q^{\frac{1}{2}- N}$ with $K$ as in \eqref{K_1}. Then $\tilde{f}_\mu |_{\mathcal{V} \setminus(\mathrm{supp}f \cup \mathcal{L}'')} \equiv 0$.
    \end{lemma}

    \begin{proof}
      By the proof of Theorem \ref{thm:merom_resolvent} we have $(I + K(\mu)) \tilde{f}_\mu  = \tilde{f}_\mu $ on $\mathcal{V} \setminus \mathcal{L}''$. But for $\lvert \mu \rvert$ large enough we have $f =(I + K(\mu)) \tilde{f}_\mu$ so that $\tilde{f}_\mu$ has to vanish identically on $\mathcal{V} \setminus(\mathrm{supp}f \cup \mathcal{L}'')$ for these $\mu$. By meromorphicity, this holds for all $\lvert \mu \rvert > q^{\frac{1}{2}- N}$. 
    \end{proof}

\begin{proposition}\label{proposition:cq0kuqgmlp}
    Suppose $f \in C_c(\mathcal{V})$ and $\mu_0 \in \mathbb{C}^\times$. Suppose near $\mu_0$ we have
    \begin{gather*}
        P(\mu) = \frac{P_{-1}}{\mu - \mu_0} + \sum_{j = 0}^\infty P_j(\mu - \mu_0)^j.
    \end{gather*}
    \begin{enumerate}
        \item Suppose $\mu_0 \notin \{\pm 1, \pm \frac{1}{\sqrt{q}} \}$ (and thus $P_{-1} = 0$). Then $P_j f$ is purely outgoing of degree at most $j$.
        \item Suppose $\mu_0 \in \{\pm 1, \pm \frac{1}{\sqrt{q}}\}$. Then $P_j f$ is purely outgoing of degree at most $j + 1$. 
    \end{enumerate}
\end{proposition}

\begin{proof}
    Because $P(\mu)$ has a simple pole at $\mu_0 \in \{\pm 1, \pm \frac{1}{\sqrt{q}}\}$, we know that $S(\mu) := \frac{1}{2}(\mu - \mu^{-1})(\mu - q^{-1} \mu^{-1}) P(\mu)$ is holomorphic on all of $\mathbb{C}^\times$. Suppose near $\mu_0$ we have
    \begin{gather*}
        S(\mu) = \sum_{j = 0}^\infty S_j(\mu-\mu_0)^j.
    \end{gather*}
    Because $S(\mu)$ is holomorphic, we can compute the coefficients of its power series expansions by simply taking higher and higher derivatives.
    
    Let $\mathcal{H}$ be the smallest connected subgraph of $\mathcal{G}$ containing $\mathcal{L}$ and the support of $f$. Let $\mathcal{C}_i$ be one of the cusps of $\mathcal{G} \setminus \mathcal{H}$. Then by \eqref{eqn_resolvent_cusp} (labelling vertices of $\mathcal{C}_i$ by $\mathbb{N}$) we have
    \begin{align*}
        [S(\mu) f](k) &= -(\mu - q^{-1} \mu^{-1}) \sum_{
                      \ell \leq  k
        } \left(\frac{\sqrt{q}}{\mu}\right)^{k - \ell} f(\ell ) \\
        &= -(\mu - q^{-1} \mu^{-1}) \left(\frac{\sqrt{q}}{\mu}\right)^{k}\sum_{
          \ell \leq  k
        } \left(\frac{\mu}{\sqrt{q}}\right)^{ \ell} f(\ell ) = A_i(f, \mu) \left(\frac{\sqrt{q}}{\mu}\right)^{k},
      \end{align*}
      where $A_i(f, \mu)$ is a constant that depends on $f$ and $\mu$ (but not on $k$), proving that $S(\mu) f$ is purely outgoing of constant type on the cusps, and thus in particular $S(\mu_0) f = S_0 f$ is purely outgoing of constant type. 

      Now we handle the funnels. Let $\mathcal{F}_j$ be one of the funnels of $\mathcal{G} \setminus \mathcal{H}$ with root $o_{\mathcal{F}_j}$. Then $\mathcal{F}_j$ is a subfunnel of a bigger funnel $\mathcal{Y}_j$ of $\mathcal{G} \setminus \mathcal{L}$. Using \eqref{eq:cq4enzzhjz}, we have
      \begin{align*}
        [S(\mu) f](v) &= -(\mu - \mu^{-1}) \sum_{y \in \mathcal{Y}_j} \Big(\frac{1}{\mu \sqrt{q}}\Big)^{d(v, y)} f(y) \\
                    &= -(\mu - \mu^{-1}) \Big(\frac{1}{\mu \sqrt{q}}\Big)^{d(v, o_{\mathcal{F}_j})} \left(\sum_{y \in \mathcal{Y}_j \cap \mathcal{H}} \Big(\frac{1}{\mu \sqrt{q}}\Big)^{d(o_{\mathcal{F}_j}, y)} f(y)\right) \\
                    &= B_j(f, \mu) (\mu \sqrt{q})^{-d(v, o_{\mathcal{F}_j})}.
      \end{align*}
      Here we have used the fact that $d(v, y) = d(v, o_{\mathcal{F}_j}) + d(o_{\mathcal{F}_j}, y)$. We conclude that $S(\mu_0)f = S_0 f$ is purely outgoing of constant type.

      We now consider derivatives in $\mu$ of $S(\mu) f$. This follows directly from the product rule. On each cusp $\mathcal{C}_i$ we get
      \begin{gather*}
          \Big(\frac{\partial}{\partial \mu}\Big)^m \Big(A_i(f, \mu) \Big(\frac{\sqrt{q}}{\mu}\Big)^k\Big) = \sum_{\ell = 0}^m \binom{m}{\ell} \Big(\big(\frac{\partial}{\partial \mu}\big)^{m-\ell} A_i(f, \mu)\Big) \prod_{j = 0}^{\ell-1} \Big(\frac{-k-j}{\mu}\Big) \Big(\frac{\sqrt{q}}{\mu}\Big)^k.
      \end{gather*}
      The argument on each funnel is virtually identical. Seeing as $S_j f = \frac{1}{j!} \frac{\partial^j}{\partial \mu^j} \big|_{\mu = \mu_0} S(\mu) f$, we conclude that indeed $S_j f$ is purely outgoing of degree at most $j$.

      We now wish to relate the power series expansion of $S(\mu)$ at $\mu_0$ with the Laurent series expansion of $P(\mu)$ at $\mu_0$. In case $\mu_0 \notin \{\pm 1, \pm \frac{1}{\sqrt{q}}\}$, then $\frac{1}{(\mu - \mu^{-1})(\mu - q^{-1} \mu^{-1})}$ is holomorphic at $\mu = \mu_0$, and we conclude that $P_j$ can be expressed as a linear combination of $S_\ell$ with $\ell \leq j$, from which the result follows. In case $\mu_0 \in \{\pm 1, \pm \frac{1}{\sqrt{q}}\}$, then $\frac{1}{(\mu - \mu^{-1})(\mu - q^{-1} \mu^{-1})} = \frac{1}{\mu - \mu_0} Z(\mu)$ where $Z(\mu)$ is holomorphic at $\mu = \mu_0$, in which case we see that $P_j$ is a linear combination of $S_\ell$ with $\ell \leq j + 1$. 
\end{proof}

    \begin{proposition}\label{proposition:cq0ku2e4d0}
      If $R_{\mathcal{G}}(\mu_0)$ is defined, then $u = R_{\mathcal{G}}(\mu_0)f$ is the unique purely outgoing solution of constant type of
      \begin{equation}\label{eq:cq4hwfnu39}
        \left(\frac{A_{\mathcal{G}}}{2 \sqrt{q}} - z(\mu_0)\right)u = f
      \end{equation}
      for $f \in C_c(\mathcal{V})$.
    \end{proposition}

    \begin{proof}
      First off note that since $R_{\mathcal{G}}(\mu_0) = P(\mu_0) (I + K(\mu_0))^{-1}$, by Lemma \ref{lemma:cq4fkr0wif} and Proposition \ref{proposition:cq0kuqgmlp} we immediately get that $R_{\mathcal{G}}(\mu_0) f$ is purely outgoing of constant type. 

      Now suppose $u$ is purely outgoing of constant type. We then clearly have that $f = (\frac{A}{2 \sqrt{q}} - z(\mu_0)) u$ is compactly supported. We may therefore apply $R_{\mathcal{G}}(\mu_0)$ to it, i.e.\@ we get that
      \begin{equation*}
        u = R_{\mathcal{G}}(\mu_0) \Big(\frac{A}{2 \sqrt{q}} - z(\mu_0)\Big) u = R_{\mathcal{G}}(\mu_0) f.\qedhere
      \end{equation*}
    \end{proof}

    \begin{lemma}\label{lemma:cq4flfr8j6}
        Let $\mu \in \mathbb{C}^\times \setminus \{\pm 1\}$. Suppose $u$ is a purely outgoing solution with exceptional locus $\mathcal{H}$, parameter $\mu$, and degree $d$. Then
        \begin{gather*}
            \Big(\frac{A}{2 \sqrt{q}} - z(\mu)\Big)^J u = 0 \textnormal{ on $\mathcal{H}'$} \iff d < J.
        \end{gather*}
        In particular, if $u$ is a purely outgoing solution of 
        \begin{gather*}
            \Big(\frac{A}{2 \sqrt{q}} - z(\mu)\Big) u = 0,
        \end{gather*}
        then $u$ is purely outgoing of constant type. In case $\mu = \pm 1$, then $u$ has degree at most 1.
    \end{lemma}

    \begin{proof}
        On each cusp $\mathcal{C}_i$ and each funnel $\mathcal{F}_j$ of $\mathcal{G} \setminus \mathcal{H}$, with polynomials $p_{\mathcal{C}_i}$ and $p_{\mathcal{F}_j}$, one immediately calculates that for vertex $k \in \mathcal{C}_i$ and $v \in \mathcal{F}_j$ with $d(o_{\mathcal{F}_j}, v) = d$, we have
        \begin{align*}
            \Big[\Big(\frac{A}{2 \sqrt{q}} - z(\mu)\Big) u\Big](k) &= \frac{1}{2}(p_{\mathcal{C}_i}(k-1) \mu + p_{\mathcal{C}_i}(k+1) \mu^{-1} - p_{\mathcal{C}_i}(k)(\mu + \mu^{-1})) \Big(\frac{\sqrt{q}}{\mu}\Big)^k \\
            \Big[\Big(\frac{A}{2 \sqrt{q}} - z(\mu)\Big) u\Big](v) &= \frac{1}{2}(p_{\mathcal{F}_i}(d-1) \mu + p_{\mathcal{F}_i}(d+1) \mu^{-1} - p_{\mathcal{F}_i}(d)(\mu + \mu^{-1})) \Big(\mu \sqrt{q}\Big)^{-d}.
        \end{align*}
        In general, if $f(x)$ is a polynomial whose highest order term is of the form $c_d x^d$, then $g(x) = f(x-1) \mu + f(x+1) \mu^{-1} - f(x)(\mu + \mu^{-1})$ is degree at most $d-1$, and the coefficient of $x^{d-1}$ in $g(x)$ is $c_d d (\mu - \mu^{-1})$ which is non-zero if $\mu \neq \pm 1$. Therefore, in this case, $g(x)$ is a polynomial of degree exactly $d-1$. This shows inductively that if $\mu \neq \pm 1$ and if $u$ is purely outgoing of degree $d$, then $(\frac{A}{2 \sqrt{q}} - z(\mu))^J u$ is purely outgoing of degree $d - J$ (with the identically zero function viewed as the unique polynomial of degree $-k$ for every $k \geq 1$). As the only purely outgoing solution which is identically zero on $\mathcal{H}'$ is the function for which all of the associated polynomials are identically zero, the result follows. 
        
        The case of $\mu = \pm 1$ easily follows from direct analysis.
    \end{proof}

    \begin{corollary}
      If $0 \neq u$ is purely outgoing with parameter $\mu \neq \pm 1$ and 
      \begin{equation*}
        \left(\frac{A_{\mathcal{G}}}{2 \sqrt{q}} - z(\mu)\right)u = 0,
      \end{equation*}
      then $\mu$ is a resonance.
    \end{corollary}

    \begin{proof}
      Since $0$ is a purely outgoing solution of constant type by Lemma~\ref{lemma:cq4flfr8j6}, we have found two purely outgoing solutions so that Proposition~\ref{proposition:cq0ku2e4d0} implies that $\mu$ is a resonance.
    \end{proof}

    \begin{definition}\label{definition:cq4hv1yzhj}
      If $\mu_0$ is a resonance we write, for $\mu$ in a small neighborhood of $\mu_0$,
      \begin{equation*}
        R_{\mathcal{G}}(\mu) = \sum_{
          j = 1
        }^J \frac{A_{-j}}{(\mu - \mu_0)^j} + R_\textnormal{hol}(\mu),
      \end{equation*}
      where $A_j$ are finite-rank operators and $R_{\textnormal{hol}}(\mu)$ is holomorphic. Recall that we define
      \begin{equation*}
        \mathrm{grst}(A_{\mathcal{G}}, \mu_0) \coloneqq \textnormal{span}(\mathrm{im} A_{-1} \cup \dots \cup \mathrm{im} A_{-J})
      \end{equation*}
      the space of \emph{generalized resonant states at $\mu_0$}. We also define
      \begin{equation*}
        \mathrm{rst}(A_{\mathcal{G}}, \mu_0) \coloneqq \textnormal{grst}(A_{\mathcal{G}}, \mu_0) \cap \ker\left(\frac{A_{\mathcal{G}}}{2 \sqrt{q}} - z(\mu_0)\right)
      \end{equation*}
      the space of \emph{resonant states at $\mu_0$.} Note that, as each $A_{-j}$ has finite rank, $\mathrm{im}_{C_c(\mathcal{V})} A_{-j} = \mathrm{im}_{q^{- N d(o, \cdot)} \ell^2(\mathcal{V})} A_{-j}$.
    \end{definition}

    Before discussing the next proposition, we wish to motivate it by discussing the structure of the resolvent of an operator $B$ on a finite-dimensional space near an eigenvalue. Suppose $\lambda_0$ is an eigenvalue. Then near $\lambda_0$ we have an expansion of the form
    \begin{gather*}
        (\lambda - B)^{-1} = \sum_{j = 1}^J \frac{(B - \lambda_0)^{j-1} \Pi_{\lambda_0}}{(\lambda - \lambda_0)^j} + R(\lambda),
    \end{gather*}
    where $R(\lambda)$ is a holomorphic family of operators near $\lambda_0$, $\Pi_{\lambda_0}$ is the projection onto the $\lambda_0$ generalized eigenspace (which necessarily commutes with $B$ and thus with $B - \lambda_0$), and $J$ is such that $(B - \lambda_0)^J \Pi_{\lambda_0} = 0 = ((B - \lambda_0) \Pi_{\lambda_0})^J$. Note that in particular we have that $(B - \lambda_0) \Pi_{\lambda_0}$ is nilpotent. This result can be proved easily using the Jordan normal form. See e.g.\@ \cite{campbell_daners}. In light of the below proposition, the above expansion motivates generalized resonant states as the analogues of generalized eigenfunctions, and resonant states as the analogues of eigenfunctions.

    \subsection{Generalized resonant states: $\mu_0 \neq \pm 1$}

    \begin{proposition}\label{prop:cq4hwetoj6} 
      Let $\mu_0 \in \mathbb{C}^\times \setminus \{\pm 1\}$ be a resonance. Then there exists a neighborhood $U$ of $\mu_0$ such that
      \begin{gather*}
          R_{\mathcal{G}}(\mu) = \sum_{j = 1}^J \frac{A_{-j}}{(\mu - \mu_0)^j} + R_{\textnormal{hol}}(\mu) = \sum_{j = 1}^J \frac{B_{-j}}{(z(\mu) - z(\mu_0))^j} + R_\textnormal{hol}'(\mu)
      \end{gather*}
      where $R_{\textnormal{hol}}(\mu)$ and $R_{\textnormal{hol}}'(\mu)$ are holomorphic and
    \begin{enumerate}
        \item We have $B_{-j} = (\frac{A}{2 \sqrt{q}} - z(\mu_0))^{j-1} B_{-1}$.
        \item $-B_{-1}$ is a (finite-rank) projection operator.
        \item $\textnormal{span}(\textnormal{im}(B_{-1}) \cup \dots \cup \textnormal{im}(B_{-J})) = \textnormal{span}(\textnormal{im}(A_{- 1}) \cup \dots \cup \textnormal{im}(A_{- J}))$.
        \end{enumerate}
      \end{proposition}

      \begin{proof}
        Since $z'(\mu_0) \neq 0$ as long as $\mu_0 \neq \pm 1$, the inverse function theorem for holomorphic functions implies that there exists a neighborhood $U$ of $\mu_0$ such that $z(\mu): U \to V$ is biholomorphic with inverse $\mu(z): V \to U$. 
        
        Property (1) now follows immediately from the identity $(\frac{A}{2 \sqrt{q}} - z(\mu_0) - (z(\mu) - z(\mu_0)) R(\mu) = I$ by comparing coefficients in the Laurent series on both sides. Property (2) follows from standard resolvent formalism arguments; see e.g.\@ the proof of Theorem C.9 in \cite{dyatlov_zworski}.

        Property (3) essentially just follows by analyzing how Laurent series expansions change under holomorphic change of coordinates. We can write $\mu(z) - \mu_0 = \sum_{k = 1}^\infty a_k (z - z(\mu_0))^k$ with $a_1 \neq 0$. We have that
        \begin{gather*}
            \frac{1}{(\mu - \mu_0)^j} = \frac{1}{(z - z(\mu_0))^j} \frac{1}{(a_1 + a_2 (z - z(\mu_0)) + \dots)^j}, 
        \end{gather*}
        and $\frac{1}{(a_1 + a_2 (z - z(\mu_0)) + \dots)^j}$ is holomorphic near $z(\mu_0)$. We thus see that $A_{-j}$ can be expressed as a linear combination of $B_{-k}$ with $k \geq j$ and, by reversing the roles of $z$ and $\mu$, we see that $B_{-j}$ can be expressed as a linear combination of $A_{-k}$ with $k \geq j$. This shows Property (3).
    \end{proof}

    We can characterize the generalized resonant states as follows:

    \begin{theorem} \label{thm_resonant_state_good}
      If $\mu_0 \in \mathbb{C}^\times \setminus \{\pm 1\}$ is a resonance, then
      \begin{equation*}
        \mathrm{grst}(A_{\mathcal{G}}, \mu_0 ) = \left\{ u \in C(\mathcal{V}) \mid u \text{ purely outgoing},\, \exists \ J \in \mathbb{N} \colon \left(\frac{A_{\mathcal{G}}}{2 \sqrt{q}} - z(\mu_0)\right)^J u = 0\right\}. \end{equation*}
    \end{theorem}

    \begin{proof}
      Let us first prove that the left hand side is contained in the right hand side. Let $u \in \mathrm{grst}(A_{\mathcal{G}}, \mu_0)$ and recall that $R_{\mathcal{G}}(\mu) = P(\mu)(I + K(\mu))^{-1}$. We can write
      \begin{equation*}
        R_{\mathcal{G}}(\mu)= P(\mu)(I + K(\mu))^{-1} = \left(\sum_{
            k = -1
          }^{\infty}P_{k}(\mu - \mu_0)^k\right) \cdot\left(\sum_{
            j = 1
          }^K \frac{B_{-j}}{(\mu - \mu_0)^j}+ B_{\textnormal{hol}}(\mu)\right)
      \end{equation*}
      for some finite-rank operators $B_j \colon C_c(\mathcal{V}) \rightarrow C_c(\mathcal{V})$ and a holomorphic operator $B_{\textnormal{hol}}(\mu)$; let $B_0$ be the constant term in the power series expansion of $B_{\textnormal{hol}}(\mu)$ at $\mu_0$. Since $R_{\mathcal{G}}(\mu)$ is meromorphic we may write
      \begin{equation*}
        R_{\mathcal{G}}(\mu)= \sum_{
          j = 1
        }^J \frac{A_{-j}}{(\mu - \mu_0)^j}+ R_{\textnormal{hol}}(\mu).
      \end{equation*}
      We get that
        \begin{equation*}
            A_{-j} = \sum_{k = -1}^j P_k B_{-j-k}.
        \end{equation*}
        In case $\mu_0 \notin \{\pm 1, \pm \frac{1}{\sqrt{q}}\}$, we have that $P_{-1} = 0$, and we immediately get that every element in the image of $A_{-j}$ is purely outgoing by Proposition~\ref{proposition:cq0kuqgmlp}.

        Suppose $\mu_0 \in \{\pm 1, \pm \frac{1}{\sqrt{q}}\}$. It is clear that the image of every $A_{-k}$ with $k \geq 2$ is purely outgoing. However, for $k = -1$, we must consider the term $P_{-1} B_0$, and $B_0$ does not necessarily map compactly supported functions to compactly supported functions. In case $\mu_0 = \pm \frac{1}{\sqrt{q}}$, then $P_{-1}$ only arises from the part of $P(\mu)$ arising from the funnels. From the explicit formula \eqref{eq:cq4enzzhjz}, it is clear that the image of $P_{-1}$ consists of functions which are constant on the funnels in case $\mu_0 = \frac{1}{\sqrt{q}}$, or else are constant times $(-1)^{d(o, x)}$ on the funnels in case $\mu_0 = -\frac{1}{\sqrt{q}}$. These are exactly the outgoing conditions for such values of $\mu_0$. We can perform a similar analysis in case $\mu_0 = \pm 1$. We thus see that in all cases we have that $\textnormal{im}(A_{-j})$ is purely outgoing.

        We now consider the other direction. For this we will use in a crucial way that $\mu_0 \neq \pm 1$. Let $u$ be purely outgoing with parameter $\mu_0$ and annihilated by $(\frac{A_{\mathcal{G}}}{2 \sqrt{q}} - z(\mu_0))^M$ for some $M$. Suppose the exceptional locus of $u$ is $H$. Let $u_\mu$ denote the function which agrees with $u$ on $\mathcal{H}$, and is purely outgoing with parameter $\mu$ and with the same polynomials on each funnel and cusp as $u$.

      Then 
      \begin{equation*}
        f(\mu) \coloneqq\left(\frac{A_{\mathcal{G}}}{2 \sqrt{q}} - z(\mu)\right)^M u_\mu  
      \end{equation*}
      is finitely supported by Lemma~\ref{lemma:cq4flfr8j6} (which crucially uses the fact that $\mu_0 \neq \pm 1$) and $f(\mu_0) = 0$ so that we may write $f(\mu) =(z(\mu) - z(\mu_0))^k \tilde{f}(\mu )$ for some function $\tilde{f}$ with $\tilde{f}(\mu_0) \neq 0$ and some $k$.

      Let $R_{\mathcal{G}}(\mu)^M$ be defined via meromorphic continuation as the $M$th power of $R_{\mathcal{G}}(\mu)$ on the locus of $\mu$ on which it makes sense to take an $M$th power (i.e.\@ $|\mu| \gg 1$). Suppose near $\mu_0$ we have that
      \begin{gather*}
          R_{\mathcal{G}}(\mu) = \sum_{j = -J}^\infty B_j (z(\mu) - z(\mu_0))^j.
      \end{gather*}
      We claim then that near $\mu_0$,
      \begin{align}
        R_{\mathcal{G}}(\mu)^M &= \sum_{j = -J}^\infty \binom{j}{M-1} B_j (z(\mu) - z(\mu_0))^{j - (M-1)} \label{eqn_power_resolvent} \\
                               &= \sum_{\ell = -J-(M-1)}^{-1-(M-1)} \binom{\ell+(M-1)}{M-1} B_{\ell+(M-1)} (z(\mu) - z(\mu_0))^{\ell} + R'_{\textnormal{hol}}(\mu), \notag
      \end{align}
      where $R'_{\textnormal{hol}}(\mu)$ is holomorphic at $\mu_0$. This essentially follows from the identity
    \begin{gather*}
        \frac{1}{(1 - z)^M} = \frac{1}{(M-1)!} \frac{d^{M-1}}{d z^{M-1}} \frac{1}{1 - z}.
    \end{gather*}
    For $z(\mu) \gg 1$, we have the identity
    \begin{gather*}
      R_{\mathcal{G}}(\mu)^M = \frac{1}{(M-1)!} \frac{d^{M-1}}{d z^{M-1}} R_{\mathcal{G}}(\mu(z)).
    \end{gather*}
    By meromorphic continuation, we see that this identity is always true. The identity in \eqref{eqn_power_resolvent} then follows from simply termwise differentiating the Laurent series for $R(\mu)$.
      
    But then, for $\mu$ close to $\mu_0$, we have
    \begin{align*}
        u_\mu &= R_{\mathcal{G}}(\mu)^M(z(\mu) - z(\mu_0))^k \tilde{f}(\mu) \\
        &= \left(\sum_{\ell = -J-(M-1)}^{-1-(M-1)} \binom{\ell+(M-1)}{M-1} B_{\ell+(M-1)} (z(\mu) - z(\mu_0))^{\ell} + R'_{\textnormal{hol}}(\mu)\right) (z(\mu) - z(\mu_0))^k \tilde{f}(\mu).
      \end{align*}
      Let $L$ be the most negative number such that $B_{L + (M-1)} \tilde{f}(\mu_0) \neq 0$. We must have that $k \geq -L$ since the LHS is defined at $\mu = \mu_0$. However, if $k > -L$, then the RHS would vanish at $\mu = \mu_0$. Thus in fact $k = -L$, and we get $u_\mu = B_{L + (M-1)} \tilde{f}(\mu) + S_{\textnormal{hol}}(\mu) \tilde{f}(\mu)$, where $S_{\textnormal{hol}}(\mu)$ is a holomorphic family of operators which vanishes at $\mu = \mu_0$. Thus $u = u_{\mu_0} = B_{L + (M-1)} \tilde{f}(\mu_0)$, so $u$ is a generalized resonant state.
    \end{proof}

\subsection{Generalized resonant states: $\mu_0 = \pm 1$}

We now examine the case of $\mu_0 = \pm 1$. These points are tricky for a number of interrelated reasons which all center around the fact that $\pm 1$ are the ramification points of our parametrization of $\mu \mapsto z(\mu)$. 

\begin{proposition} \label{prop_mu_pm1}
    At $\mu_0 = \pm 1$, we have that the resolvent has a pole of order at most 2. Let $w(\mu) := \frac{\mu - \mu_0}{\sqrt{2} \mu^{\frac{1}{2}}}$. Suppose for $\mu$ close to $\mu_0$ we write:
    \begin{gather*}
        R_{\mathcal{G}}(\mu) = \frac{A_{-2}}{(\mu - \mu_0)^2} + \frac{A_{-1}}{\mu - \mu_0} + R_{\textnormal{hol}}(\mu) = \frac{B_{-2}}{w(\mu)^2} + \frac{B_{-1}}{w(\mu)} + R'_{\textnormal{hol}}(\mu).
    \end{gather*}
    Then
    \begin{enumerate}
        \item $w(\mu)^2 = z(\mu) - z(\mu_0)$.
        \item All elements in $\textnormal{im}(B_{-2})$ and $\textnormal{im}(B_{-1})$ are eigenfunctions of $\frac{A_{\mathcal{G}}}{2 \sqrt{q}}$ with eigenvalue $z(\mu_0)$.
        \item The image of $B_{-2}$ is exactly the $\ell^2$-eigenfunctions of $\frac{A_{\mathcal{G}}}{2 \sqrt{q}}$ with eigenvalue $z(\mu_0)$.
        \item $\textnormal{span}(\textnormal{im}(A_{-2}) \cup \textnormal{im}(A_{-1})) = \textnormal{span}(\textnormal{im}(B_{-2}) \cup \textnormal{im}(B_{-1}))$.
    \end{enumerate} 
\end{proposition}
\begin{proof}
  First off we remark that it is clear that $w(\mu)$ is a holomorphic change of variables for $\mu$ close to $\mu_0$. Furthermore, claim (1) is obvious. Granting for now that at $\mu_0$ we have a pole of order at most 2, then Claim (2) follows immediately from the identity $\Big( \frac{A_{\mathcal{G}}}{2 \sqrt{q}} - z(\mu_0) - w(\mu)^2 \Big) R_{\mathcal{G}}(\mu) = I$. Claim (4) follows in the same way that Claim (3) was proved in Proposition \ref{prop:cq4hwetoj6}. 

  We now follow the strategy of Lemma 8.8 of Borthwick \cite{Bor16}. The strategy is to use $\ell^2$ estimates for the resolvent in order to show that it can maximally diverge to second order at $\mu_0$. Note that for this argument to work it is essential that $\mu_0 = \pm 1$ lie in the $\ell^2$ spectrum of $A_{\mathcal G}$ and can be approached from the $\ell^2$ resolvent set, so the $\ell^2$ resolvent and our meromorphically continued resolvent coincide. 
  
  As a first step to get the $\ell^2$ bounds we note that, because $A_{\mathcal{G}}$ is self-adjoint, we get that, for $\phi \in C_c(\mathcal{V})$,
  \begin{align*}
        \Big|\sum_{v \in \mathcal{V}} \bar{\phi}(v) \Big[\Big(\frac{A_{\mathcal{G}}}{2 \sqrt{q}} - \frac{\mu + \mu^{-1}}{2} \Big) \phi \Big] (v) \Big| &\geq \Big|\textnormal{Im} \sum_{v \in \mathcal{V}} \bar{\phi}(v) \Big[\Big(\frac{A_{\mathcal{G}}}{2 \sqrt{q}} - \frac{\mu + \mu^{-1}}{2}\Big) \phi \Big] (v) \Big| \\
        &= \Big|\textnormal{Im}\big(\frac{\mu + \mu^{-1}}{2}\big)\Big| \|\phi\|^2. 
    \end{align*}
    By (1) we have $\textnormal{Im}(\frac{\mu + \mu^{-1}}{2}) = \textnormal{Im}(w(\mu)^2)$. Applying Cauchy-Schwarz, we get
    \begin{gather*}
      \Big|\left(\frac{A_{\mathcal{G}}}{2 \sqrt{q}} - z(\mu)\right) \phi \Big| \geq |\textnormal{Im}(w(\mu)^2) | \|\phi\|.
    \end{gather*}
    This in turn implies the bound
    \begin{gather*}
        \|R_{\mathcal{G}}(\mu)\| \leq \frac{1}{|\textnormal{Im}(w(\mu)^2)|},
    \end{gather*}
    for $\mu$ close to $\mu_0$. This in turn implies that $R(\mu)$ has a pole of order at most 2 at $\mu_0$.

    We thus see that $w(\mu)^2 R_{\mathcal G}(\mu)$ is a well-defined operator near $\mu = \mu_0$ acting from $C_c(\mathcal{V})$ to $C(\mathcal{V})$. Furthermore, from the $\ell^2$ bound, anything in its image is in $\ell^2(\mathcal{V})$. Therefore anything in the image of $B_{-2}$ is in $\ell^2(\mathcal{V})$. On the other hand, we already know by Claim (2) that everything in the image must be an eigenfunction with eigenvalue $z(\mu_0)$. 
    
    On the other hand, let $f\in \ell^2(\mathcal V)$ be an eigenfunction $\Big(\frac{A_{\mathcal G}}{2\sqrt q} -z(\mu_0)\Big) f=0$. Then we consider (again by approaching $\mu_0$ outside of the $\ell^2$ resolvent set)
    \begin{align*}
     f&= R_{\mathcal G}(\mu) \Big(\frac{A_{\mathcal G}}{2\sqrt{q}} -z(\mu)\Big)f =R_{\mathcal G}(\mu) \Big(\frac{A_{\mathcal G}}{2\sqrt{q}} -z(\mu_0) + w(\mu)^2\Big)f\\
    &=R_{\mathcal G}(\mu) w(\mu)^2f = B_{-2}f + O(w(\mu))
    \end{align*}
    We thus conclude, by taking the limit $\mu\to\mu_0$, that $f\in\textnormal{im}(B_{-2})$.

\end{proof}

\begin{proposition} \label{prop_always_outgoing}
    Suppose $\phi$ is an outgoing eigenfunction of constant type. Then it satisfies the outgoing condition on the entirety of all of the funnels and cusps.
\end{proposition}
\begin{proof}
    On the cusps this is immediate since if we know the value on an infinite subray of the cusp, the fact that we have an eigenfunction and that each vertex of the cusp has only two neighbors uniquely determines the value of the eigenfunction on the rest of the cusp, and it is immediate that it is indeed outgoing with the same parameters.

    On the funnels, the proof is also straightforward. We argue by contradiction. Suppose $x$ is a vertex of maximal height where the outgoing condition fails, i.e.\@ for every vertex of strictly higher height, the outgoing condition already holds. Let $y$ be a descendant one step higher than $x$ with value $c$. Then the $q$ descendants of $y$ must have value $c \frac{1}{\sqrt{q} \mu}$. We get the equation
    \begin{gather*}
      \frac{1}{2 \sqrt{q}}\left(\frac{c q}{\sqrt{q} \mu} + \phi(x)\right) = c \cdot \frac{\mu + \mu^{-1}}{2}
    \end{gather*}
    and when we solve we get $\phi(x) = c \mu \sqrt{q}$. This implies that all other descendant neighbors of $x$ also have value $c$, and that $\phi(x)$ together with its descendants satisfies the outgoing condition, contradicting the fact that we assumed that $x$ did not satisfy it.
\end{proof}

\begin{proposition} \label{prop_eigenfunction_zero}
    Suppose $\phi$ is an eigenfunction which is outgoing on all the funnels and cusps and is in $\ell^2(\mathcal{V})$ and has eigenvalue $z(\mu)$ with $\mu$ on the unit circle. Then $\phi$ is identically zero on all of the funnels and cusps.
\end{proposition}
\begin{proof}
    On the cusps for example, the value at height $k$ is $c (\frac{\sqrt{q}}{\mu})^k$ for some $c$. If we compute the contribution to the $\ell^2$-norm from the cusp, we obtain
    \begin{gather*}
        c^2 \cdot \sum_{k = 0}^\infty q^k \cdot q^{-k}
    \end{gather*}
    which is $\infty$ unless $c = 0$. 
    
    Similarly the outgoing condition on the funnels implies that the value is $c (\frac{1}{\sqrt{q} \mu})^k$ for vertices at distance $k$ from the root of the funnel. The contribution to the $\ell^2$-norm from the funnels is thus
    \begin{gather*}
        c^2 \cdot \sum_{k = 0}^\infty q^{-k} \cdot q^k
    \end{gather*}
    which again is infinite unless $c = 0$.
\end{proof}

\begin{theorem} \label{thm_bad_points}
    The generalized resonant states for $\mu_0 = \pm 1$ are the outgoing eigenfunctions of constant type with eigenvalue $z(\mu_0)$.
\end{theorem}

\begin{proof}
    Recall that $\mathcal{L}$ is the compact core, $\mathcal{L}'$ is all vertices of distance at most one from the compact core, and $\mathcal{L}''$ is all vertices of distance at most two from the compact core. Let $\mathcal{F}'$ denote all vertices on any funnel and at distance at least one from the root of the funnel (and recall that the root of every funnel lies in $\mathcal{L}' \setminus \mathcal{L}$). Similarly, let $\mathcal{C}'$ denote all vertices on any cusp and at distance at least one from the root of the cusp (and, similarly, recall that the root of every cusp lies in $\mathcal{L}' \setminus \mathcal{L}$). In particular we have $\mathcal{V} = \mathcal{L}' \sqcup \mathcal{F}' \sqcup \mathcal{C}'$. 

    The proof makes use of the explicit formula for the resolvent as $P(\mu) (I + K(\mu))^{-1}$. Recall that $P(\mu)$ can be decomposed into three parts; see \eqref{eqn_P_mu}. One part is constant, call it $Q$, for which the image of any element only depends on the restriction of that element to $\mathcal{L}'$, and its image is supported on $\mathcal{L}'$. Additionally we have a part $F(\mu)$ coming from the funnels for which the image of any element only depends on its restriction to $\mathcal{F}'$ and has image supported on $\mathcal{F}'$, and a part $C(\mu)$ coming from the cusps for which the image of any element only depends on its restriction to $\mathcal{C}'$ and has image supported on $\mathcal{C}'$.

    The components $F(\mu)$ and $C(\mu)$ are given by a very explicit formula. We first discuss $F(\mu)$. Suppose $f$ is a function supported on $\mathcal{L}''$. Let $\mathcal{F}_j'$ be all vertices on the funnel $\mathcal{F}_j$ at distance at least one from the root. Then $\mathcal{F}_j' \cap \mathcal{L}''$ consists of $q$ vertices $v^{(j)}_1, \dots, v^{(j)}_q$ which are all at distance one from the root. Note that $\mathcal{F}_j'$ itself consists of $q$ funnels with roots $v^{(j)}_1, \dots, v^{(j)}_q$. Since $F(\mu)$ is holomorphic at $\mu = \pm 1$, we can express the power series expansion of $F(\mu)$ via differentiation. Suppose $w$ is a vertex on the funnel of $\mathcal{F}_j'$ with root $v^{(j)}_i$. We see that 
    \begin{gather*} 
    [F(\mu)f](w) = -\frac{2}{\mu - q^{-1}\mu^{-1}} \Big( \frac{f(v_i^{(j)})}{(\sqrt{q} \mu)^{d(o_{\mathcal{F}_j}, w) - 1}} + \frac{\sum_{k \neq i} f(v_k^{(j)})}{(\sqrt{q} \mu)^{d(o_{\mathcal{F}_j}, w) + 1}} \Big).
    \end{gather*}
    The key observation is the following. If we know that $[F(\mu_0) f](w) = 0$ for all $w \in \mathcal{F}_j'$, then we must have that $f(v_i^{(j)}) = 0$ for all $i$. The former condition is equivalent to the vector $[f(v_1^{(j)}), \dots, f(v_q^{(j)}]^T$ being in the kernel of the matrix $(1 - (\sqrt{q} \mu)^{-2}) I + (\sqrt{q} \mu)^{-2} M$ where $M$ is the matrix of all ones. Note that the eigenvalues of $M$ are $q$ and $0$. In order for $(1 - (\sqrt{q} \mu)^{-2}) I + (\sqrt{q} \mu)^{-2} M$ to have non-trivial kernel, we must have that either $(1 - (\sqrt{q} \mu)^{-2}) = 0$ or $(1 - (\sqrt{q} \mu)^{-2}) = -q (\sqrt{q} \mu)^{-2}$. We can check that $\mu = \pm 1$ does not satisfy either of these equations, which implies that $[f(v_1^{(j)}), \dots, f(v_q^{(j)})]^T$ is identically zero. This would then imply that not only is $[F(\mu_0) f](w) = 0$, but $[F(\mu) f](w) \equiv 0$ as a function of $\mu$. The desired conclusion which we shall use later in the proof is the following: if $E$ is an operator whose image lies in $\mathcal{L}''$, and $F(\mu_0) E = 0$ with $\mu_0 = \pm 1$, then $F(\mu) E = 0$. Phrased another way, if we write
    \begin{gather*}
        F(\mu) = F_0 + F_1 (\mu - \mu_0) + 
        F_2 (\mu - \mu_0)^2 + \dots
    \end{gather*}
    with $\mu_0 = \pm 1$, then if $F_0 E = 0$ with $E$ as before, then $F_j E = 0$ for all $j$.

    We can perform a nearly identical analysis on $C(\mu)$. As before, suppose $f$ is supported on $\mathcal{L}''$. Let $\mathcal{C}_j'$ denote all vertices on the cusp $\mathcal{C}_j$ at distance at least one from the root. Then $\mathcal{C}_j' \cap \mathcal{L}''$ consists of a single vertex, $v^{(j)}$, which is distance one from the root. Suppose $w$ is some vertex on $\mathcal{C}_j$. Then
    \begin{gather*}
        -\frac{\mu - \mu^{-1}}{2} [C(\mu) f](w) = f(v^{(j)}) (\frac{\sqrt{q}}{\mu})^{d(w, v^{(j)})}.
    \end{gather*}
    We thus see that if $-\frac{\mu - \mu^{-1}}{2} [C(\mu_0) f] = 0$, then in fact $f(v^{(j)}) = 0$ for all $j$, and thus $-\frac{\mu - \mu^{-1}}{2} [C(\mu) f] \equiv 0$. We can write
    \begin{gather*}
        C(\mu) = \frac{C_{-1}}{\mu - \mu_0} + C_0 + C_1(\mu - \mu_0) + \dots
    \end{gather*}
    with $\mu_0 = \pm 1$. We thus see that if $E$ is an operator whose image lies in $\mathcal{L}''$ and such that $C_{-1} E = 0$, then $C_j E = 0$ for every $j$.

    We next recall from the proof of Proposition \ref{prop_bound_multiplicity} that if we write
    \begin{gather*}
      (I + K(\mu))^{-1} = \sum_{j = 1}^J \frac{B_{-j}}{(\mu - \mu_0)^j} + R_H(\mu), 
    \end{gather*}
    we must have that each $B_{-j}$ has image lying in $\mathcal{L}''$. 

    We now explicitly examine the resolvent
    \begin{gather*}
      R_{\mathcal{G}}(\mu) = \left(Q + \sum_{j = 0}^\infty F_j (\mu - \mu_0)^j + \sum_{j = -1}^\infty C_j (\mu - \mu_0)^j\right) \left(\sum_{j = -J}^\infty B_j (\mu - \mu_0)^j\right)
    \end{gather*}
    with $\mu_0 = \pm 1$. We know from Proposition \ref{prop_mu_pm1} that we have a pole of order at most 2 at $\mu_0$. We must therefore have that $C_{-1} B_k = 0$ for any $k \leq -2$, and hence $C_j B_k = 0$ for any $k \leq -2$ and any $j$. Similarly we must have $F_0 B_k = 0$ for any $k \leq -3$, and hence $F_j B_k = 0$ for any $k \leq -3$ and any $j$. Similarly, $Q B_k = 0$ for any $k \leq -3$. 

    We now consider the $(\mu - \mu_0)^{-2}$ term in the expansion, which would be
    \begin{align*}
        &Q B_{-2} + F_0 B_{-2} + F_1 B_{-3} + \dots + F_{J-2} B_{-J} + C_{-1} B_{-1} + C_0 B_{-2} + \dots + C_{J-2} B_{-J} \\
        = & Q B_{-2} + F_0 B_{-2} + C_{-1} B_{-1}. 
    \end{align*}
    We know that any non-zero element in the image must be an $\ell^2$-eigenfunction, and hence must be identically zero on each funnel and cusp by Proposition \ref{prop_eigenfunction_zero}. Therefore, we must have that $F_0 B_{-2} = C_{-1} B_{-1} = 0$, and thus $F_j B_{-2} = C_{\ell} B_{-1} = 0$ for all $j$ and $\ell$. Finally we examine the $(\mu - \mu_0)^{-1}$ term in the expansion, which is
    \begin{align*}
        & Q B_{-1} + F_0 B{-1} + F_1 B_{-2} + \dots + F_{J-1} B_{-J} + C_{-1} B_0 + C_0 B_{-1} + \dots + C_{J-1} B_{-J} \\
        = & Q B_{-1} + F_0 B_{-1} + C_{-1} B_0. 
    \end{align*}
    It is clear from this formula that anything in the image is purely outgoing of constant type. We thus see that all generalized resonant states must be eigenfunctions which are purely outgoing of constant type and have eigenvalue $z(\mu_0)$. To show the converse direction, we can use the exact same argument that we used in the second half of the proof of Theorem \ref{thm_resonant_state_good}. 

\end{proof}

\section{Examples and computations} \label{sec_examples}

\subsection{The resonance matrix}
We now discuss how we may use the results of the previous section to explicitly compute the resonances and generalized resonant states. Let $A_{\mathcal{L}}$ be the adjacency matrix of the induced subgraph corresponding to the compact core $\mathcal{L}$. Let $B(\mu)$ be the diagonal matrix with entry $(v, v)$ equal to $c_v \frac{\sqrt{q}}{\mu} + f_v \frac{1}{\sqrt{q} \mu}$ where
\begin{gather*}
    c_v = \sum_{\substack{\textnormal{edge $e$ containing $v$,} \\
    \textnormal{$e$ attached to a cusp}}} \frac{\mathcal{S}(v)}{\mathcal{S}(e)}, \hspace{10mm} f_v = \sum_{\substack{\textnormal{edge $e$ containing $v$,} \\
    \textnormal{$e$ attached to a funnel}}} \frac{\mathcal{S}(v)}{\mathcal{S}(e)}.
\end{gather*}
We define the \textit{resonance matrix} as
\begin{gather*}
  H(\mu) \coloneqq  \frac{1}{2 \sqrt{q}}(A_{\mathcal{L}} + B(\mu)) - z(\mu) I. 
\end{gather*}
The following is immediate:
\begin{proposition}
    The resonances are exactly those $\mu_0 \in \mathbb{C}^\times$ such that $\textnormal{det}(H(\mu_0)) = 0$. If $\mu_0$ is a resonance, then the associated resonant states correspond to elements in the kernel of $H(\mu_0)$. 
\end{proposition}

\begin{proof}
  First off we remark that Property (1) of Proposition~\ref{prop:cq4hwetoj6} and Proposition~\ref{prop_mu_pm1} imply that any time $\mu_0$ is a resonance, the set of resonant states associated to $\mu_0$ is non-trivial. By Theorems~\ref{thm_resonant_state_good} and \ref{thm_bad_points} (along with Lemma~\ref{lemma:cq4flfr8j6}), we know that resonant states are necessarily outgoing of constant type. However, Proposition~\ref{prop_always_outgoing} implies that outgoing eigenfunctions of constant type satisfy the outgoing condition on the entirety of every funnel and cusp. This means that for such outgoing eigenfunctions of constant type, the value on a given funnel or cusp is completely determined by the value at the vertex of $\mathcal{L}$ that the funnel/cusp is attached to. More specifically, if $f$ is an outgoing eigenfunction of constant type with eigenvalue $z(\mu_0)$, then given a vertex $v \in \mathcal{L}$, on every vertex $w$ on a cusp adjacent to $v$ we must have that $f(w) = \frac{\sqrt{q}}{\mu_0} f(v)$, and on every vertex $u$ on a funnel adjacent to $v$ we must have $f(u) = \frac{1}{\sqrt{q} \mu_0} f(v)$. Therefore, $f|_{\mathcal{L}}$ would lie in the kernel of $H(\mu_0)$. Conversely, suppose a non-zero function $f$ were in the kernel of $H(\mu_0)$. We can extend $f$ to a function on all of $\mathcal{G}$ by simply requiring that it satisfy the outgoing condition on the entirety of every funnel and cusp, and it is clear that such a function would be an outgoing eigenfunction of constant type.
\end{proof}

\begin{corollary} \label{cor_num_resonances}
    The number of resonances, not counted with multiplicity, is at most $2|\mathcal{L}|$. 
\end{corollary}
\begin{proof}
    The determinant of $H(\mu)$ is a polynomial in $\mu$ and $\mu^{-1}$, and thus can be written in the form $\mu^k f(\mu)$ for some $k \in \mathbb{Z}$ and some polynomial $f$. Since each entry in $H(\mu)$ can be written as $\mu^\ell g(\mu)$ where $g$ is a polynomial of degree at most 2, we conclude that $f$ has degree at most $2|\mathcal{L}|$. 
\end{proof}

\subsection{The tree} \label{sec_example_tree}
We already computed in Section \ref{sec_resolvent_tree} that the resonances of the tree correspond to $\pm \frac{1}{\sqrt{q}}$. Furthermore, the explicit formula for the resolvent in \eqref{eq:cq4enzzhjz} tells us that the generalized resonant states for $\frac{1}{\sqrt{q}}$ are constant functions, and for $-\frac{1}{\sqrt{q}}$ they are the functions of the form $(-1)^{d(o, v)} c$ for some $c$ and choice of origin $o$. We now quickly recompute this using the results of the previous section. We can think of the tree as a compact core consisting of a single vertex $o$ together with $(q+1)$ attached funnels. Therefore
\begin{gather*}
  H(\mu) = \frac{1}{2 \sqrt{q}} \left( \frac{q+1}{\sqrt{q} \mu}\right) - \frac{\mu + \mu^{-1}}{2} = -\frac{(\mu - q^{-1} \mu^{-1})}{2}. 
\end{gather*}
This clearly has zeros at $\mu = \pm \frac{1}{\sqrt{q}}$. When $\mu = \frac{1}{\sqrt{q}}$, assuming the value at $o$ is 1, then the value on each funnel is $(\frac{1}{\sqrt{q} \frac{1}{\sqrt{q}}})^k = 1$, and we can do the same analysis for $\mu = -\frac{1}{\sqrt{q}}$ and rederive the corresponding resonant state.

\subsection{Parabolic cylinder}

This consists of a single vertex $o$ with one attached cusp and $q$ attached funnels. The resonance matrix is thus
\begin{gather*}
  H(\mu) = \frac{1}{2 \sqrt{q}} \left(\frac{q}{\sqrt{q} \mu} + \frac{\sqrt{q}}{\mu}\right) - \frac{\mu + \mu^{-1}}{2} = \frac{\mu - \mu^{-1}}{2}. 
\end{gather*}
Thus the resonances are $\pm 1$. The resonant state for $\mu_0 = 1$ is 1 at $o$, $q^{k/2}$ at height $k$ on the cusp, and $q^{-k/2}$ at height $k$ on the funnels. Similarly, the resonant state for $\mu_0 = -1$ is 1 at $o$, $(-1)^k q^{k/2}$ at height $k$ on the cusp, and $(-1)^k q^{-k/2}$ at height $k$ on the funnels.

\subsection{Hyperbolic cylinder} \label{sec_hyperbolic_cylinder}

Consider the $(q+1)$-regular hyperbolic cylinder with central geodesic of length $N$. Then the resonance matrix is the $N \times N$ circulant matrix whose first row is
\begin{gather*}
  \left[\frac{q-1}{2 q \mu} - \frac{\mu + \mu^{-1}}{2}, \frac{1}{2 \sqrt{q}}, 0, \dots, 0, \frac{1}{2 \sqrt{q}}\right].
\end{gather*}
Standard facts about circulant matrices imply that the determinant of this matrix is
\begin{gather*}
  \prod_{j = 0}^{N-1} \left(\frac{q-1}{2 q \mu} - \frac{\mu + \mu^{-1}}{2} + \frac{1}{\sqrt{q}} \cos\left(\frac{2 \pi j}{N}\right)\right). 
\end{gather*}
Thus resonances occur whenever one can find $\mu$ satisfying
\begin{gather*}
  \frac{q-1}{2 q \mu} - \frac{\mu + \mu^{-1}}{2} = -\frac{1}{\sqrt{q}} \cos\left(\frac{2 \pi j}{N}\right)
\end{gather*}
Solving directly shows that solutions correspond to $\mu = \frac{1}{\sqrt{q}} e^{2 \pi i j/N}$. As $N \to \infty$ these equidistribute on the circle of radius $\frac{1}{\sqrt{q}}$. 

\subsection{Funnels and embedded eigenvalues: a counterexample} \label{sec_counterexample}

In the context of geometrically finite hyperbolic surfaces, it is a result of Patterson \cite{Pat75} for hyperbolic surfaces and Lax-Phillips \cite{lax_phillips_82} for general hyperbolic manifolds that, in the presence of funnels, there are no embedded $L^2$-eigenfunctions (i.e.\@ $L^2$-eigenfunctions with eigenvalue $\lambda \geq \frac{1}{4}$ in the hyperbolic surfaces case); see also \cite{WW23} for analogous results in the higher rank setting. The analogous statement in the graph setting, namely that a geometrically finite graph with at least one funnel has no $\ell^2$-eigenfunctions of the adjacency operator with eigenvalue in $[-2 \sqrt{q}, 2 \sqrt{q}]$, is not true as the example below shows.

Consider the graph in Figure \ref{fig_funnel_embedded}, and the corresponding function specified in the figure. It is clear that this eigenfunction is in $\ell^2(\mathcal{V})$ and the eigenvalue is $-2$, which lies in the tempered range $[-2 \sqrt{2}, 2 \sqrt{2}]$.

\begin{figure}[ht]
    \centering
    \includegraphics[width=0.3\textwidth]{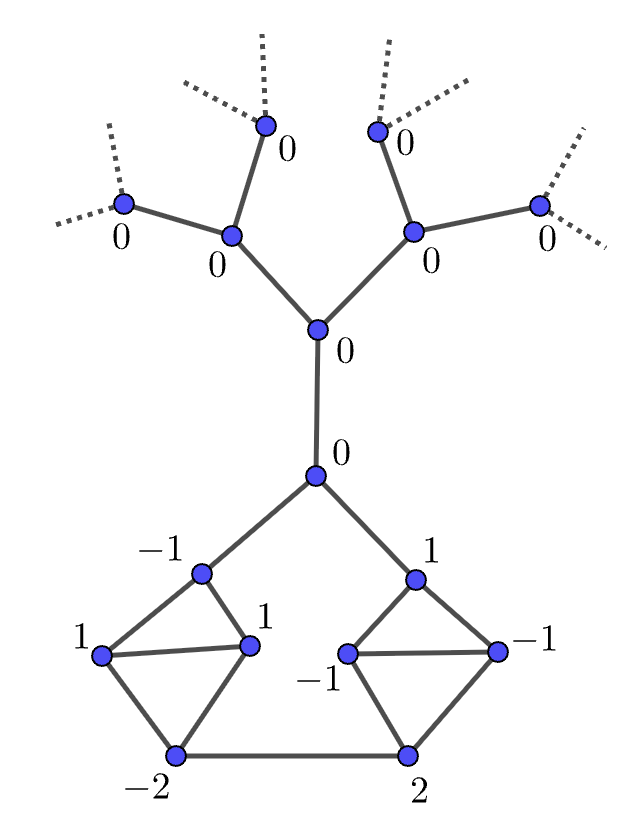}
    

    \caption{An example of a $3$-regular graph with a funnel which has an embedded eigenvalue.}
    \label{fig_funnel_embedded}
\end{figure}

\subsection{Modular curve} \label{sec_modular_curve}
Consider the modular curve from Section \ref{sec_arithmetic_lattice}. It consists of a single vertex attached $q+1$ times to a cusp. Therefore its resonance matrix is
\begin{gather*}
    H(\mu) = \frac{1}{2 \sqrt{q}}\Big(\frac{(q+1) \sqrt{q}}{\mu}\Big) - \frac{\mu + \mu^{-1}}{2}. 
\end{gather*}
This has zeros at $\mu = \pm \sqrt{q}$. The resonant states for $\mu = \sqrt{q}$ are the constant functions, and for $\mu = -\sqrt{q}$ are the functions which are of the form $c (-1)^k$ where $k$ is the height on the cusp. 

\subsection{Elliptic curves} \label{sec_elliptic_curves}
    
Takahashi \cite{takahashi} (specifically Theorem 5 therein) describes explicitly how to compute the arithmetic graph of groups described in Section \ref{sec_arithmetic_lattice} corresponding to an elliptic curve over a finite field. We discuss two examples.

The first example is taken from 2.4.4 of \cite{serre_trees}, together with Theorem 5 of \cite{takahashi}. Consider the (affine) elliptic curve 
    \begin{gather}
        y^2 + y = x^3 + x + 1 \label{eqn_first_elliptic}
    \end{gather}
    over $\mathbb{F}_2$. The corresponding graph with stabilizer functions is given in Figure \ref{fig_serre}, corresponding to choosing the point $P$ to be the unique point at infinity of the projectivization of this elliptic curve.

\begin{figure}[htbp]
    \centering
    \begin{minipage}{0.48\textwidth}
        \centering
        \includegraphics[width=\textwidth]{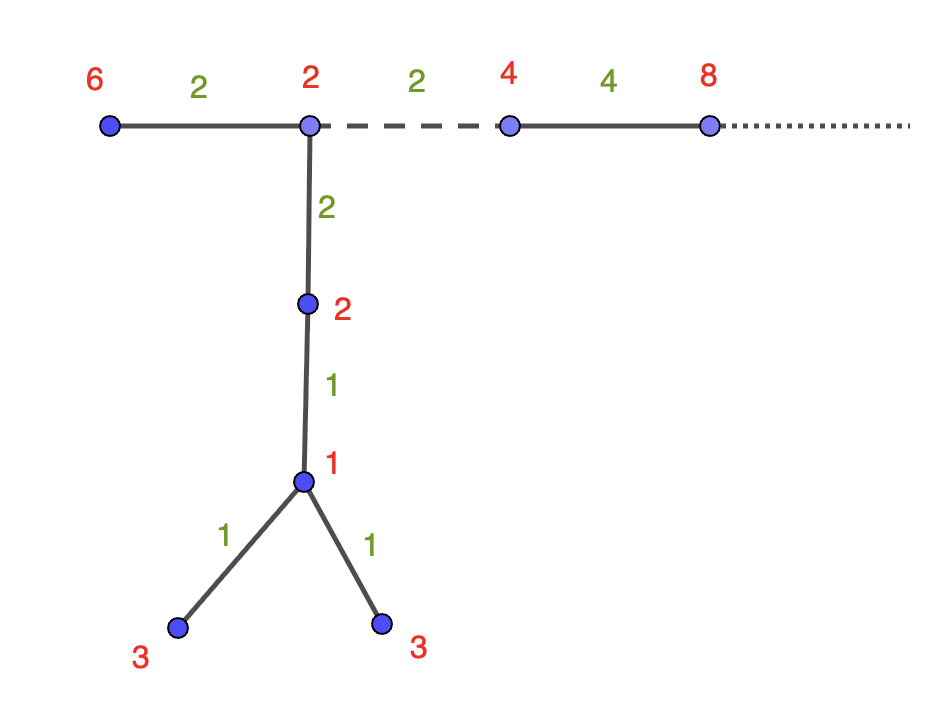}
    \end{minipage}
    \hfill
    \begin{minipage}{0.48\textwidth}
        \centering
        \includegraphics[width=\textwidth]{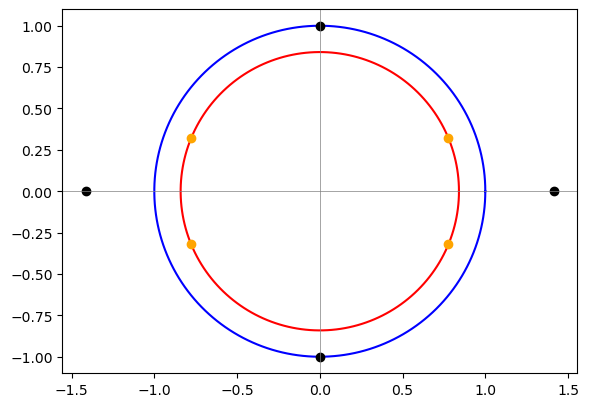}
    \end{minipage}
    \caption{The left shows the stabilizer size data for the quotient graph of groups corresponding to \eqref{eqn_first_elliptic}. The right shows the resonance plot, with black dots corresponding to $\ell^2$-eigenvalues, and orange dots corresponding to zeros of the Hasse-Weil zeta function.}
    \label{fig_serre} 
\end{figure}
    
    We find that the determinant of the resonance matrix is:
    \begin{gather*}
        (\mu^2 + 1)^2 (\mu^2 - 2) (2 \mu^4 - 2 \mu^2 + 1).
    \end{gather*}
    The resonance at $\mu = \sqrt{q}$ corresponds to the constant function, and the resonance at $\mu = -\sqrt{q}$ corresponds to the function of the form $(-1)^d$ where $d = 0$ or $1$ based on the bipartite coloring of the graph. These are clearly $\ell^2$-eigenvalues.

    We see that we also have resonances at $\mu = \pm i$. In computing the corresponding resonant states, we find that they are in fact $\ell^2$-eigenfunctions (both $\mu = i$ and $\mu = -i$ have two associated $\ell^2$-eigenfunctions).

    Finally, we are left with the polynomial $2 \mu^4 - 2 \mu^2 + 1$. This in fact corresponds to the numerator of the Hasse-Weil zeta function associated to the elliptic curve (or equivalently, the Dedekind zeta function attached to the function field of the elliptic curve). The Hasse-Weil zeta function of a (projective) curve over a finite field $\mathbb{F}_q$ is defined as
    \begin{gather*}
        Z_X(T) := \exp\Big( \sum_{r = 1}^\infty \frac{N_r}{r} T^r \Big),
    \end{gather*}
    where $N_r$ is the number of points on $X$ over the field $\mathbb{F}_{q^r}$. Weil \cite{weil} proved several remarkable facts about such zeta functions. One is that
    \begin{gather*}
        Z_X(T) = \frac{P(T)}{(1-T) (1-q T)},
    \end{gather*}
    where $P(T)$ is a polynomial of degree $2 g$ where $g$ is the genus of $X$. Second is that all zeros of the polynomial $P(T)$ lie on the circle of radius $\frac{1}{\sqrt{q}}$; this is the analogue of the Riemann hypothesis in this setting (but is indeed known to be true!).

    We can compute the Hasse-Weil zeta function of the projective closure of our elliptic curve, and we get
    \begin{gather*}
        Z_X(T) = \frac{2 T^2 - 2 T + 1}{(1-T)(1-2T)}.
    \end{gather*}
    Letting $P(T)$ denote the numerator, we find that $P(\mu^2)$ is exactly equal to the unaccounted for polynomial in the determinant of the resonance matrix. The Riemann hypothesis (Weil's theorem) implies that all the associated resonances lie on the circle of radius $\frac{1}{2^{\frac{1}{4}}}$. 

    We now compute another example. The second example is based on the example accompanying Figure 5 in \cite{takahashi}. Consider the elliptic curve 
    \begin{gather}
        y^2 = x^3 + x + 1 \label{eqn_elliptice_curve_formula2}
    \end{gather} 
    over $\mathbb{F}_3$. The corresponding graph with stabilizer functions is given in Figure \ref{fig_takahashi_example}. 

\begin{figure}[htbp]
    \centering
    \begin{minipage}{0.48\textwidth}
        \centering
        \includegraphics[width=\textwidth]{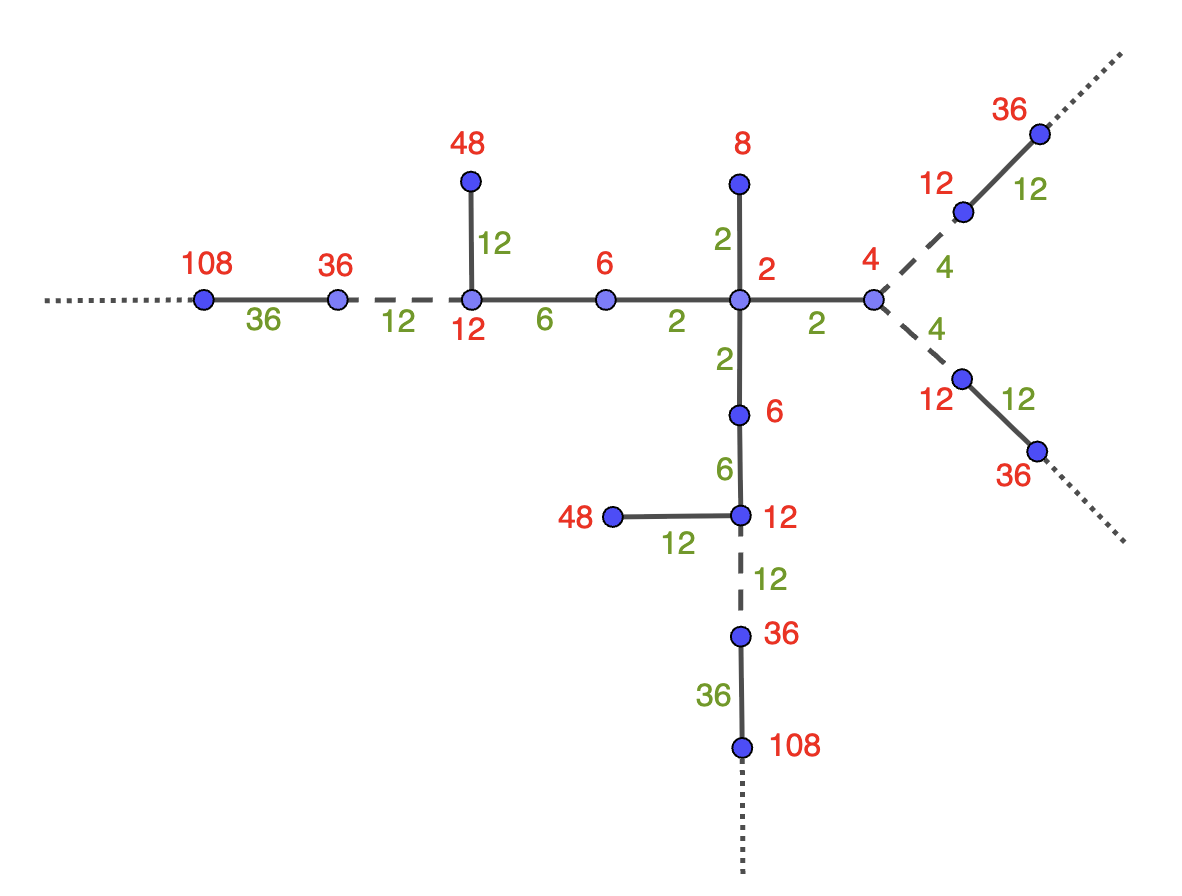}
    \end{minipage}
    \hfill
    \begin{minipage}{0.48\textwidth}
        \centering
        \includegraphics[width=\textwidth]{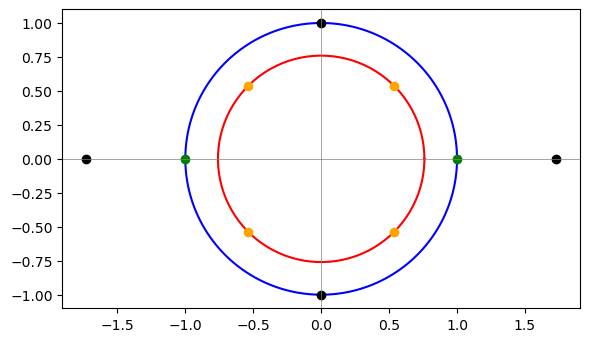}
    \end{minipage}
    \caption{The left shows stabilizer size data for the quotient graph of groups corresponding to \eqref{eqn_elliptice_curve_formula2}. The right shows the resonance plot. In addition to the $\ell^2$-eigenvalues (in black), and the resonances corresponding to the zeros of the Hasse-Weil zeta function (in orange), we have additional ``topological'' resonances at $\mu = \pm 1$ (in green).}
    \label{fig_takahashi_example}
\end{figure}

    We find that the resonances correspond to zeros of the polynomial
    \begin{gather*}
        (\mu^2 + 1)^2 (\mu - 1)^2 (\mu + 1)^2 (\mu^2 - 3) (3 \mu^4 + 1).
    \end{gather*}
    We again get the ``trivial'' resonances at $\mu = \pm \sqrt{3}$. We also compute that we get $\ell^2$-eigenvalues at $\mu = \pm i$ (each of multiplicity 2). We also have resonances at $\mu = \pm 1$ (each of multiplicity 2). However, we compute that they are in fact not $\ell^2$-eigenfunctions.

    On the other hand, the Hasse-Weil zeta function of the projective closure of this curve is
    \begin{gather*}
        \frac{3 T^2 + 1}{(1-T)(1-3T)}.
    \end{gather*}
    Again, setting $T = \mu^2$, the numerator of the Hasse-Weil zeta function shows up in the polynomial governing the resonances, and all the associated resonances lie on the circle of radius $\frac{1}{3^{\frac{1}{4}}}$.

\subsection{Arithmetic lattices} \label{sec_arithmetic2}
The resonances of the modular curve $\textnormal{SL}(2, \mathbb{Z}) \backslash \mathbb{H}$ come in three types (see, e.g., \cite{iwaniec}):
\begin{enumerate}
    \item A trivial resonance at $s = 1$ (corresponding to the constant function).
    \item Resonances corresponding to the other $L^2$-eigenvalues. By the Selberg conjecture, these are expected to all lie on the line $\textnormal{Re}(s) = \frac{1}{2}$.
    \item Resonances corresponding to poles of the Eisenstein series/scattering matrix, which is given explicitly by
    \begin{gather*}
        \phi(s) = \frac{\xi(2s - 1)}{\xi(2s)},
    \end{gather*}
    where $\xi$ is the completed Riemann zeta function. By the Riemann hypothesis, this is expected to have all of its poles lying on $\textnormal{Re}(s) = \frac{1}{4}$.
\end{enumerate}

More generally, let $F$ be a number field, and let $\mathcal{O}$ be its ring of integers. Let $M_F^\infty$ denote the set of infinite places of $F$, and for each $v \in M_F^\infty$ let $F_v$ denote the associated local field (i.e. $F_v \simeq \mathbb{R}$ or $\mathbb{C}$). Then the group $\textnormal{SL}(2, \mathcal{O})$ is a non-uniform lattice in $\prod \limits_{v \in M^\infty_F} \textnormal{SL}(2, F_v)$. We thus obtain the locally symmetric space
\begin{gather}
    \textnormal{SL}(2, \mathcal{O}) \backslash \prod \limits_{v \in M^\infty_F} \textnormal{SL}(2, F_v) / \prod \limits_{v \in M^\infty_F} K_v, \label{eqn_locally_symmetric_space}
\end{gather}
where $K_v$ is the maximal compact subgroup of $\textnormal{SL}(2, F_v)$. This is a quotient of the symmetric space $(\mathbb{H}^2)^r \times (\mathbb{H}^3)^{c}$ where $r$ is the number of real embeddings of $F$, and $2c$ is the number of complex embeddings (which come in $c$ complex conjugate pairs). The cusps of this locally symmetric space correspond to elements in the class group of $F$. Spectral theory in this setting is more complicated in part due to the fact that one must consider the full algebra generated by the Laplacians in each factor. It is better to pass to the representation-theoretic perspective; see e.g. \cite{gelbart}. Without going into too much detail, resonances in this context, properly defined and parametrized, correspond to joint eigenvalues of joint $L^2$-eigenfunctions of all these partial Laplacians. By the generalized Ramanujan conjecture such eigenvalues are expected to be either trivial or tempered. Additionally we have resonances corresponding to poles of the meromorphic continuation of the constant term of the Eisenstein series/scattering determinant. The Langlands-Shahidi method \cite{shahidi} relates the Eisenstein series to the Dedekind zeta function of the number field, and the generalized Riemann hypothesis predicts that its zeros (which give poles of the scattering determinant) all lie on the critical line. Finally there are ``topological'' resonances at the special value ($s = \frac{1}{2}$ for hyperbolic surfaces) whose multiplicity is expressed in terms of the class number and the trace of the scattering matrix; for the modular curve this value happens to be zero.

We now wish to describe how we are observing the same phenomenon in the examples from Section \ref{sec_elliptic_curves}. In the set-up described in Section \ref{sec_arithmetic_lattice}, we start with a smooth projective algebraic curve $C$ over a finite field $\mathbb{F}_q$ and a point $P$ on $C$ (which we think of as a ``point at $\infty$''). The analogue of \eqref{eqn_locally_symmetric_space} is \eqref{eqn_function_field_quotient}. Thus in analogy with the number field setting, we should expect the resonances to be of the following form:
\begin{enumerate}
    \item Trivial resonances at $\pm \sqrt{q}$ corresponding to trivial $\ell^2$-eigenfunctions.
    \item Resonances corresponding to the other $\ell^2$-eigenvalues which all lie on the unit circle (the analogue of the Selberg conjecture) as a consequence of Drinfeld's proof of the Langlands correspondence for $\textnormal{GL}(2, \mathbb{F}_q(C))$ \cite{drinfeld2, drinfeld1}. 
    \item Resonances on the circle of radius $\frac{1}{q^{\frac{1}{4}}}$ corresponding to square roots of zeros of the Hasse-Weil zeta function of $C$. The connection between resonances and the Hasse-Weil zeta function (the function field analogue of the Dedekind zeta function) should come from the connection of them both to Eisenstein series (via the Lax-Phillips scattering theory and the Langlands-Shahidi method, respectively). The fact that they all lie on this circle is a consequence of Weil's proof of the Riemann hypothesis in this setting \cite{weil}.
    \item Resonances at $\mu = \pm 1$ whose multiplicity is related to the number of cusps and the functional equation of the zeta function/trace of the scattering matrix.
\end{enumerate}
We plan to make this correspondence precise in follow-up work connecting resonances with the Lax-Phillips scattering theory in this setting.

\subsection{Large random cusped graphs and the Phillips-Sarnak conjecture} \label{sec_large_random_cusps}

We begin by showing some numerics of large random geometrically finite graphs with cusps. To generate the graphs below in Figure \ref{fig_numerics_cusps}, which were generated using Maple \cite{maple}, we started with a random $7$-regular graph on 500 vertices (thus $q = 6$). We then randomly chose $c$ edges which we removed and replaced with a cusp at each endpoint of the removed edge. We then computed those $\mu$ such that the associated resonance matrix becomes non-invertible. On the graph we show the unit circle (in blue), the circle of radius $\frac{1}{\sqrt{q}}$ (in green), the points $\pm \sqrt{q}$ (in black) and the points $\pm \frac{1}{\sqrt{q}}$ (in green), and we show a histogram plot of the distribution of magnitudes of the resonances. We remark that, by the Kesten-McKay law, a large random finite regular graph will have nearly all of its resonances (which correspond to eigenvalues) on the unit circle (which, with respect to adjacency operator eigenvalue, corresponds to the Ramanujan range $[-2\sqrt{q}, 2 \sqrt{q}]$).

Visually, we observe that, upon adding cusps, the resonances appear to either get pushed onto the real axis, or close to the imaginary axis (but away from 0). However if we keep adding more cusps, then all of the resonances are on the real line. We do not have a good theoretic explanation for this phenomenon. As we continue adding more and more cusps, the resonances are pushed closer and closer to $\pm \sqrt{q}$. This makes sense because if we replace all edges by cusps, then we get a bunch of disconnected graphs consisting of one vertex attached to $q+1$ cusps. From the perspective of spectral theory, this is essentially the same thing as the modular curve (Section \ref{sec_modular_curve}) which we computed to have resonances only at $\pm \sqrt{q}$.

\begin{figure}[p] 
    \centering
    
    \begin{subfigure}[b]{0.48\textwidth}
        \centering
        \includegraphics[width=\textwidth, keepaspectratio=true]{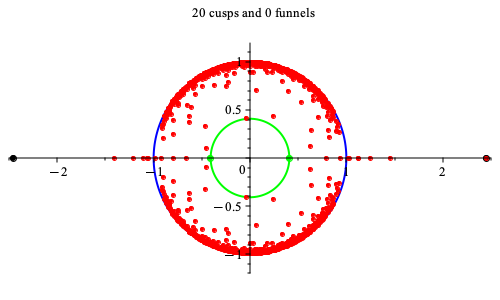}
    \end{subfigure}
    \hfill
    \begin{subfigure}[b]{0.48\textwidth}
        \centering
        \includegraphics[width=\textwidth, keepaspectratio=true]{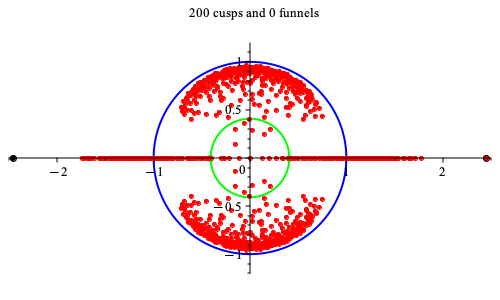}
        
    \end{subfigure}
    \vfill 

    \begin{subfigure}[b]{0.48\textwidth}
        \centering
        \includegraphics[width=\textwidth, height=0.17\textheight]{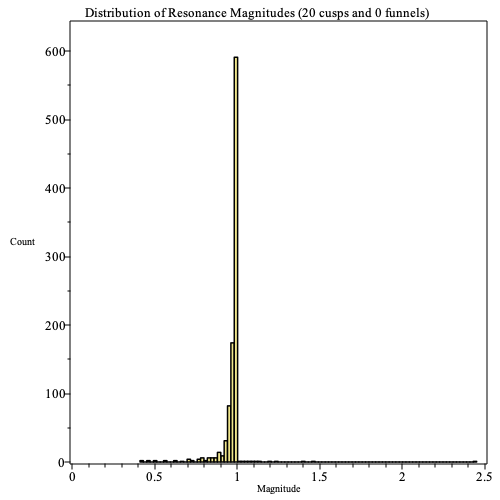}
        
    \end{subfigure}
    \hfill
    \begin{subfigure}[b]{0.48\textwidth}
        \centering
        \includegraphics[width=\textwidth, height=0.17\textheight]{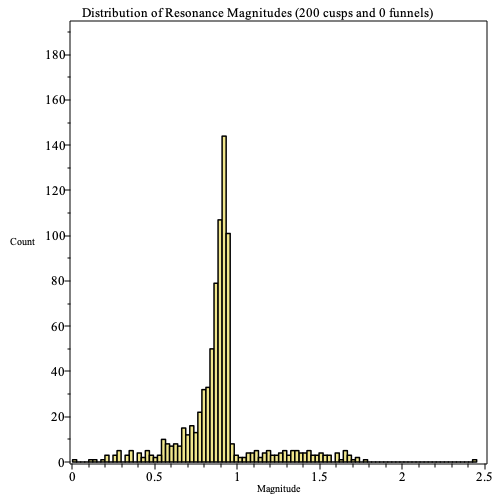}
        
    \end{subfigure}
    \vfill

    \begin{subfigure}[b]{0.48\textwidth}
        \centering
        \includegraphics[width=\textwidth, keepaspectratio=true]{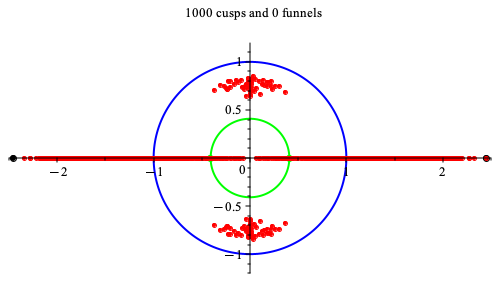}
        
    \end{subfigure}
    \hfill
    \begin{subfigure}[b]{0.48\textwidth}
        \centering
        \includegraphics[width=\textwidth, keepaspectratio=true]{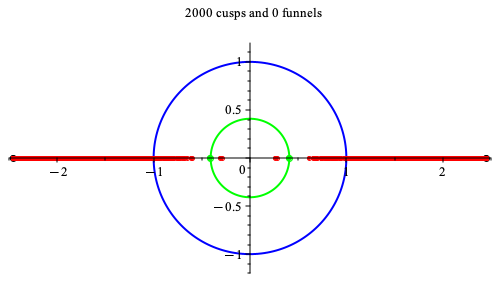}
        
    \end{subfigure}
    \vfill

    \begin{subfigure}[b]{0.48\textwidth}
        \centering
        \includegraphics[width=\textwidth, height=0.17\textheight]{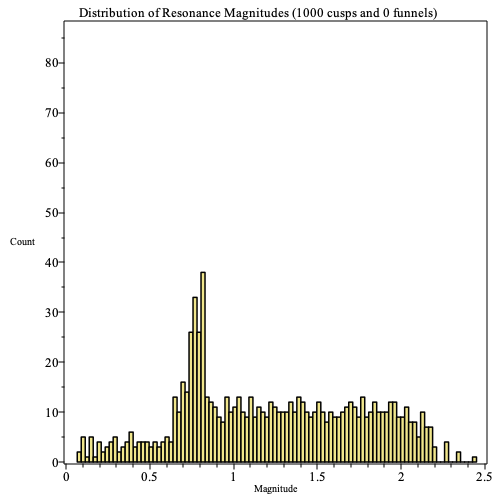}
        
    \end{subfigure}
    \hfill
    \begin{subfigure}[b]{0.48\textwidth}
        \centering
        \includegraphics[width=\textwidth, height=0.17\textheight]{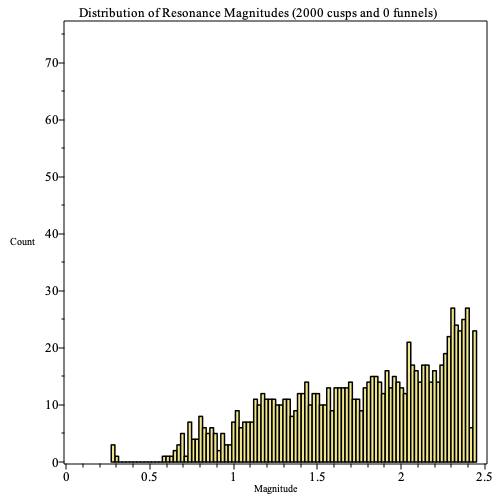}
        
    \end{subfigure}
    \vfill

    \begin{subfigure}[b]{0.48\textwidth}
        \centering
        \includegraphics[width=\textwidth]{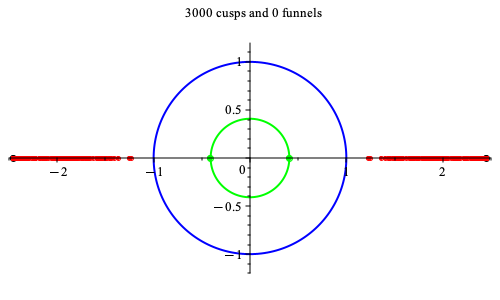}
        
    \end{subfigure}
    \hfill
    \begin{subfigure}[b]{0.48\textwidth}
        \centering
        \includegraphics[width=\textwidth]{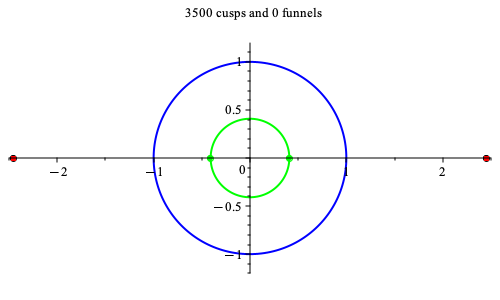}
        
    \end{subfigure}
    \vfill

    \begin{subfigure}[b]{0.48\textwidth}
        \centering
        \includegraphics[width=\textwidth, height=0.17\textheight]{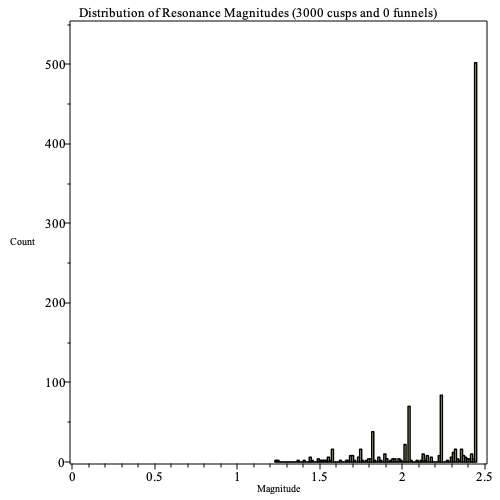}
        
    \end{subfigure}
    \hfill
    \begin{subfigure}[b]{0.48\textwidth}
        \centering
        \includegraphics[width=\textwidth, height=0.17\textheight]{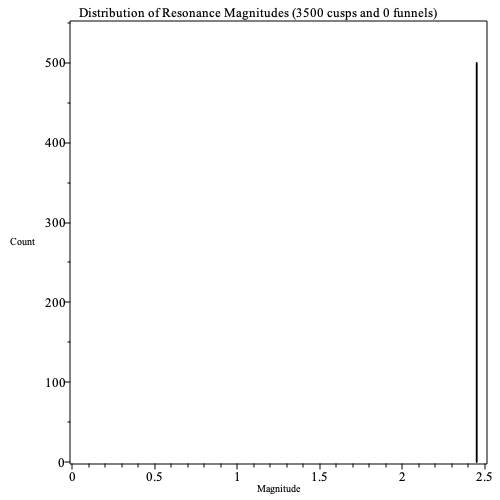}
        
    \end{subfigure}
    \vfill

    \caption{Distribution of resonances on random 7-regular graph with 500 vertices and growing number of cusps.}
    \label{fig_numerics_cusps}
\end{figure}

In the setting of arithmetic cusped hyperbolic surfaces, it is known that there is an abundance of embedded $L^2$-eigenvalues, i.e. eigenvalues above $\frac{1}{4}$ with associated eigenfunction lying in $L^2$ \cite{selberg}. In terms of the parametrization $\lambda = s(1-s)$, this corresponds to resonances lying on the line $\textnormal{Re}(s) = \frac{1}{2}$. Phillips-Sarnak \cite{phillips_sarnak1, phillips_sarnak2, phillips_sarnak3} showed that under generic perturbations, any given embedded $L^2$-eigenvalue ``dissolves'' into a resonance which is not an $L^2$-eigenvalue, i.e. it moves off of the line $\textnormal{Re}(s) = \frac{1}{2}$. Motivated by this result, they conjectured that for a generic hyperbolic surface with cusps, there are only finitely many embedded $L^2$-eigenvalues. 

Returning to the setting of cusped geometrically finite graphs, Proposition \ref{prop_eigenfunction_zero} tells us that we have an $\ell^2$-eigenvalue on the unit circle if and only if we have an eigenfunction of $\frac{A}{2 \sqrt{q}}$ with eigenvalue in $[-1, 1]$ that vanishes identically on all cusps.
In terms of the resonance matrix, this latter condition is easy to check: we would need that if $\mu_0$ is on the unit circle, then elements in the kernel of $H(\mu_0)$ are zero at the vertices in the compact core attached to cusps.
On the other hand, at least heuristically, the property of a resonant state being identically zero on all of the vertices attached to cusps should be expected to be highly non-generic. Any given geometrically finite cusped graph is going to have only finitely many embedded $\ell^2$-eigenvalues. However, we should expect that for generic cusped graphs, these eigenvalues should make up a vanishingly small fraction of the total number of resonances. We may thus pose the following version of the Phillips-Sarnak conjecture.

Let $\mathcal{G}$ be a geometrically finite cusped graph, and let $\textnormal{Res}_{\textnormal{emb}}(\mathcal{G})$ denote the set (with multiplicities) of $\ell^2$-eigenvalues of $\frac{A}{2 \sqrt{q}}$ acting on $\mathcal{G}$ lying in $[-1, 1]$. Let $\textnormal{Res}(\mathcal{G})$ denote the set (with multiplicities) of all resonances. Let $\mathcal{M}_{n, c}$ denote the (finite) set of all $(q+1)$-regular geometrically finite graphs with compact core consisting of $n$ vertices and with $c$ cusps. Let $\mathcal{G}_{n, c}$ denote an element sampled uniformly at random from $\mathcal{M}_{n, c}$.

\begin{conjecture} [Phillips-Sarnak conjecture analogue]
    Suppose $n_j$ and $c_j$ sequences such that $n_j \to \infty$ and $c_j \geq 1$. Then
    \begin{gather*}
        \mathbb{E} \Big( \frac{|\textnormal{Res}_{\textnormal{emb}}(\mathcal{G}_{n_j, c_j})|}{|\textnormal{Res}(\mathcal{G}_{n_j, c_j})|} \Big) \to 0
    \end{gather*}
    as $j \to \infty$.
\end{conjecture}

\subsection{Large random funneled graphs and the Jakobson-Naud conjecture} \label{sec_large_random_funnel}

Let $X$ be a geometrically finite hyperbolic surface with funnels but no cusps, i.e. convex cocompact. Let $\delta$ be the rightmost resonance; note that $\delta$ necessarily lies on the real line. Define
\begin{gather*}
    G(X) := \inf \Big\{ \sigma \in \mathbb{R} : \# \{\textnormal{resonances $s$ with } \textnormal{Re}(s) \geq \sigma \} < \infty \Big\}.
\end{gather*}
In \cite{jakobson_naud} Jakobson and Naud conjecture that $G(X) = \frac{\delta}{2}$. In words, we can interpret this conjecture as saying that if the biggest resonance is $\delta$, then ``almost all'' other resonances lie to the left of the line $\textnormal{Re}(s) = \frac{\delta}{2}$ (and infinitely many resonances lie close to this line). This conjecture remains open.

We wish to formulate an analogue of this conjecture for geometrically finite graphs with funnels but no cusps, i.e. convex cocompact (see Proposition \ref{prop_convex_cocompact}). Since there are only finitely many resonances, it does not make sense to study the analogous quantity to $G(X)$. It should instead be a statement about sequences of convex cocompact graphs. Referring back to Table \ref{table_hyperbolic_regular}, the formulation of the Jakobson-Naud conjecture is in terms of the ``$s$-parametrization'', but our conventions for parametrizing resonances in the graph setting is more analogous to the ``$r$-parametrization''. Switching to the analogue of the ``$s$-parametrization'' amounts to scaling all resonances by a factor of $\sqrt{q}$. Thus if $R^{\textnormal{max}}$ is the biggest resonance, then the Jakobson-Naud conjecture in this setting would be concerned with resonances $\mu$ such that $\sqrt{q} |\mu| > \sqrt{\sqrt{q} R^{\textnormal{max}}}$, i.e. $|\mu| > q^{-\frac{1}{4}} \sqrt{R^{\textnormal{max}}}$.

Suppose $\mathcal{G}_n$ is a sequence of $(q+1)$-regular convex cocompact graphs. Let $\mathcal{L}_n$ denote the compact core of $\mathcal{G}_n$. We say that $\mathcal{G}_n$ \textit{Benjamini-Schramm converges} to the $(q+1)$-regular tree if, for every $R > 0$, the fraction of vertices in $\mathcal{L}_n$ whose local injectivity radius is less than $R$ goes to zero as $n \to \infty$. Let $R_n^{\textnormal{max}}$ denote the greatest modulus of any element in $\textnormal{Res}(\mathcal{G}_n)$. 

\begin{conjecture}[Jakobson-Naud conjecture analogue, deterministic version] \label{JN_conj_weak}
    Suppose $\mathcal{G}_n$ is a sequence of $(q+1)$-regular convex cocompact connected graphs which Benjamini-Schramm converges to the $(q+1)$-regular tree. Then, 
    \begin{gather*}
        \lim_{n \to \infty} \frac{\# \{ \mu \in \textnormal{Res}(\mathcal{G}_n) : |\mu| > q^{-\frac{1}{4}} \sqrt{R_n^{\textnormal{max}}} \}}{\#\textnormal{Res}(\mathcal{G}_n)} = 0,
    \end{gather*}
    and for every $\varepsilon > 0$, 
    \begin{gather*}
        \liminf_{n \to \infty} \frac{\# \{ \mu \in \textnormal{Res}(\mathcal{G}_n) : |\mu| > q^{-\frac{1}{4}} \sqrt{R_n^{\textnormal{max}}} - \varepsilon \}}{\#\textnormal{Res}(\mathcal{G}_n)} > 0.
    \end{gather*}
\end{conjecture}

\begin{conjecture}[Jakobson-Naud conjecture analogue, random version] \label{JN_conj_strong}
    Suppose $\mathcal{G}_n$ is a sequence of random $(q+1)$-regular convex cocompact graphs, with $\mathcal{G}_n$ selected uniformly at random from the finite set of $(q+1)$-regular convex cocompact graphs with compact core of size $a_n$ and with $b_n$ funnels. Suppose $a_n \to \infty$. Then for every $\varepsilon > 0$, we have
    \begin{gather*}
        \# \{ \mu \in \textnormal{Res}(\mathcal{G}_n) : |\mu| > q^{-\frac{1}{4}} \sqrt{R_n^{\textnormal{max}}} + \varepsilon \} = O(1)
    \end{gather*}
    with high probability as $n \to \infty$.
\end{conjecture}

Note that Conjecture \ref{JN_conj_weak} would imply that under Benjamini-Schramm convergence, we have that for every $\varepsilon > 0$,
\begin{gather*}
    \# \{ \mu \in \textnormal{Res}(\mathcal{G}_n) : |\mu| > q^{-\frac{1}{4}} \sqrt{R_n^{\textnormal{max}}} + \varepsilon \} = o(\#\textnormal{Res}(\mathcal{G}_n)),
\end{gather*}
whereas Conjecture \ref{JN_conj_strong} says that we can upgrade the $o(\#\textnormal{Res}(\mathcal{G}_n))$ to $O(1)$ with high probability if our sequence is chosen at random.

These conjectures are consistent with what is already known in certain extremal cases. For example, convex cocompact graphs without funnels are simply finite graphs. In that case, Conjecture \ref{JN_conj_weak} follows from the Kesten-McKay law. More precisely, for sequences of finite regular graphs we know that $R_n^{\textnormal{max}} = \sqrt{q}$, and the Kesten-McKay law says that under the assumption of Benjamini-Schramm convergence, the distribution of resonances (i.e. eigenvalues) converges weakly to an explicit measure supported entirely on the unit circle (in this case $q^{-\frac{1}{4}} \sqrt{R_n^{\textnormal{max}}} = 1$). Conjecture \ref{JN_conj_strong} in this case follows from Friedman's proof of Alon's conjecture \cite{friedman}, namely that random regular graphs have optimal spectral gap with high probability.

At the other extreme, the only connected convex cocompact graph having as many funnels as possible relative to the size of the compact core is the infinite tree. In that case all of the resonances are $\pm \frac{1}{\sqrt{q}}$, so $R_n^{\textnormal{max}} = \frac{1}{\sqrt{q}}$, and $q^{-\frac{1}{4}} \sqrt{R_n^{\textnormal{max}}} = \frac{1}{\sqrt{q}}$. The normalized distribution of resonances converges to (and is identically equal to) $\frac{1}{2}(\delta_{\frac{1}{\sqrt{q}}} + \delta_{-\frac{1}{\sqrt{q}}})$. Thus Conjecture \ref{JN_conj_weak} is trivially true in this case.

A slightly more interesting example is furnished by sequences of larger and larger hyperbolic cylinders as studied in Section \ref{sec_hyperbolic_cylinder}. In that case it was found that all resonances lie on the circle of radius $\frac{1}{\sqrt{q}}$, and in fact the distribution of resonances converges to the uniform measure on the circle of radius $\frac{1}{\sqrt{q}}$. Such sequences clearly Benjamini-Schramm converge to the regular tree, and Conjecture \ref{JN_conj_weak} is evidently seen to hold. 

These explicit examples lead us to pose the following question.

\begin{question}
    Suppose $\mathcal{G}_n$ is a sequence of $(q+1)$-regular convex cocompact graphs. Let $\nu_n$ denote the distribution of resonances of $\mathcal{G}_n$, normalized to be a probability measure. 
    \begin{enumerate}
        \item What conditions on $\mathcal{G}_n$ guarantee that $\nu_n$ weakly converges?
        \item In particular, letting $f_n$ denote the number of funnels of $\mathcal{G}_n$, does Benjamini-Schramm convergence to the $(q+1)$-regular tree plus the convergence of $\frac{f_n}{|\mathcal{L}_n|}$ imply that $\nu_n$ weakly converges?
        \item What measures can arise as weak limits of $\nu_n$?
    \end{enumerate}
\end{question}

Figure \ref{fig_numerics_funnels} below shows numerical evidence for Conjectures \ref{JN_conj_weak} and \ref{JN_conj_strong}. We generate random 7-regular graphs with funnels by the same procedure described in the previous section (where, of course, we add funnels in place of cusps). The black circle in the resonance plots shows the spectral gap predicted by the Jakobson-Naud conjecture. In the histogram plots, the solid red line represents the biggest resonance, and the dotted red line represents the predicted spectral gap. We indeed see that as one adds any number of funnels, most resonances lie within the predicted circle, with many of them lying very close to this circle itself. The numerical support for the veracity of the Jakobson-Naud conjecture in this setting is quite striking, in particular if one compares with the numerical experiments for convex cocompact hyperbolic surfaces where in the numerically accessible regime only very vague signatures of the Jakobson-Naud conjecture could be observed (see, \cite[\S5.3]{BW16} and \cite[Appendix]{DBW19}).
    
Our formulation of Conjecture \ref{JN_conj_strong} is in part inspired by the work of Bordenave, Lelarge, and Massoulié \cite{bordenave_et_al}. Therein, the authors study the distribution of eigenvalues of the non-backtracking random walk on Erdős-Réyni random graphs, and more generally the Stochastic Block Model. Their plots for the distribution of eigenvalues for such models (see Figure 1 of \cite{bordenave_et_al}) bear a striking resemblance to the distribution of resonances for random funneled graphs, particularly when the number of funnels is proportional to the number of vertices (such as the case of 1000 funnels in Figure \ref{fig_numerics_funnels}). They prove, among other things, that if the largest eigenvalue has magnitude $c$, then for every $\varepsilon > 0$ we have that 
    \begin{gather*}
        \#\{ \textnormal{eigenvalues } \lambda : |\lambda| \geq \sqrt{c} + \varepsilon\} = O(1)
    \end{gather*}
    with high probability as the number of vertices goes to $\infty$. This is analogous to Conjecture \ref{JN_conj_strong}.  

\begin{figure}[p] 
    \centering
    
    \begin{subfigure}[b]{0.48\textwidth}
        \centering
        \includegraphics[width=\textwidth, keepaspectratio=true]{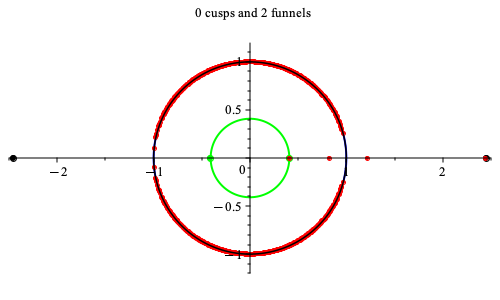}
    \end{subfigure}
    \hfill
    \begin{subfigure}[b]{0.48\textwidth}
        \centering
        \includegraphics[width=\textwidth, keepaspectratio=true]{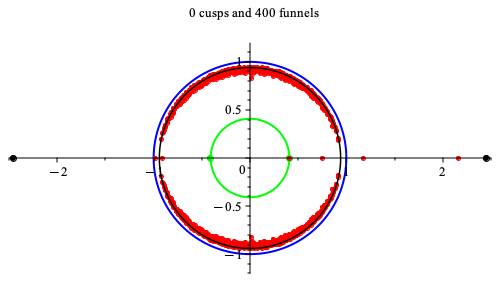}
        
    \end{subfigure}
    \vfill 

    \begin{subfigure}[b]{0.48\textwidth}
        \centering
        \includegraphics[width=\textwidth, height=0.17\textheight]{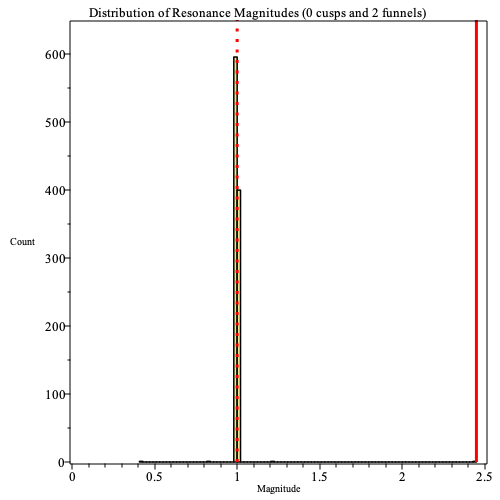}
        
    \end{subfigure}
    \hfill
    \begin{subfigure}[b]{0.48\textwidth}
        \centering
        \includegraphics[width=\textwidth, height=0.17\textheight]{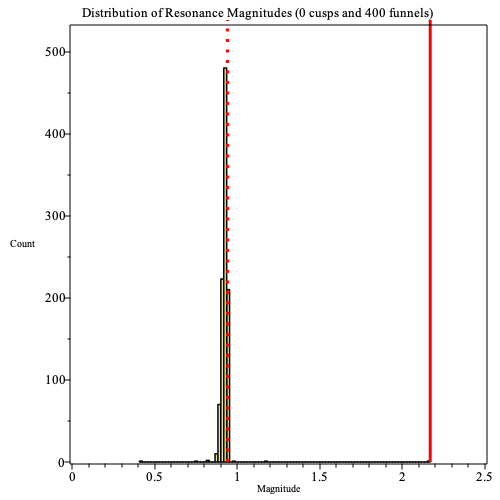}
        
    \end{subfigure}
    \vfill

    \begin{subfigure}[b]{0.48\textwidth}
        \centering
        \includegraphics[width=\textwidth, keepaspectratio=true]{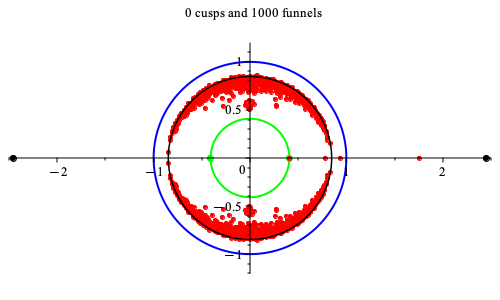}
        
    \end{subfigure}
    \hfill
    \begin{subfigure}[b]{0.48\textwidth}
        \centering
        \includegraphics[width=\textwidth, keepaspectratio=true]{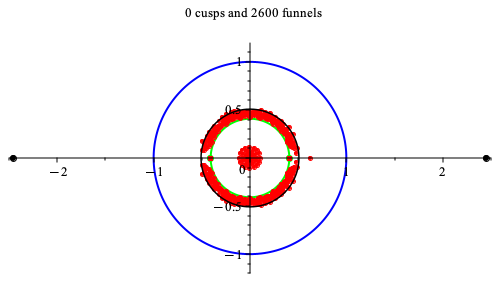}
        
    \end{subfigure}
    \vfill

    \begin{subfigure}[b]{0.48\textwidth}
        \centering
        \includegraphics[width=\textwidth, height=0.17\textheight]{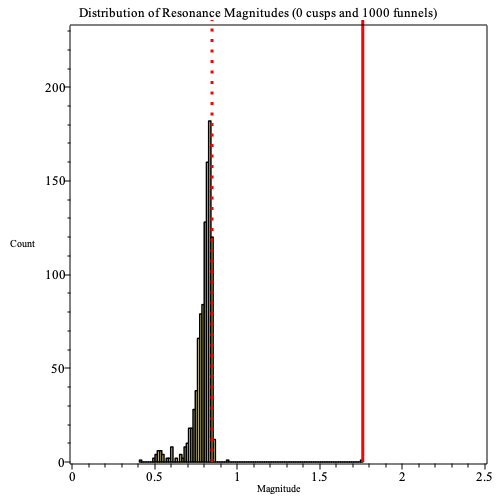}
        
    \end{subfigure}
    \hfill
    \begin{subfigure}[b]{0.48\textwidth}
        \centering
        \includegraphics[width=\textwidth, height=0.17\textheight]{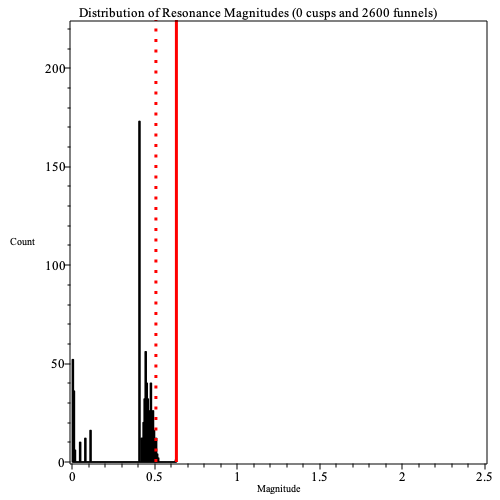}
        
    \end{subfigure}
    \vfill

    \begin{subfigure}[b]{0.48\textwidth}
        \centering
        \includegraphics[width=\textwidth]{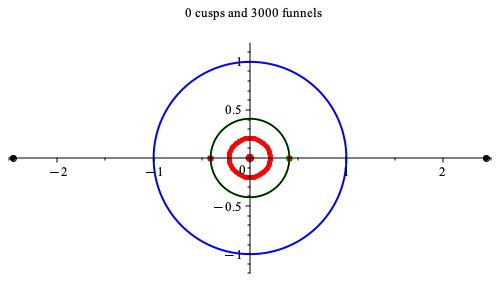}
        
    \end{subfigure}
    \hfill
    \begin{subfigure}[b]{0.48\textwidth}
        \centering
        \includegraphics[width=\textwidth]{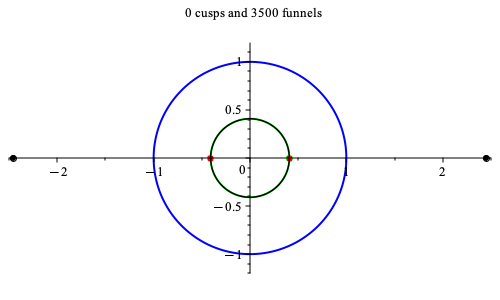}
        
    \end{subfigure}
    \vfill

    \begin{subfigure}[b]{0.48\textwidth}
        \centering
        \includegraphics[width=\textwidth, height=0.17\textheight]{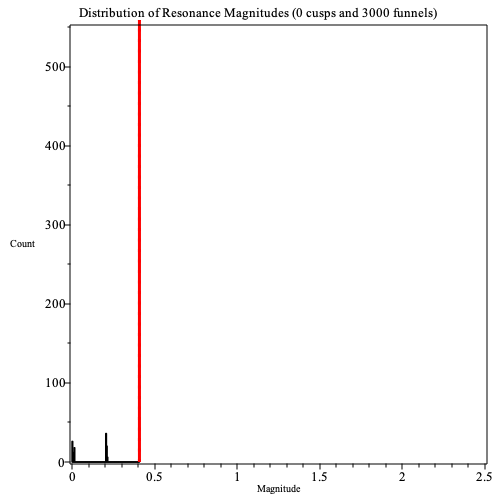}
        
    \end{subfigure}
    \hfill
    \begin{subfigure}[b]{0.48\textwidth}
        \centering
        \includegraphics[width=\textwidth, height=0.17\textheight]{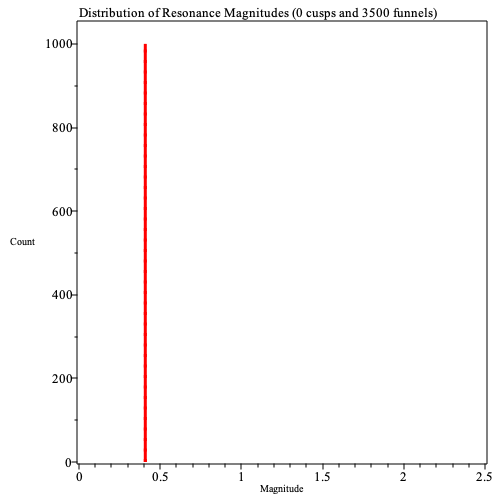}
        
    \end{subfigure}
    \vfill

    \caption{Distribution of resonances on random 7-regular graph with 500 vertices and growing number of funnels.}
    \label{fig_numerics_funnels}
\end{figure}

\section*{Acknowledgements}
We thank Bartosz Trojan for helpful conversations at an early stage of this project as well as Joachim Hilgert for many encouraging discussions about resonances on graphs. We are also grateful to Frédéric Paulin for bringing his paper \cite{paulin} to our attention, Frédéric Naud for discussions on the Jakobson-Naud conjecture, and Lior Alon for guidance on the metric graphs literature. We thank Michael Magee and Joe Thomas with whom we had useful conversations particularly related to bounding the number of resonances. Finally, we thank Farrell Brumley and Radu Toma, who helped us understand the connection between resonances, Eisenstein series, and zeta functions. 

The authors used Gemini and ChatGPT at various stages of development of this paper for informal discussions, brainstorming, searching the literature, image and code generation, and proofreading. In particular, conversations with Gemini were helpful in discovering the identity in \eqref{eqn_power_resolvent}, formulating Proposition \ref{prop_mu_pm1}, in the analysis in Section \ref{sec_hyperbolic_cylinder}, and in finding the example in Section \ref{sec_counterexample}. Gemini was used to generate Figures \ref{fig_geometrically_finite}, \ref{fig_cylinders}, and \ref{fig_exceptional} and to help write the Maple code used to generate Figures \ref{fig_numerics_cusps} and \ref{fig_numerics_funnels}. ChatGPT was used to proofread the paper. However, none of the text of the paper was AI generated, and all mathematical output of AI was only used as suggestions and was independently assessed and verified by the authors. The authors take full responsibility for the mathematical correctness of the paper.

    \printbibliography 
\end{document}